%% file: main.tex
\pdfoutput=1
\documentclass[hidelinks,onefignum,onetabnum]{siamart251216}

\usepackage{epsfig} 
\usepackage{times} 
\usepackage{amsmath} 
\usepackage{amssymb}  
\usepackage{amsfonts}
\usepackage{mathptmx} 
\usepackage[cal=cm]{mathalfa}
\usepackage{mathtools}
\usepackage{tabularx}
\usepackage{epstopdf}
\usepackage{textcomp}
\usepackage{float}
\usepackage{epic,color}
\usepackage{mathrsfs}
\usepackage{dsfont}
\usepackage{algorithm}
\usepackage{algorithmic}
\usepackage{multirow} 
\usepackage[T1]{fontenc}
\usepackage{booktabs}

\usepackage{tikz}
\usetikzlibrary{arrows, positioning, calc}
\usetikzlibrary{arrows.meta}
\usetikzlibrary{fit}
\usepackage{pgfplots}
\usepgfplotslibrary{groupplots}

\headers{Duality Theory for Filter Stability}{P. G. Mehta}

\title{Duality Theory for Filter Stability in Hidden Markov
  Models\thanks{This work was funded by the AFOSR award
		FA9550-23-1-0060 and the NSF award 2336137.}}

\author{Prashant G. Mehta\thanks{Coordinated Science Laboratory, University of Illinois
    Urbana-Champaign, Urbana, IL 
  (\email{mehtapg@illinois.edu}, \url{https://mehta.mechse.illinois.edu/}).}}

\usepackage{amsopn}

\input{_definitions}
\input{jin_thesis_macros}

\ifpdf
\hypersetup{
  pdftitle={Duality Theory for Filter Stability},
  pdfauthor={P. G. Mehta}
}
\fi

\begin{document}

\maketitle

\begin{abstract}
This article is concerned with the problem of nonlinear (stochastic) filter
stability for a discrete time hidden Markov model (HMM) with finite state and
observation spaces.  For this purpose, a comprehensive duality theory is
developed whereby the dual control system is a backward stochastic difference
equation (BS$\Delta$E).  The duality is expressed in two inter-related forms.
In its \emph{variational} form, the minimum mean squared error of optimal
filtering equals the value of a dual optimal control problem.  In its
\emph{structural} form, controllability (resp.\ stabilizability) of the
BS$\Delta$E and observability (resp.\ detectability) of the HMM are dual
properties.

The dual optimal control problem supplies the Lyapunov function and the
dissipation: the former is the value function (a variance), the latter the
optimal running cost accumulated over the horizon.  The main result is an
invariance principle: under any one of three model properties---ergodicity of
the Markov matrix, or observability or detectability of the HMM---zero
dissipation over a finite horizon forces the Lyapunov function to vanish.  A
converse is also obtained: detectability is necessary.

Two applications are given.  The first is filter stability in
$\chi^2$-divergence: the invariance principle is applied in the manner of
LaSalle, to limits of trajectories along which the dissipation vanishes
asymptotically, the role of the model properties being to render the
zero-dissipation set trivial.  The second is a quantitative form of the
invariance principle, through a Poincar\'e constant, from which convergence
rates follow.  Both the model properties that govern stability and the
Lyapunov function that certifies it are thus obtained from a dual optimal
control problem, as in the classical theory for the linear Gaussian model.
\end{abstract}



\section{Introduction}\label{sec:intro}

It has long been an aspiration to view estimation as a problem of
optimal control, going back to Kalman's work~\cite{kalman1960general}. In the linear Gaussian
setting, this connection is realized: both the model properties that
govern stability of the Kalman filter and its Lyapunov function can be
derived from a dual linear quadratic optimal control problem.
Beyond
the linear Gaussian model, the natural test of the duality is the problem
of \emph{filter stability}: whether the nonlinear filter forgets an
incorrect prior.  Whether Kalman's duality extends to this setting,
and if so in what form, is far from obvious.

The aspiration has been taken up repeatedly---most notably in the
minimum-energy dualities of Bryson--Frazier, Mortensen, and Hijab, and in
their stochastic formulation by Fleming and Mitter and by Mitter and
Newton~\cite{bryson1963smoothing,mortensen1968,hijab1980minimum,fleming1982optimal,mitter2003}.
These
dualities, however, are not Kalman's: their constraint is a copy of
the signal model rather than the dual system, and detectability does
not arise as the dual of stabilizability~\cite[Sec.~7.3.1 and
7.3.2]{bensoussan2018estimation}. Todorov concluded that Kalman's
duality does not extend beyond the linear Gaussian
setting~\cite{todorov2008general}. Van Handel similarly observed, in
his survey of nonlinear filter stability, that proofs of nonlinear
filter stability must ``start from scratch''~\cite{van2010nonlinear}. 

Yet the possibility of a deeper connection was recognized.  In 
his PhD thesis, van Handel listed the connections between filter stability and
stochastic control as meriting further investigation, observing that
the relationship, in the linear case, between filter stability and
controllability and observability~\cite{ocone1996asymptotic} ``makes it
tempting to look for such connections in the nonlinear
case''~\cite[p.~17]{van2006filtering}.  The suggestion remained a direction for
future work.  The
subsequent development of observability and detectability for nonlinear
filters, and the stability proofs based on them, proceeded instead by
probabilistic
methods~\cite{van2009observability,van2009uniform,chigansky2009intrinsic,van2010nonlinear},
now known as the \emph{intrinsic approach} to filter stability.

At a minimum, a duality theory should do for the nonlinear filter what it does for the Kalman filter:
\begin{itemize}
\item[(a)] identify the model properties that govern stability, and
\item[(b)] provide a Lyapunov function that certifies it.
\end{itemize}
For the hidden Markov model (HMM) in its basic form---discrete time
with finite state and observation spaces---this paper does exactly
that: both the model properties and the Lyapunov function are derived
from a dual optimal control problem. 

The focus on the finite-state discrete-time HMM is motivated by two
considerations.  First, it is the simplest setting in which the essential
difficulties of nonlinear filtering already appear: even for an ergodic
finite-state chain, the filter need not forget an incorrect prior
(Example~\ref{ex:counter-example}, due to Kaijser~\cite{kaijser1975limit}).  It is therefore the
natural setting in which the duality theory should first be understood, free of the
technical overhead of the general case.  Second, it is the standard model for
reinforcement learning in partially observed Markov decision processes (POMDPs).
There, filter stability is a key ingredient in learning near-optimal policies
from a finite window of past observations~\cite{kara2023convergence,kara2022near},
and in bounding the loss incurred by controlling with an incorrect
prior~\cite{mcdonald2022robustness}.

\subsection{Approach and contributions of this paper}
While the intrinsic approach is based on the forward map
$\gamma_0\mapsto\gamma_T$ between the Radon--Nikodym derivatives at
time $t=0$ and $t=T$ (\Sec{sec:gamma_T_defn}), the dual approach of this paper is based on a
\emph{backward map} $\gamma_T\mapsto y_0$ (Eq.~\eqref{eq:bmap}).
Filter stability in the sense of $\chi^2$-divergence follows from a
variance decay property of the backward map
(\Prop{prop:chisq_filter_stability}).  To show this property, the
backward map is embedded in a backward process
(Eq.~\eqref{eq:Y_and_M}), which is the solution of the dual optimal control system.

The contributions are as follows.

\begin{enumerate}
\item The main result is an \emph{invariance principle},
  \Lemma{lemma:Lasalle}, for the backward process.  The principle relates the model properties that
  govern stability and the Lyapunov function that certifies it. 
  \begin{itemize}
  \item[(a)] The three model properties are ergodicity of the Markov
    matrix, and observability and detectability of the HMM, the first
    and third under a non-degeneracy condition on the emission
    matrix (Assumption~\ref{ass:non-degeneracy}).  
  \item[(b)] The Lyapunov function is the variance of a backward
      process (\Prop{Prop:Lyapunov}).  It is non-increasing in the backward direction of
      time, and the dissipation---the amount by which it decreases---is identified
      explicitly, as the running cost of the dual optimal control problem.  Filter
      stability follows if the Lyapunov function at the initial time decays to zero as
      $T\to\infty$.
  \end{itemize}
  The invariance principle is the statement that, under any one of the three model
  properties, zero dissipation over a finite horizon forces zero
  variance.  Conversely, if this implication holds for every horizon,
  prior, and terminal condition, then the HMM is detectable
  (\Prop{thm:necessity}).  Therefore, detectability is necessary for
  the invariance principle to hold in this generality. 
  
\item \emph{Duality theory for the HMM}.  The dual control system is a backward
  stochastic difference equation (BS$\Delta$E).  As in the classical setting, duality
  is expressed in two inter-related manners:
  \begin{itemize}
   \item[(i)] \emph{Variational (minimum-variance) duality}:
     \Prop{prop:MU-martingale} and \Corollary{cor:U-opt-zero} show that the minimum mean squared error
    (MMSE) of optimal filtering equals the value of the dual optimal control
    problem. Moreover, it identifies the Lyapunov
    function with that value function. 
  \item[(ii)] \emph{Structural duality}: \Prop{prop:duality} shows
    that controllability
    (resp.\ stabilizability) of the BS$\Delta$E and observability (resp.\ detectability)
    of HMM are dual properties.
  \end{itemize}
The development parallels linear systems theory throughout: the
controllable subspace, its recursion, the stabilization index, and the
annihilator pairing with the nonobservable subspace all appear in
their familiar form.

\item \emph{Filter stability} (\Thm{thm:filter-stability}).  Assuming
  equivalent priors and any one of the three model properties, the
  filter is stable in $\chi^2$-divergence. The proof is a stability
  argument in the classical sense: 
  The invariance principle (\Lemma{lemma:Lasalle}) is applied in the manner
  of LaSalle, to limits of trajectories along which dissipation vanishes
  asymptotically.  
  The role of model properties is to render the zero-dissipation set
  trivial.

\item \emph{Rates}.  The conclusion of \Lemma{lemma:Lasalle} is
  made quantitative through a Poincar\'e constant $\kappa(\rho)$  (Defn.~\ref{def:poincare}).  This yields a decay
  bound along the filter's own trajectory (\Prop{prop:var_contractive}) and a dichotomy
  (\Thm{thm:dichotomy}): either the filter is stable in $\chi^2$-divergence, or its
  skeleton converges to the set on which $\kappa$ degenerates.
\end{enumerate}

\subsection{Literature} 
For the linear Gaussian model, filter stability is classical, and is understood through
Kalman's duality between filtering and
control~\cite{kalman1961},~\cite[Sec.~4.4]{kwakernaak1969linear},~\cite[Sec.~2]{ocone1996asymptotic}.
The model property for stability of the Kalman filter is detectability of the linear
Gaussian model, which by an elementary duality is equivalent to stabilizability of the
dual control system.  A standard proof technique for the asymptotic convergence of the
Kalman filter error covariance is based on the dual linear quadratic (LQ) problem: a
stabilizing control bounds the cost uniformly in time, and the value function serves as a
Lyapunov function~\cite[Lemma~2.2]{ocone1996asymptotic},~\cite[Ch.~9]{xiong2008introduction},%
~\cite[Proof of Thm.~23.2]{brockett2015finite}.

These optimal control foundations extend far beyond the LQ case, and without them it is
difficult to imagine the modern stability theory for deterministic nonlinear systems.
Input-to-state stability (ISS) has its roots in dissipativity theory, with general
(non-positive-definite) running costs~\cite{sontag2008input}, and the dual notion of
detectability~\cite{sontag1997output} is central to the stability of model predictive
control~\cite{Teel_MPC_2005,Grune_2022}.  A primary motivation for introducing incremental
input/output-to-state stability (i-IOSS) is its use in deducing the asymptotic stability of
the minimum-energy estimator (MEE)~\cite{krener2003convergence},~\cite[Defn.~4.1 and
Prop.~4.9]{rawlings2017model}.  No comparable theory exists for the nonlinear filter---one
in which both the model properties and the Lyapunov function come from a dual control
problem.

The literature on nonlinear filter stability distinguishes two
cases~\cite{chigansky2006stability,chigansky2009intrinsic}:
\begin{itemize}
\item[(i)] The \emph{ergodic signal case}, where the hidden Markov process forgets its
  initial condition and the filter inherits the property.
\item[(ii)] The \emph{nonergodic signal case}, where the observations provide enough
  information about the hidden state for the filter to correct an erroneous prior.
\end{itemize}
For the ergodic signal case, early work is based on contraction analysis of the random
matrix products that arise from recursive application of Bayes'
formula~\cite{atar1997lyapunov,legland2000projective,legland2000exponential},
\cite[Sec.~4.3]{Moulines2006inference},~\cite{delmoral2001stability,douc2009forgetting},
and, for the Zakai equation, on formulae for the Lyapunov
exponents~\cite{atar1997exponential,Atar1999,baxendale2004asymptotic}.
The analysis
requires a non-degeneracy assumption
(Assumption~\ref{ass:non-degeneracy}), satisfied by the white noise observations; without it, the filter need not be
stable even when the signal is ergodic~\cite{kaijser1975limit,delyon1988lyapunov} (see
Example~\ref{ex:counter-example}).  For the nonergodic signal case, notable early
contributions are~\cite{clark1999relative}, where the relative entropy is shown to be a
Lyapunov function for the filter, and~\cite{chigansky2006role}, where the filter is shown to
be stable for functions that predict the observations.  For a class of
diffusions, filter stability results of this type are obtained by PDE methods
in~\cite{stannat2005stability,stannat2006}.  Definitions of observability and
detectability, and their relationship to filter stability, first appear
in~\cite{van2009observability,van2009uniform}.  In both cases the proofs directly exploit
the representation of the filter as a conditional expectation, through classical
probabilistic techniques---the intrinsic
approach~\cite{chigansky2009intrinsic}.

McDonald and Y\"uksel~\cite{mcdonald2018stability,mcdonald2023stochastic}
introduce a different, explicitly stated notion of observability for
nonlinear systems, and show that it yields stability of the predictor, with
filter stability following under additional conditions.  They also
establish a set of implications among the stability criteria
themselves---weak merging, total variation merging, and merging in
relative entropy---which complements the relative entropy bounds
of~\cite{clark1999relative}.

This paper builds on our prior
work~\cite{kim2019duality,duality_jrnl_paper_I,duality_jrnl_paper_II}, where duality
for nonlinear filtering was first introduced for the continuous-time
HMM with white noise observations.  In this prior work, a backward
stochastic differential equation (BSDE) was introduced as the dual
control system, and observability of the HMM and controllability of
the BSDE were shown to be dual
properties~\cite{duality_jrnl_paper_I,JinPhDthesis}. The use of this
duality to prove filter stability was begun
in~\cite{kim2021ergodic,kim2021detectable,kim2024variance,kim2024backward}, 
based on a backward map and 
a variance decay property, both of which are central here as well, but
stability results were possible only in special
cases~\cite[Thm.~3]{kim2024variance}.

Relative to our prior work, the contributions of the present paper are the
following.  Our previous results were limited to white noise observations,
while the present work concerns the HMM in settings of most practical
relevance.  In some ways, duality theory for the finite setting is easier,
being based entirely on linear algebra, which also makes the connection to
Kalman's duality transparent.  In other ways it is more delicate,
requiring careful track of the time horizons: those for the three
model properties, and the horizon at which the Poincar\'e constant is
evaluated.  These have no counterpart in continuous time. Apart
from these differences, the following are novel contributions
of this paper:
\begin{enumerate}
\item The main result---the invariance principle
  (\Lemma{lemma:Lasalle})---is new.  So is its converse (\Prop{thm:necessity}).
  \item The proof of filter stability is new. To the best of our
  knowledge, the proof given here 
  is the first variational, or control-theoretic, proof of nonlinear filter
  stability under conditions on the model alone (\Thm{thm:filter-stability}).
\item The regularity theory for the Poincar\'e constant
  is new.  The analysis of the constant
  as a function on the probability simplex is new (\Prop{prop:N-strata} and \Lemma{prop:beta-regularity}).  
 \end{enumerate}

It is fair to ask what is gained, given that the intrinsic approach has
settled filter stability under necessary and sufficient conditions.  Our
answer:
\begin{enumerate}
\item The finite setting is treated comprehensively.  The three model
  properties reduce to conditions of linear algebra, as transparent as in the
  linear case, and the indices that mark the horizons at which they take effect
  are identified explicitly.  One of these is related to the classical index of
  primitivity~\cite{legland2000exponential}; the other two, and the role of the
  indices in filter stability, appear to be new.
\item The model properties for stability are not assumed but derived.  Asking what is
  required for the variance to decay, the theory returns stabilizability of
  the dual control system, and the converse (\Prop{thm:necessity})
  shows that stabilizability is also necessary.
\item The same theory delivers the rates, through the Poincar\'e constant,
  which for a given model is computed by solving a finite-dimensional
  eigenvalue problem. This makes the theory useful in practice and
  informs the future work.  
\end{enumerate}

We close by returning to the aspiration with which we began.  For the
finite-state HMM the answer is affirmative: Kalman's duality does extend to
the purposes of filter stability, and it extends in both respects---it
identifies the properties of the model responsible for stability, and it
supplies the Lyapunov function with which stability is proved. 

\subsection{Paper outline} \Sec{sec:hmm-prelim} contains mathematical
preliminaries.  \Sec{sec:bmap} introduces the backward map and states
the main result (\Lemma{lemma:Lasalle}).  \Sec{sec:imp_ass} presents a proof of
the main result for the ergodic case.  \Sec{sec:duality} develops the duality
theory for the HMM, which is used in \Sec{sec:observable_proof} to prove the
main result for the observable and detectable cases, and to establish its
converse.  The main result is applied in \Sec{sec:filter_stability} to prove
filter stability, and in \Sec{sec:rates} to derive rates.  The appendices
contain the technical proofs.

A reader may wish to review \Defn{defn:likelihood_ratio} of $\gamma_T$ and
\Defn{defn:Z_spaces} of $\clZfo{T}(\rho)$, and then proceed to
\Sec{sec:bmap}, returning to \Sec{sec:hmm-prelim} as needed.  For easy
reference, Table~\ref{tab:n_hmmx} collects the notation.

\section{Preliminaries}\label{sec:hmm-prelim}

\subsection{Hidden Markov model (HMM)}
Throughout this paper, we consider discrete-time stochastic processes
on a finite time-horizon $\mathbb{T}=\{0,1,2,\hdots,T\}$ with
$T<\infty$. Fix a measurable space $(\Omega,\clF_T)$ along with a
filtration $\{\clF_t:t\in\mathbb{T}\}$ with respect to which all the
stochastic processes below are adapted. The basic objects of the model
are as follows (see~\cite{elliott2008hidden,Moulines2006inference}):
\begin{enumerate}
\item The {state-space} $\bS=\{1,2,\hdots,d\}$ is finite.  The
  state-process $X=\{X_t:\Omega\to\bS, \; t\in\mathbb{T}\}$ is $\bS$-valued.  
\item The {observation-space} $\bO=\{0,1,2,\hdots,m\}$ is finite of
  cardinality $|\bO|=m+1$.  The observation process
  $Z=\{Z_t:\Omega\to\bO,\; t=1,2,\hdots,T\}$ is $\bO$-valued.  
\item The state transition matrix $A$ is a row-stochastic $d\times d$
  matrix: for each $x\in\bS$, $A(x,\cdot)$ is a probability vector on
  $\bS$.
\item The emission matrix $H$ is a row-stochastic $d\times(m+1)$
  matrix: for each $x\in\bS$, $H(x,\cdot)$ is a probability vector on
  $\bO$.
\end{enumerate}
Let $\mathbb{T}^{-}:=\{0,1,2,\hdots,T-1\}$.  
The filtrations are as follows:
\begin{align*}
\clF_t&:=\sigma(X_0,Z_1,X_1,\hdots,Z_t,X_t),\quad t\in\mathbb{T},\\
        \clZ_t&:=\sigma(Z_1,Z_2,\hdots,Z_t), \quad t=1,2,\hdots, T,
\\    \clG_{t}&:=\clF_t\vee \clZ_{t+1},\quad t\in\mathbb{T}^-,
\end{align*}
with $\clZ_0:=\{\emptyset,\Omega\}$. $\clF=\{\clF_t:t\in\mathbb{T}\}$ is
referred to as the canonical filtration with respect to which all
processes are adapted.

For each $\rho\in\clP(\bS)$, define a probability measure $\sP^\rho$
on $(\Omega,\clF_T)$ over the space of trajectories of $(X,Z)$ as
follows (see~\cite[Defn.~2.2.1]{Moulines2006inference},~\cite[Sec.~2.2]{clark1999relative}):
\begin{align*}
\sP^\rho (X_0=\xi_0,\hdots,X_T = \xi_T,&\;
Z_1=\zeta_1,\hdots,Z_T = \zeta_T)  := \rho(\xi_0) \Pi_{t=1}^T
  A(\xi_{t-1},\xi_t) H(\xi_{t-1}, \zeta_t),\\
 & \forall \; (\xi_0,\xi_1,\hdots,\xi_T) \in
   \bS^{T+1},\quad (\zeta_1,\hdots,\zeta_T) \in \bO^T.
\end{align*}
The associated expectation operator is denoted by
$\E^\rho(\cdot)$. This construction gives 
\begin{align*}
\sP^\rho(X_0=x) &= \rho(x), \quad x\in \bS  \\
\sP^\rho(X_{t+1}=x' \mid X_{t}=x) &= A(x,x'),  \quad x,x'\in\bS,\quad
  t\in\mathbb{T}^-. \\
\sP^\rho(Z_{t+1} = z \mid X_{t}=x) &= H(x,z),  \quad x\in\bS,\;\;z\in\bO,\quad t\in\mathbb{T}^-.
\end{align*}
recovering the usual conditional specification of an HMM as a
consequence of the construction rather than as a standing
assumption. We denote the model as HMM $(A,H)$.  With $m=1$, there are
only two observations, $\bO=\{0,1\}$. Such a model is referred to as
binary-valued HMM.  The graphical model of the HMM is depicted in
\Fig{fig:hmm}.

\begin{figure}
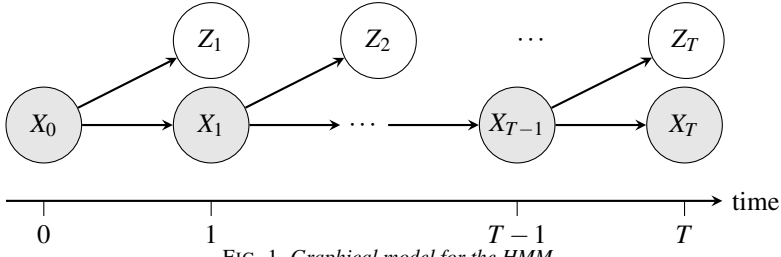

  \begin{center}
    \include{figure/hmm.tex}
  \end{center}
  \vspace{-4em}
  \caption{Graphical model for the HMM.
  }
  \label{fig:hmm}
  \vspace*{-10pt}
\end{figure}

Two choices of $\rho$ are of particular interest throughout the
paper:
\begin{itemize}
\item $\rho=\mu$.
 The measure $\mu$ has the meaning of the true prior that is used to
 generate the observations.   
\item  $\rho=\nu$.  
The measure $\nu$ has
  the meaning of the incorrect prior that is used to initialize the
  nonlinear filter. 
\end{itemize}
Suppose $\mu\ll \nu$.  Then the explicit construction shows that $\sP^\mu\ll\sP^\nu$ on $(\Omega,\clF_T)$, with
\[
\frac{\ud\sP^\mu}{\ud\sP^\nu} = \frac{\mu(X_0)}{\nu(X_0)}, \quad \sP^\nu\text{-a.s.}
\] 
The notation $\E^x$ is used to denote the case where the prior
$\rho=\delta_x$ (i.e., $X_0=x$ w.p. $1$).

\newP{Notation}
The space of functions and measures on $\bS$ are both identified with
$\Re^d$: a real-valued function $f$ (resp., a measure $\rho$) is
identified with a column vector in $\Re^d$ where the $x^{\text{th}}$
element of the vector represents $f(x)$ (resp., $\rho(x)$), for
$x\in\bS$, and $\rho(f) := \rho^\tp f$. The probability simplex in
$\Re^d$ is denoted by $\clP(\bS)$.  
Lower case symbol, e.g., $f$, is used to denote a
deterministic function while an upper case symbol, e.g.,
$F$, is used to denote a random function.  For a vector $f\in\Re^d$, $\text{diag}(f)$ is $d\times
d$ diagonal matrix with entries given by the components of $f$. A product of two functions $f$ and $g$ is denoted by
$fg$ ($(fg)(x)=f(x)g(x)$ for $x\in\bS$), and $f^2=ff = \text{diag}(f)f$.
The unity function is denoted by $\ones$ (e.g., as a function on $\bS$, $\ones (x)=1$ for all
$x\in\bS$). For $\rho,\tilde\rho\in\clP(\bS)$,
the total variation distance $\|\rho-\tilde\rho\|_{\text{TV}} := \half
\sum_{x\in\bS} |\rho(x) - \tilde{\rho}(x)|$.

\subsection{Nonlinear filter}\label{sec:nonlinear-filter}
Fix $\rho\in\clP(\bS)$.  In the theory of nonlinear filtering, there
are closely related objects of interest:
\begin{align*}
\text{(un-normalized filter)} \quad   \sigma^\rho_t(x) &:= \sP^\rho\big(Z_1, Z_2,\hdots,Z_t, X_t = x\big),
  \qquad x\in\bS, \quad t\in\mathbb{T}. \\
\text{(Normalized filter)} \quad  \pi_t^\rho(x)
                                                       &:=\sP^\rho\big([X_t=x] \mid
                                                  \clZ_t\big),\quad x \in \bS,\quad t\in\mathbb{T}. 
\end{align*}
Both $\sigma^\rho_t$ and $\pi^\rho_t$ are $\clZ_t$-measurable measures on
$\bS$, the latter a probability measure.  The un-normalized measure is
well-defined for all paths 
$(Z_1=\zeta_1,\ldots,Z_t=\zeta_t)$, for
$(\zeta_1,\ldots,\zeta_t)\in\bO^t$, and satisfies the
recursion
\[
\sigma^\rho_{t+1}(f) = \sigma_{t}^\rho(\text{diag}(H(:,Z_{t+1})) Af),\quad
f\in \Re^d, \quad t\in \mathbb{T}^-,\quad \sigma_0^\rho = \rho,
\]
where note $\sigma_t^\rho(f) = \sum_{x\in \bS} \sigma_t^\rho(x)
f(x)$.  The normalized filter is called as such because
\[
\pi_{t}^\rho(f) =
\frac{\sigma_{t}^\rho(f)}{\sigma_{t}^\rho(\ones)},\quad
\sP^\rho\text{-a.s.},\quad 
f\in\Re^d,\quad t\in \mathbb{T}.  
\]
The $\sP^\rho\text{-a.s.}$ qualifier is necessary because of the
possibility that the denominator is zero.  Noting $\sigma_t^\rho(1) =
\sP^\rho(Z_1,Z_2,\ldots,Z_t)$, these are the paths which are
emitted with zero probability.

A recursive formula for
$\pi_t^\rho\mapsto\pi_{t+1}^\rho$ is also obtained based on the
recursive formula for the unnormalized counterpart.  The formula is
referred to as the nonlinear filter, also called the forward algorithm.   

The filter stability problem is concerned with asymptotic
properties of the normalized filter with respect to different initializations, $\mu,\nu\in\clP(\bS)$. We use the superscripts $\pi_t^\mu$ and $\pi_t^\nu$ to
refer to the particular cases $\rho=\mu$ and $\rho=\nu$; note these
are conditional measures with respect to $\sP^\mu$ and $\sP^\nu$,
respectively.

\subsection{Likelihood ratio}\label{sec:gamma_T_defn}
We define the Radon--Nikodym (RN) derivative $\gamma_T$ of $\pi^\mu_T$ with
respect to $\pi^\nu_T$.  We begin with an assumption.

\begin{assumption}[Equivalent priors]\label{ass:equivalent-priors}
  Let $\bS'=\supp(\nu)$.  
The measures $\mu\sim\nu$ (are equivalent) with 
\[
0< \gamma_0^{\min{}} := \min_{x\in\bS'} \gamma_0(x) \leq  \max_{x\in\bS'} \gamma_0(x) =: \gamma_0^{\max{}} < \infty.
\]
Let $\alpha := \gamma_0^{\min{}} /\gamma_0^{\max{}}$\footnote{Because $S'$ is
finite, $\mu\sim\nu$ already implies the two bounds. The assumption serves to
name $\gamma^{\min}_0$, $\gamma^{\max}_0$ and $\alpha$, which appear in the
estimates of \Sec{sec:filter_stability} and \Sec{sec:rates}.}.
\end{assumption}

Since $\mu=\gamma_0\nu$ and $\rho\mapsto\sigma_T^\rho$ is linear,
\[
\gamma_0^{\min{}}\,\sigma_T^\nu(x) \;\leq\; \sigma_T^\mu(x) \;\leq\;
\gamma_0^{\max{}}\,\sigma_T^\nu(x), \qquad \forall\,x\in\bS.
\]
Summing over $x\in\bS$, we have 
$\gamma_0^{\min{}}\,\sigma_T^\nu(\ones) \;\leq\; \sigma_T^\mu(\ones) \;\leq\;
\gamma_0^{\max{}}\,\sigma_T^\nu(\ones)$. 
We have two cases to consider:
\begin{enumerate}
\item If $\sigma_T^\nu(x)>0$. Then
  $\sigma_T^\mu(\ones)\geq\gamma_0^{\min{}}\sigma_T^\nu(\ones)>0$,
  both filters are well defined.  The RN derivative
\[
  \gamma_T(x) \;:=\;
  \frac{\pi_T^\mu(x)}{\pi_T^\nu(x)}
\;=\; \frac{\sigma_T^\mu(x)}{\sigma_T^\nu(x)}\cdot
      \frac{\sigma_T^\nu(\ones)}{\sigma_T^\mu(\ones)}
\;\in\; \Big[\frac{\gamma_0^{\min{}}}{\gamma_0^{\max{}}},\,
             \frac{\gamma_0^{\max{}}}{\gamma_0^{\min{}}}\Big]
             \;=\; [\alpha,\alpha^{-1}].
\]
\item If $\sigma_T^\nu(x)=0$, then
  $\sigma_T^\mu(x)\leq\gamma_0^{\max{}}\sigma_T^\nu(x) =0 $. In this
  case, 
 \begin{equation}
\label{eq:convention}\tag{convention} \gamma_T(x):=1.
\end{equation}
 \end{enumerate}

 Based on this convention, we adopt the following definition:
 \begin{definition}\label{defn:likelihood_ratio}Let $\mu,\nu\in\clP(\bS)$ with $\mu\sim\nu$.  Then
\[
  \gamma_T(x) := \begin{cases} \frac{\pi_T^\mu(x)}{\pi_T^\nu(x)}, &
    \sigma_T^\nu(x)>0 \\[5pt]
    1, & \sigma_T^\nu(x)=0    
\end{cases},  \qquad
x\in\bS.
\]
\end{definition}
While $\gamma_0$ is a deterministic function on $\bS$, $\gamma_T$ is a
$\clZ_T$-measurable random function on $\bS$.
By~\eqref{eq:convention}, $\gamma_0(x)=1$, for all $x\notin\supp(\nu)$.   

\subsection{Filter stability}\label{sec:filter_stability_defn}
Filter
stability is the assertion that $\gamma_T\to\ones$, in a suitable
sense, as $T\to\infty$.  The $\chi^2$-divergence between $\pi_T^\mu$
and $\pi_T^\nu$ is defined as
\[
\chisq(\pi^\mu_T\mid\pi^\nu_T) :=
\pi^\nu_T\big((\gamma_T-1)^2\big),\quad \sP^\nu\text{-a.s.}
\]
Note the right hand side $\pi^\nu_T\big((\gamma_T-1)^2\big) =
\sum_{x\in \bS} \pi^\nu_T(x) (\gamma_T(x)-1)^2$. So, the choice of
$\gamma_T(x)=1$ is a matter of convention.  Another choice yields
identical $\chi^2$-divergence.  With this particular choice,
$\gamma_T(x) \in [\alpha,\alpha^{-1}]$, for all $x\in \bS$.  This
holds surely---for all sample paths, even those with zero probability.  
\begin{definition}\label{def:filter-stability}
Let $\mu,\nu\in\clP(\bS)$ with $\mu\sim\nu$.  The nonlinear filter is \emph{stable} in the sense of $\chi^2$ if
\[
  \E^\mu\big(\chisq(\pi_T^\mu\mid\pi_T^\nu)\big) \stackrel{(T\to
  \infty)}{\longrightarrow} 0.
\]
(The expectation is
with respect to $\sP^\mu$ because $\mu$ is the correct prior.)
\end{definition}

Stability in the sense of $\chi^2$ is the strongest of the standard
notions: by well-known inequalities between $f$-divergences, it
implies the corresponding statements with $\chisq$ replaced by the
Kullback--Leibler divergence or the total variation distance, and
hence also stability in the $L^2$ sense \cite[Sec.~7.1]{JinPhDthesis}.  The latter two are
the definitions of historical interest~\cite{ocone1996asymptotic}.

      \subsection{Model properties} \label{sec:model_props}
      These are of two types:
\begin{enumerate}
\item Properties related to the Markov chain $X$ that depend only on
  matrix $A$.   
\item Properties related to the HMM that depend upon both $(A,H)$.  
\end{enumerate}
We begin with some definitions relevant to the first case.  By
Perron--Frobenius theory applied to each recurrent class, the eigenvalues of $A$ with unit modulus are
semisimple roots of unity~\cite[Cor.~8.4.10]{horn_johnson_1985}.  Let $p$ denote the least common multiple
of their orders, so that $\lambda^p = 1$ for every eigenvalue
$\lambda$ of $A$ with $|\lambda| = 1$. 

\begin{definition}[Peripheral eigenspace]\label{def:peripheral}
The \emph{peripheral eigenspace} of the Markov matrix $A$ is
\[
  \mathrm{E}^{\mathrm{PER}}(A) \;:=\; \ker\big(A^p - I\big),
\]
the span of the real and imaginary parts of the eigenvectors of $A$
whose eigenvalues have unit modulus. 
\end{definition}

Since $A\ones=\ones$, the peripheral spectrum is nonempty and
$\ones \in \mathrm{E}^{\mathrm{PER}}(A)$.

\begin{definition}
The Markov matrix $A$ is \emph{ergodic} if 
\[
\mathrm{E}^{\mathrm{PER}} (A) = \mathrm{span}\{ \ones\}.
\]
In this case, $p=1$. 
\end{definition}
Let $\bS^{\text{\tiny RECC}} \subseteq \bS$ denote the set of
recurrent states.  In general, $\bS^{\text{\tiny RECC}}$ may contain more than one
recurrent class.  
For $x \in \bS$ and $t \ge 1$, let
\[
  {\mathcal R}_t(x) \;:=\; \{\, x' \in \bS : A^t(x,x') > 0 \,\}
\]
denote the set of states reachable from $x$ in exactly $t$ steps.

\begin{subequations}\label{eq:indices}
We have two cases as follows:
\begin{enumerate}
\item If $A$ is ergodic, define 
\begin{equation}\label{eq:tau0}
\tau_0(A) := \min\Big\{t\ge 1 \;:\; A^t(x,x') > 0,\;\;
\forall\, x\in\bS,\;\; \forall\, x'\in \bS^{\text{\tiny RECC}} \Big\}.
\end{equation}
Wielandt's bound shows $\tau_0(A) \leq (d-|\bS^{\text{\tiny RECC}}|)
+ (|\bS^{\text{\tiny RECC}}|-1)^2 + 1$~\cite[p.~241]{dulmage1962exponent}. 
\item More generally, define
\begin{equation}\label{eq:tau1}
  \tau_1(A) \;:=\; \min\big\{\, t \ge 1 \;:\;
    {\mathcal R}_t(x) \cap \bS^{\text{\tiny RECC}}
    = {\mathcal R}_{t+p}(x) \cap \bS^{\text{\tiny RECC}} ,
    \ \ \forall\, x \in \bS \,\big\} .
  \end{equation}
This is finite for every $A$: within each recurrent class of period
$p_j$, the sets ${\mathcal R}_t(x) \cap \bS^{\text{\tiny RECC}}$ are
eventually periodic in $t$ with period $p_j$, and $p$ is a common
multiple of the $p_j$~\cite[Sec.~1.4]{kemeny-snell},~\cite[Sec.~5.4]{meyn2009markov}.  When $A$ is ergodic, $p=1$ and
$\tau_1(A) = \tau_0(A)$. 
\end{enumerate}

\begin{example}[Counter-example of filtering theory]
\label{ex:counter-example}
	Consider an HMM on state-space
	$\bS=\{1,2,3,4\}$ and observation-space $\bO=\{0,1\}$, with 
	\[
	A = \begin{pmatrix}
		\half & \half  & 0 & 0\\
		0 & \half  & \half  & 0\\
		0 & 0 & \half  & \half \\
		\half  & 0 & 0 & \half 
              \end{pmatrix},\qquad H = \begin{pmatrix} 1 & 0 \\ 0 & 1 \\1 & 0 \\ 0 & 1
              \end{pmatrix}.
            \]
       The matrix $A$ is ergodic
       (indeed irreducible and aperiodic), yet the filter is not
       stable: 
       \[
         \|\pi_t^\mu - \pi_t^\nu \|_\tv = 
         \begin{cases}
          | \frac{\mu(1)}{\mu(1) +
           \mu(3)} - \frac{\nu(1)}{\nu(1) +
           \nu(3)}|,\quad \text{if} \;\; Z_1=0,\\[5pt]
         | \frac{\mu(2)}{\mu(2) +
           \mu(4)} - \frac{\nu(2)}{\nu(2) +
           \nu(4)} |,\quad \text{if} \;\; Z_1=1,
          \end{cases}\qquad t\geq 1.
        \]
The total variation distance between the two filters does not depend
on $t$, and so does not vanish for generic $\mu\neq\nu$. This example, originally due to
Kaijser~\cite{kaijser1975limit}, shows that ergodicity of the
underlying Markov chain does not by itself guarantee filter
stability. The obstruction is fragile, however: perturbing the model
so that $H(x,z)>0$ for all $(x,z)\in\bS\times\bO$ is already enough
to restore filter stability~\cite[Sec.~II-C]{van2010nonlinear},~\cite{baxendale2004asymptotic}. 
       \end{example}

Observability is a model property related to both $A$ and $H$.   Its
definition is based on the marginal probability measure, 
 \[
\sP^\rho\mid_{\clZ_T}:=\sP^\rho
(Z_1=\zeta_1,Z_2=\zeta_2,\hdots,Z_T=\zeta_T),\qquad (\zeta_1,\zeta_2,\hdots,\zeta_T)\in\bO^T.
\]

\begin{definition}[Defn.~2 in~\cite{van2009observability}]\label{def:observability}
	The HMM $(A,H)$ is \emph{observable} if 
\begin{equation*}
	\quad \sP^\mu|_{\clZ_T} = 
	\sP^\nu|_{\clZ_T} \quad \forall\, T>0 \; \Longrightarrow\; \mu = \nu,\qquad \mu,\nu\in\clP(\bS).
\end{equation*}
\end{definition}
Let $\mu \sim_T \nu$ denote the equivalence relation $\sP^\mu|_{\clZ_T} =
\sP^\nu|_{\clZ_T}$ (so that Definition~\ref{def:observability}
reads: observable iff $\mu\sim_T\nu \implies \mu=\nu$ for all $T>0$). Define 
\[
\clN_T := \mathrm{span}\{\mu-\nu \in \Re^d \;:\; \mu,\nu\in\clP(\bS),\;
\mu\sim_T\nu\} \;\subset\; \Re^d.
\]
The subspace of \emph{nonobservable measures} $\clN:= \bigcap_{T>0}
\clN_T$.  The \emph{observability index} is
\begin{equation}\label{eq:tau2}
\tau_2(A,H) := \min\big\{\tau : \clN_\tau = \clN\big\}.
\end{equation}
\end{subequations}
In finite state-space settings of this paper, $\tau_2(A,H) \leq d-1$ (this
is shown in~\Prop{prop:CT-recursion}(3)).

The definition of detectability is as follows:
\begin{definition}[Defn.~6 and Lemma~10 in~\cite{van2009observability}]
\label{def:detectability}
The HMM $(A,H)$ is \emph{detectable} if
$A^\transpose|_{\clN}$ is Schur stable, i.e., the spectral radius of
$A^\transpose|_{\clN}$ is less than 1.

(As part of the proof of~\Prop{prop:duality} in this paper, it is shown that
$\clN$ is $A^\transpose$-invariant, so the restriction $A^\transpose|_{\clN}$
is well defined.  If the HMM is observable then $\clN=\{0\}$ and the Schur
condition holds vacuously.  In particular, an observable HMM is detectable.)
\end{definition}

While more can be done to relate the probabilistic definition above to
the HMM parameters $(A,H)$, we defer this program to
\Sec{sec:dual_control_system}.  Before moving to the next part of this
section, we introduce the non-degeneracy assumption of the filter
stability theory.

\begin{assumption}[Non-degeneracy condition]\label{ass:non-degeneracy}
  Each entry of the emission matrix is positive.  In particular, let $H_{\min{}} =
  \min \{H(x,z):x\in \bS,\;z\in \bO\}$.  Then $H_{\min{}}>0$.  
\end{assumption}

\begin{remark}\label{rem:surely}
  The non-degeneracy condition implies that
\begin{align*}
  \sP^\rho
(Z_1=\zeta_1,Z_2=\zeta_2,\hdots,Z_t=\zeta_t)&>0,\quad
\forall\;(\zeta_1,\zeta_2,\hdots,\zeta_t)\in\bO^t,\quad \rho\in
                                              \clP(\bS),\quad t\in \mathbb{T},\\
  \sigma_t^\rho(\ones) \geq H_{\min{}}^t & > 0, \quad \text{surely}, \quad \forall
  \;\rho\in \clP(\bS), \quad t\in \mathbb{T}.
\end{align*}
The latter means that $\pi_t^\rho(f)=\sigma_{t}^\rho(f) /\sigma_{t}^\rho(\ones)$ is well-defined for all sample
paths (the denominator $\sigma_t^\rho(\ones)$ is always
positive). Moreover, this holds for all $\rho\in \clP(\bS)$.   
\end{remark}

We introduce the assumption here for expository convenience; only a
subset of the results actually need it, and we flag each such use
explicitly.  Also note that the HMM in
Example~\ref{ex:counter-example} does not satisfy the non-degeneracy
assumption.

\subsection{Martingale increment processes}
A vector-valued mapping $e:\bO\to \Re^m$ is defined as,
\[
e(1) := \begin{bmatrix} 1 \\ 0 \\ \vdots \\ 0 \end{bmatrix}_{m\times 1}, e(2)
:= \begin{bmatrix} 0 \\ 1\\ \vdots\\ 0 \end{bmatrix}_{m\times 1}, \hdots
e(m)
:= \begin{bmatrix} 0 \\ 0 \\ \vdots \\ 1 \end{bmatrix}_{m\times 1}, \;\; e(0) :=
-e(1) -e(2) - \hdots -e(m).
\]
(Recall here that the cardinality $|\bO| = m+1$.)      
Define
\[
h^k(x):=H(x,k)-H(x,0), \quad x\in \bS ,\;\; k=1,2,\hdots,m, \quad
\text{and} \quad 
h(x):= \begin{bmatrix} h^1(x) \\ h^2(x) \\ \vdots \\
  h^m(x) \end{bmatrix}_{m\times 1},\quad x\in \bS.
\]
For each fixed $x\in \bS$, $h(x)$ is a $m\times 1$ vector.
The notation is useful to introduce the two $\clF$-martingale increment processes as follows:
\begin{align*}
  B_{t+1}(f) &:= f(X_{t+1}) - (Af)(X_t), \quad f\in \Re^d, \quad t\in\mathbb{T}^-,\\
  W_{t+1} &:= e(Z_{t+1}) - h(X_t), \quad  t\in\mathbb{T}^-.
\end{align*}
The time indexing is used so that the processes are adapted with
respect to $\clF$ (e.g., $W_{t} $ and $B_{t}$ are $\clF_{t}$-measurable). For these processes, the formulae for conditional moments
are given in the following proposition from~\cite{chang2025dual}:
\begin{proposition}[Prop.~8 in \cite{chang2025dual}]\label{prop:B_W_Gamma_R}
  Consider $B$ and $W$. Then
  \begin{align*}
    \E^\rho (B_{t+1}(f) \mid \clG_{t}) &=0, \quad
                                                        \E^\rho
                                         (|B_{t+1}(f)|^2\mid\clG_{t})
                                         = (\Gamma f)(X_t),
                                        \quad  \sP^\rho \text{-a.s.}, \;
                                         f\in\Re^d , \;  t\in\mathbb{T}^-,\\[5pt]
    \E^\rho (W_{t+1} \mid \clF_t) &=0,\quad \E^\rho (W_{t+1} W_{t+1}^\transpose \mid \clF_t) = R(X_t) ,
                                        \quad  \sP^\rho \text{-a.s.}, \; \quad \quad  t\in\mathbb{T}^-,
  \end{align*}
where
  \begin{align*}
    (\Gamma f)(x) &:= \sum_{y\in\bS} A(x,y) f^2(y) - (Af)^2(x),\quad x\in\bS, \\
    R(x) &:=  \mathrm{diag}(h(x)) + H(x,0) (I+\ones \ones^\transpose)  - h(x) h^\transpose (x), \quad x\in\bS.
  \end{align*}
  Suppose the non-degeneracy Assumption~\ref{ass:non-degeneracy} holds.
Then
\[
    R(x) \;\succeq\; H_{\min{}}\, I \;\succ\; 0 ,
    \qquad \forall\, x \in \bS .
  \]
\end{proposition}
\begin{proof}
Appendix~\ref{appdx:prelim}. 
\end{proof}

The operator $\Gamma$ is the \emph{carr\'e du champ} associated with $A$: by the
formula above, $(\Gamma f)(X_t)$ is the conditional variance of $f(X_{t+1})$
given $X_t$.

\begin{example}[m=1]
  Consider the binary-valued HMM with $\bO=\{0,1\}$. 
Then $e(1)=1$ and $e(0)=-1$.  
  For this model, $h(x)$ and $R(x)$ are both scalars given by
  \begin{equation*}
    h(x)=H(x,1) - H(x,0),\quad R(x) = (1-h^2(x)),\quad x \in\bS.
  \end{equation*}
\end{example}

For easy reference, the nomenclature for HMM-related parameters is given in Table~\ref{tab:n_hmmx}.

\begin{table}[t]
  \begin{tabularx}{\textwidth}{l l}
    \toprule
    Symbol & Description \\
    \midrule
    Spaces & \\
    $\mathbb{T}$ & Time-horizon: $\{0,1,2,\hdots,T\}$ with $T<\infty$. \\
    $\bS$ & State space: Finite set $\{1,2,\hdots,d\}$. \\
    $\bO$ & Observation space: Finite set $\{0,1,2,\hdots,m\}$. \\
    \midrule 
    Model parameters & \\
    $\rho$ & Prior ($d\times 1$ probability vector). \\
    $A$ & Markov transition matrix (row stochastic $d\times d$ matrix). \\
    $H$ & Emission matrix (row stochastic $d\times (m+1)$ matrix). \\
      \midrule 
    Other parameters & \\
    $h$ & $h(x)$ is a $m\times 1$ vector for each $x\in\bS$ (it is obtained
          from $H$). \\
    $\Gamma$ & \emph{carr\'e du champ} $\Gamma:\Re^d\to\Re^d$ (it is obtained
          from $A$).\\
    $R$ & $R(x)$ is a $m\times m$ matrix for each $x\in\bS$  (it is obtained
          from $H$).\\
      \midrule 
    Indices & \\
$\tau_0(A)$   & Mixing index of an ergodic $A$ (Eq.~\eqref{eq:tau0}).\\
$\tau_1(A)$   & Index at which the reachable sets become $p$-periodic (Eq.~\eqref{eq:tau1}).\\
    $\tau_2(A,H)$ & Observability index (Eq.~\eqref{eq:tau2}). \\
    \bottomrule
  \end{tabularx}
  \caption{Nomenclature for the HMM.}
  \label{tab:n_hmmx}
\end{table}

\subsection{Function spaces}  \label{sec:function_spaces}

Define 
\begin{align*}
  \clY_T &:= \big\{\, Y : \bS \times \mathbb{T} \to \Re
     \;\big|\; Y_t(x) \;\text{is}\; \clZ_t\text{-measurable}, \
           \forall \ x\in\bS,\ t \in \mathbb{T} \,\big\},\\
  \clV_T &:= \big\{\, V : \bS \times \mathbb{T}^- \to \Re^{m}
     \;\big|\; V_t(x) \;\text{is}\; \clZ_t\text{-measurable}, \
           \forall \ x\in\bS,\ t \in \mathbb{T}^-  \,\big\},\\
  \clU_T &:= \big\{\, U : \mathbb{T}^- \to \Re^{m}
     \;\big|\; U_t \;\text{is}\; \clZ_t\text{-measurable}, \ \forall\ t \in \mathbb{T}^-  \,\big\}.
\end{align*}
$\clY_T$ and $\clV_T$ are spaces of $\clZ$-adapted stochastic
processes: at time $t$, $Y_t$ is a
random real-valued function on $\bS$, and $V_t$ is a random
$\Re^{m}$-valued function on $\bS$.  When matrix notation is
convenient, we identify $V_t$ with an element of $\Re^{m \times d}$
whose $x$-th column is $V_t(x)$.  $\clU_T$ is referred to as the space
of admissible control inputs.

One further space is needed, of random functions at the terminal time; these
serve as terminal conditions throughout.

\begin{definition}[Function space for terminal condition]\label{defn:Z_spaces}
Define
\begin{align*}
  \clZf{T} &:= \big\{F : \bS\to\Re \;\big|\; F(x) \text{ is } \clZ_T\text{-measurable},\ \forall \,x\in\bS\big\},
  \qquad\\
  \clZfo{T}(\rho) &:= \big\{F \in \clZf{T} \;\big|\; \sigma^\rho_T(F) = 0\big\}.
\end{align*}
\end{definition}

Let $F \in \clZf{T}$.  From Doob-Dynkin lemma, there is a unique
deterministic function $s:\bS \times \bO^T \to \Re$ such that
\begin{equation*}
F(x)= s(x;Z_1,\ldots,Z_T),\quad x\in \bS.  
\end{equation*}
The sup norm is
\[
\|F\|_\infty := \max\big\{ |s(x;\zeta_1,\ldots,\zeta_T)| \;:\;
 x\in\bS,\;\zeta\in\bO^T\big\} ,
\]
the maximum being over all paths, including those of zero probability.

\begin{example}
Recall the Defn.~\ref{defn:likelihood_ratio} for the R-N
derivative $\gamma_T$.  Based on this, $\gamma_T$ is well-defined as a
$\clZ_T$-measurable random function on $\bS$, and
furthermore $\gamma_T-1 \in \clZfo{T}(\nu)$. 
\end{example}

\begin{example}[$T=0$]
$\clZf{0}$ is the space of deterministic functions on $\bS$, which is denoted
as $\Re^d$ in this paper. For each fixed $\rho\in \clP(\bS)$,
$\clZfo{0}(\rho)$ is the co-dimension 1 subspace of functions with zero
mean, i.e., $\rho(f) = \sum \rho(x)f(x)=0$.     
\end{example}

More generally, $\clZfo{T}(\rho)$ is a subspace of the finite-dimensional
space $\clZf{T}$; the dimensions of both are given, for later use,
in~\Prop{Prop:rank-is-constant}.  For this purpose,
define the maps:
\begin{equation}\label{eq:r_rho_s_rho}
\begin{aligned}
\mathfrak{p}^\rho_T&: \bO^T\times \bS \to [0,1] \quad \text{by} \quad 
\mathfrak{p}^\rho_T(\zeta,x) :=
\sP^\rho(Z_1=\zeta_1,\ldots,Z_T=\zeta_T,X_T=x),  \\[2pt]
\mathfrak{q}^\rho_T&: \bO^T\to [0,1] \quad \text{by} \quad 
\mathfrak{q}^\rho_T(\zeta)= \sum_{x\in\bS} ^\rho \mathfrak{p}_T(\zeta,x) .
\end{aligned}
\end{equation}
Then $(\mathfrak{q}^\rho_T)^{-1}(0)$ is the set of sequences in $\bO^T$ with zero
probability under $\sP^\rho$.
Because $\rho\mapsto \mathfrak{q}^\rho_T$ is linear,
$(\mathfrak{q}^\rho_T)^{-1}(0)= \cap_{x\in\supp(\rho)}
(\mathfrak{q}^x_T)^{-1}(0)$.   Note this set is empty under the
non-degeneracy Assumption~\ref{ass:non-degeneracy}.

\begin{proposition}\label{Prop:rank-is-constant}
The space $\clZf{T} \cong \Re^{d(m+1)^T}$. The dimension of subspace
$\clZfo{T}(\rho)$ depends only on $\supp(\rho)$ given by,
\[
\mathrm{dim}\big(\clZfo{T}(\rho)\big) = (d-1)(m+1)^T +
|(\mathfrak{q}^\rho_T)^{-1}(0)|,\quad \forall\, 
\rho \in \clP(\bS).
\]
\end{proposition}
\begin{proof}
Appendix~\ref{appdx:prelim}. 
\end{proof}

\subsection{Notation}  
For $F \in \clZf{T}$, we denote
\[
  \dvar^\rho\big(F(X_T)\big) \;:=\; \E^\rho\big(|F(X_T) - \pi_T^\rho(F)|^2\big).
\]
In particular, for $F \in \clZfo{T}(\rho)$ (so that $\pi_T^\rho(F) =
0$, $\sP^\rho$-a.s.),
\[
  \dvar^\rho\big(F(X_T)\big) \;=\; \E^\rho\big(|F(X_T)|^2\big).
\]
The same definition applies verbatim with $T$ replaced by any $t$.  In particular, at $t=0$, $\dvar^\rho(f(X_0)) = \rho(f^2)-\rho(f)^2$, for $f\in\Re^d$.

\section{Backward map and the variance decay property}
\label{sec:bmap}

\subsection{Backward map}
A key original concept introduced in this paper is the \emph{backward
map} $\gamma_T\mapsto y_0$, defined as
follows\footnote{The backward map was first introduced in our prior
work, in the continuous-time setting of the model~\cite{kim2024backward,kim2024variance}.
Its adaptation to the discrete-time, finite state-space setting of
this paper is new.}:
\begin{equation}\label{eq:bmap}
	y_{0}(x) :=  \E^x (\gamma_T(X_T)),\quad x\in\bS.
      \end{equation}
The function $y_0:\bS\to \Re$ is deterministic and non-negative. Note
that the
function $y_0$ depends on $T$ and on the pair $(\mu,\nu)$, because $\gamma_T$ does.  

Let
$\rho\in\clP(\bS)$ then
\[
y_0(X_0) = \E^\rho (\gamma_T(X_T) \mid X_0),\quad \sP^\rho\text{-a.s.}
\] 
Setting $\rho=\mu$ (resp.,~$\rho=\nu$), upon taking expectations,
\begin{align*}
  \nu(y_0) &=  \E^\nu (\gamma_T(X_T)) = \E^\nu\big(\E^\nu(\gamma_T(X_T)|\clZ_T)\big) = 1, \\
  \mu(y_0) &= \E^\mu\big(\gamma_T(X_T)\big) = \E^\mu\big(\E^\mu(\gamma_T(X_T)|\clZ_T)\big) = \E^\mu\big(\pi_T^\mu(\gamma_T)\big).
\end{align*}
Now, $\pi_T^\mu(\gamma_T)=\pi_T^\nu(\gamma_T^2)$, and therefore,
\[
\mu(y_0)=\E^\mu\big(\pi_T^\nu(\gamma_T^2)\big).
\]
Noting $\pi_T^\nu(\gamma_T^2)-1 = \chisq\big(\pi_T^\mu\mid\pi_T^\nu\big)$ is the $\chisq$-divergence,
\[
\E^\mu\big(\chisq\big(\pi_T^\mu\mid\pi_T^\nu\big)\big) = \mu(y_0) -
\nu(y_0).
\]
A straightforward consequence of this identity is the following
result. 
\begin{proposition}\label{prop:chisq_filter_stability}
Let $\mu,\nu\in\clP(\bS)$ with $\mu\sim\nu$. 
Suppose the following holds:
  \begin{equation}\label{eq:y0_approaches_constant}
    y_0(x) \stackrel{(T\to
  \infty)}{\longrightarrow} 1 \quad \forall \; x \in
\mathrm{supp}(\nu).
\tag{variance decay property}
\end{equation}
Then $\E^\mu(\chisq(\pi_T^\mu\mid\pi_T^\nu))\to0$.
\end{proposition}
\begin{proof}
Since $\mu \sim \nu$, 
$\mathrm{supp}(\mu)=\mathrm{supp}(\nu)$. Then
\[
\mu(y_0) - \nu(y_0) = \sum_{x\in\mathrm{supp}(\mu)} \mu(x)\,(y_0(x)-1),
\]
a finite sum. If $y_0(x)\to1$ for all $x\in\mathrm{supp}(\nu)$, each
term vanishes as $T\to\infty$, so $\E^\mu(\chisq(\pi_T^\mu\mid\pi_T^\nu))\to0$.
\end{proof}

The terminology ``~\eqref{eq:y0_approaches_constant}''
is motivated by the following contraction property of the backward map.
Indeed, by~\eqref{eq:bmap} and Jensen's inequality,
\begin{equation}\label{eq:jensens_ineq}
	\dvar^\nu\big(y_0(X_0)\big)
	\le
	\dvar^\nu\big(\gamma_T(X_T)\big).\tag{Jensen's inequality}
\end{equation}
Thus, the backward map $\gamma_T\mapsto y_0$ is non-expansive with
respect to the variance functional: it cannot increase the variance of
its argument.
Filter stability requires that this non-expansive behavior
be strengthened, asymptotically, into decay of the variance to zero as
$T\to\infty$.

\subsection{Lyapunov equation and dissipation}\label{sec:dissipation}
Let $\rho\in\clP(\bS)$ and $T\geq 1$.  For $F\in \clZf{T}$, define
\begin{equation}\label{eq:Y_and_M}
\begin{aligned}
\text{(backward process)}& \qquad Y_t(x) := \E^\rho\big(F(X_T)\mid\clZ_t\vee[X_t=x]\big),
  \quad x\in\bS, \quad t\in \mathbb{T},\\
\text{(martingale)}& \qquad M_t := Y_t(X_t), \quad t\in \mathbb{T}.
\end{aligned}
\end{equation}
By construction, $Y_t\in\clZf{t}$ for each $t\in\mathbb{T}$, with $Y_T=F$.
Since $F$ is $\clZ_T$-measurable, the Markov property gives
\[
M_t = \E^\rho\big(F(X_T)\mid\clF_t\big), \qquad \sP^\rho\text{-a.s.},
\quad t\in\mathbb{T}.
\]
That is, $M$ is an $(\clF,\sP^\rho)$-martingale, with $M_T=F(X_T)$ and
$M_0=Y_0(X_0)$.

\begin{figure}[t]
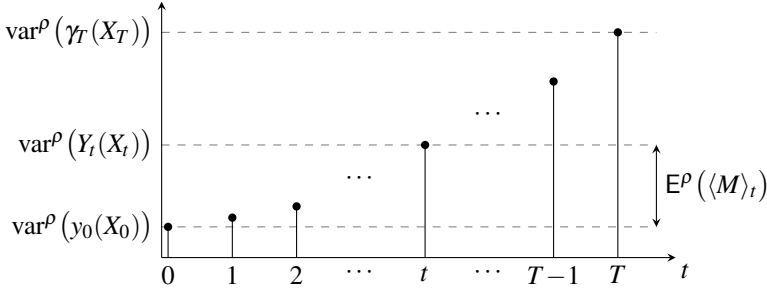

  \centering
  \include{figure/variance_decay.tex}
\caption{The backward map is embedded in a backward process for which the
    variance is a Lyapunov function, and the dissipation is the amount by
    which it decreases.}
\label{fig:lyapunov}
\end{figure}

The motivation behind this construction is as follows:  
For the special case where $\rho=\nu$ and 
$F=\gamma_T$, we have $Y_0(x)=y_0(x)$,
for each $x\in \supp(\rho)$.  In effect, for this special case, the process
$Y=\{Y_t:t\in \mathbb{T}\}$ `embeds' the two end points of the
backward map: $Y_T=F=\gamma_T$ at the terminal time and $Y_0=y_0$ at
the initial time.  The exact nature of the embedding is made precise
in the following section where the BS$\Delta$E is introduced.  The
embedding is useful because of the following result (see~\Fig{fig:lyapunov}). 

\begin{proposition}\label{Prop:Lyapunov}
Let $\rho\in\clP(\bS)$, $c\in\Re$, $F-c\ones \in \clZfo{T}(\rho)$.  Then 
\begin{equation}\label{eq:Lyapunov_equation}
\dvar^\rho\big(Y_0(X_0)\big) +
\E^\rho (\langle M\rangle_t) = \dvar^\rho\big(Y_t(X_t)\big), \quad  t\in \mathbb{T}. \tag{Lyapunov equation}
\end{equation}
where $\langle M\rangle=\{\langle M\rangle_t:t\in \mathbb{T}\}$ is the predictable quadratic variation of
$M$.
\end{proposition}
\begin{proof}
Appendix~\Sec{sec:prop:Lyapunov_proof}.  
\end{proof}

The middle term, 
$\E^\rho (\langle M\rangle_t)$, is referred to as the dissipation up to
time $t$.  Since $t\mapsto \E^\rho(\langle M\rangle_t)$ is
non-decreasing, the \eqref{eq:Lyapunov_equation} shows that
$t\mapsto\var^\rho(Y_t(X_t))$ is non-decreasing.  The variance is
thus a Lyapunov function and $\E^\rho(\langle M\rangle_t)$ is
exactly the dissipation---amount by which the Lyapunov function
decreases---over $[0,t]$.  Note that the decrease is in the backward
direction of time.

We are now ready to state the main result of this paper.  The main
result provides a characterization of the \emph{zero-dissipation set}
on which the dissipation vanishes identically.
\begin{lemma}[Invariance principle]\label{lemma:Lasalle}
  Let $\rho\in\clP(\bS)$, $c\in\Re$, and $F-c\ones \in \clZfo{T}(\rho)$.
Suppose any
one of the following conditions hold:
\begin{enumerate}
\item[(i)]  $A$ is ergodic, 
  Assumption~\ref{ass:non-degeneracy} holds, and $T\geq \tau_0(A)=:\tau$; or 
\item[(ii)]  The HMM is observable and $T \geq \tau_2(A,H) =:\tau$; or
\item[(iii)]  The HMM is detectable, 
  Assumption~\ref{ass:non-degeneracy} holds, and $T \geq  \max
  \{\tau_1(A), \tau_2(A,H)\}=:\tau$.  
\end{enumerate}
Then 
\[
\E^\rho (\langle M \rangle_\tau) = 0 \quad \implies \quad Y_0(x)=c,\quad \forall \; x\in \mathrm{supp}(\rho).
\]
\end{lemma}

\begin{remark}[Relationship to the backward map]
  Let $\rho=\nu$, $F=\gamma_T$.  Because $\gamma_T - \ones \in
  \clZfo{T}(\nu)$, Prop.~\ref{Prop:Lyapunov}
  applies with $c=1$.  Then,
  \[
    Y_0(x) =\E^\nu[F(X_T) \mid X_0=x ] = \E^\nu[\gamma_T(X_T) \mid
    X_0=x] = y_0(x),\quad x\in \supp(\nu), 
  \]
  where $y_0$ is the output of the backward
  map~\eqref{eq:bmap}.  We note the significance of the three results:
  \begin{enumerate}
     \item From \Prop{prop:chisq_filter_stability}, filter stability
       is the assertion that the variance decays to zero.
      \item \Prop{Prop:Lyapunov} shows that variance is a
      Lyapunov function and identifies the dissipation.  
     \item \Lemma{lemma:Lasalle} is the invariance principle.  Under
     appropriate conditions on the HMM, it shows that dissipation = 0
     implies variance = 0.  
  \end{enumerate}
\end{remark}

The outline of the remainder of this paper is as follows:
\begin{enumerate}
\item The remainder of this section embeds the backward map in a BS$\Delta$E,
which yields an explicit formula for the dissipation $\E^\rho(\langle M\rangle_\tau)$.
All that follows rests on this formula.
\item \Sec{sec:imp_ass} proves part (i), where $A$ is ergodic.  The argument
gives more: \Prop{prop:part-i-implication} characterizes the zero-dissipation
set under Assumption~\ref{ass:non-degeneracy} and $\tau\geq \tau_1(A)$, with none of the three
conditions of \Lemma{lemma:Lasalle}.  It is used in the proof of part (iii).
\item Parts (ii) and (iii) require a duality theory for the HMM, developed in
\Sec{sec:duality}: the BS$\Delta$E is a linear control system, and
observability and detectability of the HMM are the controllability and
stabilizability of its dual.
\item \Sec{sec:observable_proof} proves parts (ii) and (iii); notably, part
(ii) needs no assumption beyond observability.  The section closes with a
converse: if the conclusion of \Lemma{lemma:Lasalle} holds for every $\tau$,
$\rho$, $c$ and $F$, then the HMM is detectable.  
\item In the final two sections, we discuss two applications of
  \Lemma{lemma:Lasalle}:
  \begin{itemize}
  \item \Sec{sec:filter_stability} presents a novel proof of filter
    stability.  For this purpose, \Lemma{lemma:Lasalle} is applied in the manner of LaSalle, to limits of
    trajectories along which the dissipation vanishes  asymptotically.
  \item \Sec{sec:rates} makes the conclusion of \Lemma{lemma:Lasalle}
    quantitative, based on identifying the largest
constant $\beta(\rho)$ for which
\begin{equation*}
  \E^\rho (\langle M \rangle_\tau) \;\ge\; \beta(\rho)\,
    \dvar^\rho (Y_0(X_0)),
  \qquad \forall\, F -c\ones \in\clZfo{\tau}(\rho),\quad c\in \Re.
\end{equation*}
The constant $\beta(\rho)$ is used to define the Poincar\'e constant which provides
a framework for convergence rates.
\end{itemize}
\end{enumerate}

We begin with the preliminaries to introduce the BS$\Delta$E and the
associated notation.

\subsection{Embedding the backward map in a BS$\Delta$E}
The backward stochastic difference equation (BS$\Delta$E), over a
time horizon $[0,T]$, is defined as follows\footnote{With a slight
  abuse of notation, we use the notation for the backward process ($Y$) to also denote the solution of the BS$\Delta$E. This is by
  intention based on the probabilistic characterization given in \Prop{prop:backward-map-and-bsde}.}
\begin{subequations}\label{eq:dual_BSDE}
\begin{align}
Y_t(x) &= (AY_{t+1})(x) + h^\tp(x) (U_t+V_t(x)) - V_t^\transpose(x)
         e(Z_{t+1}),
  \quad x\in\bS,\quad t\in \mathbb{T}^-, \label{eq:dual_BSDE_a}\\
Y_T(x)  & = F(x),\quad x\in\bS.  \label{eq:dual_BSDE_b}
\end{align}
\end{subequations}
The process $U:=\{U_t:t\in\mathbb{T}^-\}$ is referred to as the
control input, $U\in\clU_T$, and $F$ is referred to as the terminal
condition, $F\in\clZf{T}$.
For a given $(U,F)\in\clU_T\times\clZf{T}$, a solution of the
BS$\Delta$E is a pair $(Y,V)\in\clY_T\times\clV_T$. Concerning
these, we have the following well-posedness result:

\begin{proposition}\label{prop:existence_BSDE}
Suppose $U \in \clU_T$ and $F\in \clZf{T}$.  Then there exists a
unique $(Y,V) \in \clY_T \times \clV_T$ that
solves~\eqref{eq:dual_BSDE}.  
\end{proposition} 
\begin{proof}
  Appendix~\Sec{sec:prop:existence_BSDE}.
\end{proof}

\newP{Notation} 
By \Prop{prop:existence_BSDE}, each pair
$(F,U)\in\clZf{T}\times\clU_T $ determines a unique solution of the
BS$\Delta$E~\eqref{eq:dual_BSDE}, denoted henceforth as
\[
(Y,V) = \mathrm{BS\Delta E}_T[F,U].
\]
Based on this solution, define
\[
\sfL_T : \clZf{T}\times\clU_T\to\Re^d \quad \text{by} \quad \sfL_T[F,U] := Y_0.
\]
A special case of interest is when $U=0$ where we drop the second
argument and write
\[
\text{(notation with}\;U=0\text{)} \qquad (Y,V) = \mathrm{BS\Delta E}_T[F],\qquad Y_0=\sfL_T[F].
\]
Finally, define $\ell:\Re^d\times\Re^{m\times d}\times\Re^m\times\bS\to\Re$ by
\begin{equation*}\label{eq:running_cost_formula}
  \ell(y,v,u;x):= (\Gamma y)(x) + (u+v(x))^\transpose R(x) (u+v(x)),
  \quad y\in\Re^d,\;v\in\Re^{m\times d},\;u\in\Re^m,\;x\in\bS,
\end{equation*}
where $v(x)$ denotes the $x$-th column of the $m\times d$ matrix
$v=\begin{bmatrix}v(1) & \hdots & v(d)\end{bmatrix}$.


\begin{proposition}
  \label{prop:backward-map-and-bsde}
  Let $F\in \clZf{T}$ and consider $(Y,V) = \mathrm{BS\Delta
    E}_T [F]$. Two equivalent characterizations
  of the solution are as follows:
  \begin{enumerate}
  \item An analytical characterization, obtained by solving~\eqref{eq:dual_BSDE} one step at a time,
    \[
    Y_t(x) =
    \sum_{z\in\bO}\big(\mathrm{diag}(H(\cdot,z))\,A\,Y_{t+1}^z\big)(x),
    \quad Y_{t+1}^z(x):=Y_{t+1}(x)\big|_{[Z_{t+1}=z]}, \quad x\in\bS,
    \]
    for $t\in \mathbb{T}^-$, with $Y_T(x)=F(x)$.
  \item A probabilistic characterization,
    \[
    Y_t(x) = \E^\rho\big(F(X_T)\mid\clZ_t\vee[X_t=x]\big),
    \quad x\in\bS, \quad t\in \mathbb{T}, \quad \sP^\rho\text{-a.s.}
    \]
    The representation holds for arbitrary choice of $\rho\in \clP(\bS)$. 
  \end{enumerate}
\end{proposition}
\begin{proof}
Appendix~\Sec{sec:prop:backward-map-and-bsde}.  
\end{proof}

Either characterization of \Prop{prop:backward-map-and-bsde} could
serve as a starting point for studying the backward map~\eqref{eq:bmap}. The
probabilistic characterization, in particular, was already used for
this purpose in the
preceding \Sec{sec:dissipation}.  Henceforth, we work with the 
BS$\Delta$E, in part because its solution yields a crisp formula for the
predictable quadratic variation of the martingale.  

\begin{proposition}\label{prop:M0-martingale}
Let $\rho\in\clP(\bS)$, $F\in \clZf{T}$, and $(Y,V) = \mathrm{BS\Delta
  E}_T[F]$.  Then 
  \begin{align*}
  M_t & := \E^\rho (F(X_T) \mid \clF_t) = Y_t(X_t),\quad 
        \sP^\rho\text{-a.s.}, \quad
    t\in \mathbb{T},\\
    \langle M \rangle_t & = \sum_{s=0}^{t-1} \E^\rho\big(
                          \ell(Y_{s+1},V_s,0;X_s) \mid \clF_s\big), \quad 
        \sP^\rho\text{-a.s.}, \quad
    t\in \mathbb{T}.
  \end{align*}
\end{proposition}
\begin{proof}
Appendix~\Sec{sec:martingale_proofs_M0}.  
\end{proof}

Some technical results regarding the solution are noted here for later
use.
\begin{proposition}\label{prop:technical}
Consider the notation and hypotheses of \Prop{prop:M0-martingale}.
Then
\begin{enumerate}
\item (Continuity).  The maps 
\[
(F,\rho) \mapsto \E^\rho\big(\langle M\rangle_T\big),
\qquad
(F,\rho) \mapsto \dvar^\rho\big(Y_0(X_0)\big)
\]
are polynomial: the first is linear in $\rho$ and quadratic in $F$, and the
second is quadratic in each.  In particular both are jointly continuous on
$\clZf{T}\times\clP(\bS)$.
\item (Boundedness).  For every $t\in\mathbb T$,
  \[
    \|Y_t\|_\infty \le \|F\|_\infty, \qquad
    \dvar^\rho\big(Y_t(X_t)\big) \le \|F\|_\infty^2, \qquad
    \E^\rho\big(\langle M\rangle_t\big) \le \|F\|_\infty^2 .
  \]
\item (Restriction).  Let $\tau\le T$ and consider
  $(\tilde{Y},\tilde{V})=\mathrm{BS\Delta E}_\tau[Y_\tau]$.  Then
  \[
    (\tilde{Y},\tilde{V}) 
      = \big\{ (Y_t, V_t)\,:\, 0\le t\le\tau\big\}.
  \]
\item (Centering).  If $F-c\ones\in\clZfo{T}(\rho)$ for some
  $c\in\Re$, then $Y_\tau-c\ones\in\clZfo{\tau}(\rho)$ for every
  $0\le\tau\le T$.
\end{enumerate}
\end{proposition}
\begin{proof}
Appendix~\Sec{sec:prop:technical}.  
\end{proof}


\section{Implications of Assumption~\ref{ass:non-degeneracy}}
\label{sec:imp_ass}
Even
though we begin with a proof of \Lemma{lemma:Lasalle}, part (i), the
main result of this section is \Prop{prop:part-i-implication}.  While
part (i) concerns the ergodic signal model,
\Prop{prop:part-i-implication} provides a characterization of
the zero dissipation set more generally.  
\subsection{Proof of \Lemma{lemma:Lasalle}, part (i)}
We provide a proof for the ergodic signal case (part
(i)). The proofs for parts (ii) and (iii) are given in the next 
\Sec{sec:duality}.

\begin{proof}[Proof of part (i)]
Let $(Y,V) = \mathrm{BS\Delta E}_T[F]$.  Since $\E^\rho (\langle
M\rangle_\tau) = 0$, we have $ \langle M
  \rangle_\tau = 0$, $\sP^\rho\text{-a.s.}$ which implies
  \[
    (\Gamma
  Y_{t+1})(X_t) = 0,\quad  V_t^\tp (X_t) R(X_t) V_t(X_t) =0, \quad 0\leq t\leq \tau-1, \quad \sP^\rho\text{-a.s.}.
  \]
  Because of Assumption~\ref{ass:non-degeneracy},
  \[
    V_t^\tp (X_t) R(X_t) V_t(X_t) =0 \implies V_t(X_t) = 0, \quad 0\leq t\leq \tau-1, \quad \sP^\rho\text{-a.s.}.
  \]
  Therefore, along the trajectory,
  \[
Y_t(X_t) = (AY_{t+1})(X_t),\quad  0\leq t\leq \tau-1, \quad \sP^\rho\text{-a.s.},
\]
and because $(\Gamma
Y_{t+1})(X_t) = 0$, $\sP^\rho\text{-a.s.}$, we have
\[
Y_t(X_t) = Y_{t+1}(X_{t+1}),\quad  0\leq t\leq \tau-1, \quad \sP^\rho\text{-a.s.}.
\]
Therefore, for each fixed $x\in\mathrm{supp}(\rho)$, 
\begin{align*}
  Y_0(x) &= Y_1(x'), \quad \forall \; x'\in \mathcal R_1(x) \quad
           (\text{i.e.},\; x'\in \bS \;\;\text{s.t.} \;\; A(x,x')>0),
\end{align*}
where we have dropped the qualifier ``$\sP^\rho\text{-a.s.}$''
because of Assumption~\ref{ass:non-degeneracy}
(see Remark~\ref{rem:surely}).  Therefore,
restricted to points $x'\in \mathcal R_1(x) \subset \bS$, $Y_1(\cdot)$ is a
deterministic function that moreover takes a constant value ($=Y_0(x)$).  Similarly, 
\begin{align*}
  Y_0(x) &= Y_2(x'), \quad \forall \; x'\in \mathcal R_2(x)  \quad
           (\text{i.e.},\; x'\in \bS \;\;\text{s.t.} \;\; A^2(x,x')>0),
  \\
  \vdots \quad &= \quad \vdots \\
  Y_0(x) &= Y_\tau(x'), \quad \forall \; x'\in \mathcal R_\tau(x)  \quad
           (\text{i.e.},\; x'\in \bS \;\;\text{s.t.} \;\; A^\tau(x,x')>0).
\end{align*}
Because of ergodicity hypothesis, $\bS^{\text{\tiny RECC}} \subset
R_\tau(x)$ for all $x\in \bS$. 
Therefore, the function $Y_\tau(\cdot)$ is deterministic and constant
on $\bS^{\text{\tiny RECC}}$.  Pick any $\bar{x}\in \bS^{\text{\tiny
    RECC}}$. Then 
\[
  Y_0(x) = Y_\tau(\bar{x}),\quad \forall x\in\mathrm{supp}(\rho).
\]
That is, $Y_0$ is a constant on the support of $\rho$.  The constant
is given by 
$\rho(Y_0) = \E^\rho (F(X_T)) = c$.  The final step is the {\em only} place
in the proof 
where the assumption of $F-c\ones\in\clZfo{T}(\rho)$ is used. 
\end{proof}

Based on the proof of part (i), the zero dissipation set admits a
sharper characterization: $Y_0$ coincides,
on $\supp(\rho)$, with a peripheral eigenvector of $A$.  Notably, the characterization requires none of the three conditions
of \Lemma{lemma:Lasalle}.  It needs only the non-degeneracy
Assumption~\ref{ass:non-degeneracy} and a horizon $\tau\ge\tau_1(A)$.   The
characterization is of independent interest and is used in the proof
of part (iii).

\begin{proposition}[Zero dissipation set under
  Assumption~\ref{ass:non-degeneracy}]\label{prop:part-i-implication}
Consider the notation and hypotheses of~\Prop{prop:M0-martingale}.  Suppose (i) the
non-degeneracy Assumption~\ref{ass:non-degeneracy}, (ii) the zero
dissipation hypothesis, 
$\E^\rho(\langle M\rangle_\tau)=0$, and (iii) $\tau\ge\tau_1(A)$.  Then
\begin{enumerate}
\item $Y_\tau(x') = Y_0(x)$ for every $x'\in\mathcal R_\tau(x)$ and
  $x\in\supp(\rho)$; and
\item there exists $\tilde f\in\mathrm E^{\mathrm{PER}}(A)$ such that
  $\tilde f(x)=Y_0(x)$ for every $x\in\supp(\rho)$.
\end{enumerate}
\end{proposition}
\begin{proof}
Appendix~\ref{appdx:prop:part-i-implication}.
\end{proof}

The characterization provides another proof of part (i).  If $A$ is
ergodic then $\mathrm E^{\mathrm{PER}}(A)=\mathrm{span}\{\ones\}$, so
$\tilde f = \kappa\ones$ for some $\kappa\in\Re$ and hence
$Y_0(x)=\kappa$ for $x\in\supp(\rho)$.  Under the additional
hypothesis $F-c\ones\in\clZfo{T}(\rho)$ of \Lemma{lemma:Lasalle},
$\kappa = \rho(Y_0) = \E^\rho(F(X_T)) = c$.


\begin{example}[Continued from Example~\ref{ex:counter-example}]
If $\E^\rho(\langle M\rangle_T)=0$, then $M_T=M_0$,
$\sP^\rho\text{-a.s.}$; that is,
\[
\gamma_T(X_T) = Y_0(X_0),\quad \sP^\rho\text{-a.s.}
\]
For the counter-example, however, $R(x)=1-h^2(x)=0$ for every
$x\in\bS$, so $\E^\rho(\langle M \rangle_T)=0$ no longer forces
$V_t(X_t)=0$.
\end{example}

We now turn to proofs of parts (ii) and (iii).  These proofs are based
on optimal control techniques which are described in the following section.

\section{Duality theory for HMM}\label{sec:duality} The framework
described in this section generalizes the classical 
duality theory for linear time-invariant systems.
\begin{enumerate}
\item \emph{Structural duality}. For the HMM, we have already introduced the notions of
  observability and detectability (\Sec{sec:model_props}). The dual
  notions are presented in \Sec{sec:dual_control_system}.
\item \emph{Variational duality}. The foundational link between optimal filtering and optimal
  control is the duality principle presented in
  \Sec{sec:duality_principle}.  This is followed by the formulation
  and solution of the dual optimal control problem in
  \Sec{sec:optimal_control}. 
\end{enumerate}
The presentation closely mirrors that of linear systems theory.

\subsection{Structural duality}\label{sec:dual_control_system}
The BS$\Delta$E~\eqref{eq:dual_BSDE} is referred to as the \emph{dual control
system}, with $U$ as the control input.  Since the BS$\Delta$E is linear,
controllability is characterized by the range of the input-to-state map,
$(c\ones,U)\mapsto Y_0$, with terminal condition $F=c\ones$.  Recall that, in linear
systems theory, the terminal condition is set to zero when defining
controllability~\cite[Defn.~11.2]{hespanha2018linear}.  Here the constant $c$ is retained because
$\sfL_T[\ones,0]=\ones$.  

\begin{definition}[Controllability and stabilizability]\label{def:controllable}
  The \emph{controllable subspace} of the dual control system~\eqref{eq:dual_BSDE} is
\[
\clC_T := \big\{ \sfL_T[c\ones, U] \;:\; c\in\Re,\ U\in\clU_T \big\}.
  \]
Set $\clC:=\bigcup_{T>0}\clC_T$ (the union is well-defined as a
subspace based on a nestedness property shown in
\Prop{prop:CT-recursion}(2) below). 
The dual control system~\eqref{eq:dual_BSDE} is \emph{controllable} if
$\clC = \Re^d$, and \emph{stabilizable} if
$\mathrm E^{\mathrm{PER}}(A) \subseteq \clC$.
\end{definition}

The following two propositions characterize $\clC_T$ and 
$\clC$.

\begin{proposition}\label{prop:CT-recursion}
The sequence $\{\clC_T:T\geq 0\}$ satisfies the following:
\begin{enumerate}
\item (Recursion formula) For every $T\ge1$,
\[
\clC_T = A\clC_{T-1} + \sum_{k=1}^m \big(\mathrm{diag}(h^k)\,A\big)\clC_{T-1}, \qquad \clC_0:=\mathrm{span}\{\ones\}.
\]
\item (Monotonicity) $\clC_{T-1}\subseteq\clC_T$ for every $T\ge1$.
\item (Stabilization) $\clC_T = \clC$ for every
  $T \ge \tau$, where $\tau:=  \min\{T : \clC_{T+1} = \clC_T\} \le
  d-1$.
\end{enumerate}
\end{proposition}
\begin{proof}
Appendix~\Sec{sec:proof:controllability_main}. 
\end{proof}

\begin{corollary}\label{cor:controllable_space}
$\clC$ is the smallest subspace in $\Re^d$ that satisfies:
\begin{enumerate}
\item The constant function $\ones\in\clC$; and
\item If $g\in\clC$, then $Ag\in\clC$ and $h^k \, (Ag)\in\clC$
  for $k=1,\ldots,m$.
\end{enumerate}
\end{corollary}
\begin{proof}
Appendix~\Sec{sec:proof:controllability_main}. 
\end{proof}

The following proposition relates these control-theoretic properties
to the corresponding properties of the HMM introduced in
Sec.~\ref{sec:hmm-prelim}.

\begin{proposition}\label{prop:duality}
$\clN_T = \clC_T^{\perp}$ for every $T > 0$.  In particular, the
  two indices coincide:
  \[\min\{T : \clC_T = \clC\} = \min\{T : \clN_T = \clN\}
    = \tau_2(A,H).
  \]
  Consequently, the dual control system is
  controllable (resp.\ stabilizable) if and only if the
  HMM $(A,H)$ is observable (resp.\ detectable).
\end{proposition}
\begin{proof}
Appendix~\Sec{sec:proof:controllability_duality}.
\end{proof}

\subsection{Variational duality}\label{sec:duality_principle}

Let
$U\in\clU_T$, $F\in\clZf{T}$, and $(Y,V) = \mathrm{BS\Delta E}_T[F,U]$
and define
\[
\sJ_T^\rho(U;F):= \dvar^\rho\big(Y_0(X_0)\big) +
\E^\rho\Big(\sum_{s=0}^{T-1} \ell\big(Y_{s+1},V_s,U_s;X_s\big)\Big).
\]
The following proposition states the duality principle.  Part 1
generalizes \Prop{prop:M0-martingale} 
from the special case $U = 0$ to an arbitrary control $U\in\clU_T$. 

\begin{proposition}\label{prop:MU-martingale}
Let $\rho\in\clP(\bS)$, $U\in\clU_T$, $F\in\clZf{T}$, and
$(Y,V)=\mathrm{BS\Delta E}_T[F,U]$.
\begin{enumerate}
\item Define the $\clF$-adapted process
$M^U:=\{M^U_t:0\leq t\leq T\}$ by
\[
M_t^U:= Y_t(X_t) + \sum_{s=0}^{t-1} U_s^\tp e(Z_{s+1}),\quad 0\leq t\leq T.
\]
Then $M^U$ is a $(\clF,\sP^\rho)$-martingale, with quadratic variation
\[
\langle M^U \rangle_t = \sum_{s=0}^{t-1} \E^\rho\big(\ell(Y_{s+1},V_{s},U_{s};X_{s}) \mid \clF_s\big),
\quad 0\leq t\leq T,\quad \sP^\rho\text{-a.s.}.
\]
Setting $U=0$ recovers the martingale $M$ of
\Prop{prop:M0-martingale} as a special case. 
\item Consider an estimator
  \[
    S_T := \rho(Y_0) - \sum_{t=0}^{T-1}
    U_{t}^\transpose e(Z_{t+1}).
\]
    Then
\begin{equation}\label{eq:duality-principle}
\E^\rho\big(|F(X_T)-S_T|^2\big) = \sJ_T^\rho(U;F).\tag{duality principle}
\end{equation}
\end{enumerate}
\end{proposition}
\begin{proof}
Appendix~\ref{sec:martingale_proofs}.
\end{proof}

\begin{remark}\label{rem:estimator-bijection} 
  Taking expectations in (part~1) at $t = T$ gives,
  \begin{align*}
   \E^\rho (F(X_T)) &= \E^\rho(M_T^U) - \sum_{t=0}^{T-1}
    \E^\rho \big( U_{t}^\transpose e(Z_{t+1}) \big) \\
    & = \E^\rho(M_0^U) - \sum_{t=0}^{T-1}
    \E^\rho \big( U_{t}^\transpose e(Z_{t+1}) \big) =
      \E^\rho(S_T),\quad \forall \,U\in \clU_T.
  \end{align*}
Therefore, $S_T$ is an unbiased estimator of $F(X_T)$.  Moreover,
since unbiasedness determines the constant in the decomposition of
\Corollary{prop:existence_uniqueness_ST_U}, the map $U \mapsto S_T$ is
a \emph{bijection} from $\clU_T$ onto the set of all unbiased $\clZ_T$-measurable
estimators of $F(X_T)$. Consequently, minimizing
the control objective $\sJ^\rho_T(\,\cdot\,;F)$ over $\clU_T$ is
equivalent to minimizing the mean-squared error over all such
estimators, whose solution is the conditional expectation
$\pi^\rho_T(F)$.  This is the sense in which
\eqref{eq:duality-principle} is an expression of duality between
optimal filtering and optimal control.
\end{remark}

\subsection{Dual optimal control problem}\label{sec:optimal_control}
Let $\rho\in\clP(\bS)$
and $F\in\clZ_T$.  The dual optimal control problem is given by,
\begin{equation}\label{eq:dual-optimal-control}
  \min_{U \in \clU_T} \; \sJ_T^\rho(U;F) 
  \quad \text{subject to} \quad (Y,V)=\mathrm{BS\Delta E}_T[F,U].\tag{OCP}
\end{equation}
An explicit solution of~\eqref{eq:dual-optimal-control} is given 
in our prior paper~\cite[Thm.~12]{chang2025dual}.  The hypotheses of
\Lemma{lemma:Lasalle} requires only a consideration of functions
$F-c\ones\in\clZfo{T}(\rho)$.  For such
functions, the solution is described in the following corollary.  A
self-contained proof is provided.

\begin{corollary}\label{cor:U-opt-zero}
  Let $c\in\Re$.  Consider the~\eqref{eq:dual-optimal-control} with $F-c\ones
  \in\clZfo{T}(\rho)$. Then
  \begin{enumerate}
  \item For every $U\in\clU_T$,
      \begin{equation}\label{eq:jensens_ineq_stronger_OCP_general}
\dvar^\rho\big(F(X_T)\big) \leq \dvar^\rho\big(Y_0(X_0)\big) +
\E^\rho(\langle M^U\rangle_T).\tag{optimality principle}
\end{equation}
\item An optimal control is $U^\opt=0$ and
  the inequality holds as an
  equality at $U=U^\opt$.  
\end{enumerate}
\end{corollary}
\begin{proof}
The MMSE estimator of $F(X_T)$ among $\clZ_T$-measurable
estimators is the conditional expectation,
\[
  \E^\rho (F(X_T)\mid\clZ_T) = \pi_T^\rho(F) = c,\quad \sP^\rho\text{-a.s.}.
\]
Fix $U\in\clU_T$, and let $(Y,V)=\mathrm{BS\Delta E}_T[F,U]$ and
$S_T=\rho(Y_0)-\sum_{t=0}^{T-1}U_t^\transpose e(Z_{t+1})$, which is
$\clZ_T$-measurable. By the MMSE property and the duality principle,
\[
  \var^\rho\big(F(X_T)\big) = \E^\rho\big(|F(X_T)-c|^2\big)
  \leq \E^\rho\big(|F(X_T) - S_T|^2\big) = \sJ_T^\rho (U;F),
\]
proving~\eqref{eq:jensens_ineq_stronger_OCP_general} for every $U\in\clU$.
At $U=0$, using the probabilistic solution characterization in
\Prop{prop:backward-map-and-bsde}, $S_T=\rho(Y_0)=\E^\rho(F(X_T))=c$,
attaining equality above. Hence, 
$U=0$ is an optimal control\footnote{Under the non-degeneracy Assumption~\ref{ass:non-degeneracy}, the minimizer
of $\sJ_T^\rho(U;F)$ is in fact unique~\cite[Thm.~12]{chang2025dual}, so equality
in~\eqref{eq:jensens_ineq_stronger_OCP_general} holds iff 
$U=U^\opt=0$. Uniqueness is not needed for the purposes of this paper.}.
\end{proof}

There is a more elegant way to frame
the~\eqref{eq:jensens_ineq_stronger_OCP_general}: as the statement
that a certain process is a supermartingale, with equality (a true
martingale) at the optimal control.  The result is used for 
the proof of \Prop{prop:var_contractive} in this paper.   

\begin{proposition}\label{cor:U-opt-zero_elegant}
Let $c\in\Re$.  Consider the~\eqref{eq:dual-optimal-control} with $F-c\ones
  \in\clZfo{T}(\rho)$. 
Let $U\in\clU_T$. Then the process $\Theta^U := \{\Theta^U_t
: 0 \le t \le T\}$ defined as follows
\[
\Theta_t^U := \pi_t^\rho\big((Y_t-\pi_t^\rho(Y_t))^2\big) -
\sum_{s=0}^{t-1} \E^\rho \big( \ell(Y_{s+1},V_s,U_s;X_s) \mid \clZ_s\big), \qquad t\in\mathbb{T}
\]
is a $(\clZ,\sP^\rho)$-supermartingale. It is a martingale for
$U = U^\opt$. 
\end{proposition}
\begin{proof}
Appendix~\ref{appendix:OCP_supermg}.  Using the martingale characterization,
the~\eqref{eq:jensens_ineq_stronger_OCP_general} is a one-line proof:
\[
\dvar^\rho\big(F(X_T)\big) - \sum_{s=0}^{T-1} \E^\rho \big(
\ell(Y_{s+1},V_s,U_s;X_s)
\big) =\E^\rho(\Theta_T^U) \leq
\E^\rho(\Theta_0^U) = \dvar^\rho\big(Y_0(X_0)\big)
\]
with equality holding at $U=U^\opt=0$.   
\end{proof}

We are now ready to complete
the proof of parts (ii) and (iii) of the main result described in
\Lemma{lemma:Lasalle}.  

\section{Proof of \Lemma{lemma:Lasalle}, parts (ii) and (iii)}
\label{sec:observable_proof}

\begin{proof}[Proof of part (ii)]  We show that
\begin{equation}\label{eq:obs_identity}
\rho\big((Y_0-\rho (Y_0))\tilde f\big) = 0, \qquad \forall\,\tilde f\in\clC.
\end{equation}
The identity is used also in the proof of part (iii).

We work over the time-horizon $[0,\tau]$.  For this purpose, consider 
\[
Y_\tau(x) := \E^\rho\big(F(X_T) \mid \clZ_\tau \vee [X_\tau =x] \big),\quad
x\in \bS. 
\]
From~\Prop{prop:technical}(4),
\[
F - c\ones \in \clZfo{T}(\rho)\quad \implies \quad Y_\tau - c\ones \in
\clZfo{\tau}(\rho).
\]

Next, since $\tau\ge\tau_2(A,H)$, \Prop{prop:duality} gives
$\clC_\tau=\clC$.  Therefore, for each $\tilde f\in\clC$ there exist
$\tilde U\in\clU_\tau$ and $\tilde c\in\Re$ such that 
  \[
\sfL_\tau[\tilde c \ones, \tilde U]=\tilde f \quad \implies \quad \sfL_\tau[0, \tilde U] =
\tilde f - \tilde c\ones.
\]

Consider the solutions of the BS$\Delta$Es, and their respective
martingales, 
\begin{align*}
(Y,V) &= \mathrm{BS\Delta E}_\tau[Y_\tau,0],
& M_t &= Y_t(X_t), 
\quad & 0\le t\le \tau,\\
(\tilde Y,\tilde V) &= \mathrm{BS\Delta E}_\tau[0,\tilde U],
&\tilde M_t &= \tilde Y_t(X_t) + \sum_{s=0}^{t-1}\tilde U_s^\tp e(Z_{s+1}) ,
              \quad & 0\le t\le \tau, \\
              (\bar Y,\bar V) &= \mathrm{BS\Delta E}_\tau[Y_\tau,\epsilon\tilde U], \qquad
&\bar M_t &= \bar Y_t(X_t) + \epsilon\sum_{s=0}^{t-1}\tilde U_s^\tp e(Z_{s+1}),
\quad & 0\le t\le \tau,
\end{align*}
where $\epsilon\in\Re$ is arbitrary.  From
the~\eqref{eq:jensens_ineq_stronger_OCP_general}, because $ \bar
Y_\tau= Y_\tau$, 
\begin{align*}
\var^\rho\big(Y_\tau (X_\tau)\big) = \var^\rho\big(\bar Y_\tau (X_\tau)\big)& \le \var^\rho\big(\bar Y_0(X_0)\big) +
\E^\rho (\langle \bar M\rangle_\tau).
\end{align*}

By linearity of the BS$\Delta$E, $\bar Y= Y+\epsilon\tilde Y$ and
$\bar M = M+\epsilon\tilde M$. Consequently, because $\tilde Y_0 =
\tilde f - \tilde c\ones$, the two terms on the right hand side are
given by,
\begin{align*}
\var^\rho\big(\bar Y_0(X_0)\big) &= \var^\rho\big(Y_0(X_0)\big) +
2\epsilon\,\rho\big((Y_0-\rho (Y_0))\tilde f\big) +
\epsilon^2\var^\rho\big(\tilde f(X_0)\big), \\
\E^\rho\big(\langle \bar M\rangle_\tau \big) &= \E^\rho\big(\langle
M\rangle_\tau\big) + 2\epsilon\,\E^\rho\big(\langle M,\tilde
M\rangle_\tau\big) + \epsilon^2\,\E^\rho\big(\langle\tilde M\rangle_\tau\big).
\end{align*}
We have the following:
\begin{enumerate}
\item $O(1)$ term: By the zero-dissipation hypothesis, $\E^\rho\big(\langle
  M\rangle_\tau\big) =0$, which implies
  \[
\var^\rho\big(Y_0(X_0)\big) = \var^\rho\big(Y_\tau(X_\tau)\big).
\]
\item  $O(\epsilon)$ term: Since $\langle M\rangle_\tau \ge 0$, the
  hypothesis gives $\langle M\rangle_\tau = 0$, $\sP^\rho$-a.s.  The
  Kunita--Watanabe inequality (see~\cite[p.~88]{le2016brownian}) then yields
  \[
    \big|\langle M,\tilde M\rangle_\tau\big|^2
      \le \langle M\rangle_\tau \, \langle \tilde M\rangle_\tau = 0,
    \qquad \sP^\rho\text{-a.s.},
  \]
  and therefore $\E^\rho\big(\langle M,\tilde M\rangle_\tau\big)=0$.
\item $O(\epsilon^2)$ term: For the fixed $\tilde U$,
  \[
    C := \var^\rho(\tilde f(X_0)) + \E^\rho(\langle\tilde M\rangle_\tau)
    < \infty,
  \]
  because $\bS$ and $\bO$ are finite and the sum defining
  $\langle\tilde M\rangle_\tau$ has finitely many ($\tau$) terms.
\end{enumerate}
Collecting these terms together, 
\[
0 \le 2\epsilon\,\rho\big((Y_0-\rho (Y_0))\tilde f\big) +
\epsilon^2 C. \qquad \forall\,\epsilon\in\Re.
\]
As the $O(\epsilon^2)$ coefficient is finite and $\epsilon$ ranges
over all of $\Re$, the coefficient of $\epsilon$ must vanish.  This
proves~\eqref{eq:obs_identity}.

We now use the hypothesis of part (ii), that the HMM is observable.
By \Prop{prop:duality}, this is equivalent to $\clC=\Re^d$, so
\eqref{eq:obs_identity} holds for every $\tilde f\in\Re^d$.  Taking
$\tilde f=\ones_{\{x\}}$ gives
$\rho(x)\big(Y_0(x)-\rho(Y_0)\big)=0$, and therefore
\[
  Y_0(x) = \rho(Y_0) = \E^\rho\big(F(X_T)\big) = c,
  \qquad \forall\, x\in\supp(\rho),
\]
the last equality because $F-c\ones\in\clZfo{T}(\rho)$.

\end{proof}

\begin{proof}[Proof of part (iii)]
From \Prop{prop:part-i-implication}, there exists $\tilde f \in \mathrm{E}^{\mathrm{PER}}(A)$ such
that
\[
  \tilde f(x) = Y_0(x), \qquad \forall\, x \in \supp(\rho).
\]
Detectability gives $\mathrm{E}^{\mathrm{PER}}(A)\subseteq\clC$
(\Prop{prop:duality}).  So, $\tilde f$ is admissible
in~\eqref{eq:obs_identity}.  Since $\tilde f$ and $Y_0$ agree on
$\supp(\rho)$,
\[
  0 = \rho\big((Y_0-\rho(Y_0))\tilde f\big)
    = \rho\big((Y_0-\rho(Y_0))Y_0\big)
    = \var^\rho\big(Y_0(X_0)\big),
\]
and therefore $Y_0(x) = \rho(Y_0) = \E^\rho(F(X_T)) = c$ for every
$x\in\supp(\rho)$. 
\end{proof}

\subsection{Converse of the invariance principle}\label{sec:necessity}

Lemma~\ref{lemma:Lasalle}\,(iii) asserts that detectability is
sufficient for the invariance principle.  A converse is given in the
following (note it does not need Assumption~\ref{ass:non-degeneracy}).   

\begin{proposition}[Necessity of detectability]\label{thm:necessity}
  Consider~\eqref{eq:Y_and_M} with $T=\tau$.  
Suppose for every $\tau >0$, $\rho\in\clP(\bS)$, $c\in\Re$, and  $F -
c\ones \in \clZfo{\tau}(\rho)$, 
\[
\E^\rho \big( \langle M \rangle_\tau \big) = 0 \quad\Longrightarrow\quad
Y_0(x) = c, \qquad \forall\, x\in\supp(\rho).
\]
Then the $\mathrm{HMM}(A,H)$ is detectable.
\end{proposition}
\begin{proof}
Appendix~\ref{appdx:thm:necessity}.  
\end{proof}

\subsection{Dissipation and variance}  At this point, it is useful to
introduce some additional notation which makes the dependencies
explicit.  Specifically, consider~\eqref{eq:Y_and_M}.  Note that
$M_t=Y_t(X_t)$ and, by the Markov property of $X$, $Y_t$ does not depend
upon $\rho$.  Define, 
\[
  \text{(energy)} \qquad  \clE_T : \clZf{T}\to\Re^d
    \quad\text{by}\quad (\clE_T[F])(x)
      := \E^x\big(\langle M\rangle_T\big), \quad x\in\bS .
    \]
For each $x$, $F\mapsto(\clE_T[F])(x)$ is a non-negative quadratic
form.  Notably, the energy does not depend upon $\rho$.  The other
object of interest is $Y_0$ which also depends \emph{only} upon $F$, 
given by $Y_0 = \sfL_T[F]$. Then for each $\rho\in \clP(\bS)$, 
\begin{align*}
  \text{(dissipation)} \qquad 
\E^\rho \big( \langle M \rangle_T \big) = \sum_{x\in
  \bS} \rho(x)  \E^x \big( \langle M \rangle_T \big) =
  \rho(\clE_T[F]),\\
  \text{(variance)} \qquad \dvar^\rho(Y_0(X_0)) = \rho(
(\sfL_T[F])^2) - \big( \rho(\sfL_T[F]) \big)^2.
\end{align*}
The notation is useful because it makes the dependencies on
$F$ and $\rho$ explicit.  The content of \Lemma{lemma:Lasalle} and \Prop{thm:necessity} is a
relationship between these two objects.

\section{Filter stability}\label{sec:filter_stability}

Consider the BS$\Delta$E solution and the associated martingale,
\[
(Y, V) = \mathrm{BS\Delta E}_T[\gamma_T],\qquad M_t = Y_t(X_t),\quad
t\in\mathbb{T}. 
\]
Our goal is to apply \Lemma{lemma:Lasalle} to show
that the variance decay property holds, i.e.,
\[
\dvar^\nu(Y_0(X_0)) \to 0,\quad \text{as}\quad T\to\infty.
\]
Filter stability then follows from using
\Prop{prop:chisq_filter_stability}.  We begin by describing a uniform
bound on the BS$\Delta$E solution.

\begin{proposition}[Uniform bound under equivalent
  priors]\label{lem:uniform_ratio_bound} Suppose 
  Assumption~\ref{ass:equivalent-priors} holds s.t. $
  \gamma_0^{\min{}} \leq \gamma_0(x) \leq \gamma_0^{\max{}}$, 
  $\forall\,x\in \supp(\nu)$.  
  Then, under the Defn.~\ref{defn:likelihood_ratio} for $\gamma_T$,
\[
0<\alpha \;\leq\; Y_t(x) \;\leq\; \alpha^{-1} < \infty,\quad\forall\, x \in \bS,\quad t\in \mathbb{T}.
\]
where recall $\alpha = \gamma_0^{\min{}} /\gamma_0^{\max{}}$.
\end{proposition}
\begin{proof}
  Appendix~\ref{appdx:lem:uniform_ratio_bound}.
\end{proof}

We are now ready to state the main result on filter stability.  

\begin{theorem}\label{thm:filter-stability}
Suppose $\mu\sim\nu$ and at least one of the three conditions in
Lemma~\ref{lemma:Lasalle} holds. Then
$
  \E^\mu\big(\chisq(\pi_T^\mu\mid\pi_T^\nu)\big) \stackrel{(T\to
  \infty)}{\longrightarrow} 0$.
\end{theorem}

Throughout the remainder of this section,
Assumption~\ref{ass:equivalent-priors}
is in force and $\tau$ is fixed according to
one of the three conditions in
Lemma~\ref{lemma:Lasalle}.

\subsection{Proof strategy and technical issues}
By the uniform bound in \Prop{lem:uniform_ratio_bound},
\[
\dvar^\nu\big(\gamma_{T}(X_{T})\big) < C, \quad\forall\;T.
\]
From~\eqref{eq:Lyapunov_equation},
\[
\dvar^\nu\big(Y_0(X_0)\big) +
\E^\nu(\langle M \rangle_T) = \dvar^\nu\big(\gamma_T(X_T)\big) .\
\]
Using formula for the quadratic variation in
\Prop{prop:M0-martingale}, 
\[
\E^\nu\big(\langle M \rangle_{T}\big) = \sum_{s=0}^{T-1} \E^\nu( \ell(Y_{s+1},V_s,0;X_s))
\leq \dvar^\nu\big(\gamma_{T}(X_{T})\big) < C, \quad \forall \; T.
\] 
For every $\epsilon>0$, there exists $T_\epsilon<\infty$ such that,
for every $T> T_\epsilon$, there exists a time $
\bar t=\bar t(T,\epsilon)\in[0,T-\tau]
$ such that
\[
\sum_{s=\bar t}^{\bar t+\tau-1}  \E^\nu( \ell(Y_{s+1},V_s,0;X_s))< \epsilon.
\]
Therefore, we have a time interval $[\bar t, \bar t + \tau]$ over
which the solution $\{(Y_s, V_s) : \bar t \leq s \leq \bar t + \tau\}$
has small dissipation. There are two issues with a direct application
of \Lemma{lemma:Lasalle}:
\begin{enumerate}
\item The lemma only asserts that dissipation $=0$ implies variance
  $=0$. One therefore needs to pass to the limit $\epsilon \to 0$.
\item Passing to this limit is complicated by the fact that the
  interval $[\bar t, \bar t + \tau]$ over which the dissipation is
  small depends on $T$. In particular, $\bar t = \bar t(T, \epsilon)$
  may grow unboundedly as $\epsilon \to 0$ and $T \to \infty$.
\end{enumerate}
To address the second issue, we construct a \emph{shifted solution} of
a related BS$\Delta$E on the fixed interval $[0, \tau]$.

\subsection{Construction of a shifted solution} The shifted solution
is obtained on an auxiliary measurable space.  The notation for the
same is fixed with the following definition.

\begin{definition}[Auxiliary measurable space]
The \emph{auxiliary measurable space} $(\tilde\Omega,\tilde\clF_\tau)$ is
defined with an independent copy of $(\tilde X,\tilde Z) =
\{\tilde X_0,(\tilde X_t,\tilde Z_t):1\leq t\leq\tau\}$.  For each fixed
$\rho\in\clP(\bS)$, $\tilde\sP^\rho$ is the probability measure on
$(\tilde\Omega,\tilde\clF_\tau)$ corresponding to the
$\mathrm{HMM}(A,H)$ with prior $\tilde X_0\sim\rho$, and $\tilde\E^\rho$
is the associated expectation.
\end{definition}

The construction of the shifted solution is described next.  For both
the proposition and the corollary that follows, we fix $T$ and $\bar t
\in [0, T-\tau]$, and let $\{(Y_s, V_s) : \bar t \leq s \leq \bar
t+\tau\}$ denote the BS$\Delta$E solution restricted to the time
interval $[\bar t, \bar t+\tau]$.  The key step in the construction of
the shifted solution is the following Markov property of the
BS$\Delta$E.

\begin{proposition}[Markov property of the BS$\Delta$E]
  \label{prop:Markov_BSDE}
 There exists a deterministic function
$\shift:\clP(\bS)^2\times \bO^\tau \times \bS \to\Re$ such that
\[
 Y_{\bar t + \tau}(x) =
\shift(\pi_{\bar t}^\nu,\pi_{\bar t}^\mu,Z_{\bar t+1},Z_{\bar
  t+2},\hdots,Z_{\bar t + \tau};x),\quad x\in \bS.
\]
The function $\shift (\cdot)$ depends upon $\bar t, T$, and the HMM parameters but
not upon $(Z_1,Z_2,\hdots,Z_{\bar t})$.
\end{proposition}  
\begin{proof}
See Appendix~\ref{sec:markov_property}.
\end{proof}

\begin{definition}[Shifted solution]\label{defn:shifted_solution}
  On the auxiliary space $(\tilde\Omega,\tilde\clF_\tau)$, set
  \begin{subequations}\label{eq:construction}
\begin{equation}
\tilde F(x)\big|_{[\tilde Z_1=\zeta_1,\ldots,\tilde Z_\tau=\zeta_\tau]}
:= \shift\bigl(\pi_{\bar t}^\nu,\pi_{\bar t}^\mu,
\zeta_1,\ldots,\zeta_\tau;x\bigr), \quad x\in\bS, \quad
(\zeta_1,\ldots,\zeta_\tau)\in\bO^\tau.
\end{equation}
The \emph{shifted solution} $\{(\tilde Y_t,\tilde V_t):0\leq
t\leq\tau\}$ and the associated martingale are defined by,
\begin{equation}
(\tilde Y,\tilde V)=\mathrm{BS\Delta E}_\tau[\tilde F],\qquad \tilde
M_t := \tilde Y_t(\tilde X_t),\quad 0\leq t\leq\tau.
\end{equation}
The auxiliary space is equipped with a random probability measure
$\tilde\sP^\rho$, defined as follows: for any
$\tilde\clF_\tau$-measurable $\Psi:\tilde\Omega\to\Re$,
\begin{equation}
\tilde\E^\rho(\Psi)(\omega) := \tilde\E^{\rho(\omega)}(\Psi),
\qquad\omega\in\Omega,
\qquad\text{where } \rho = \pi_{\bar t}^\nu .
\end{equation}
\end{subequations}
The subspace $\clZfo{\tau}(\rho)$ is likewise a random
subspace, and $\tilde F$, $(\tilde Y,\tilde V)$ and $\tilde M$ are random
elements, all read pointwise in $\omega\in\Omega$.   Because $\rho=\pi_{\bar t}^\nu$ is defined
only on $[\sigma_{\bar t}^\nu(\ones)>0]$, the construction is defined
$\sP^\nu$-a.s. and all statements concerning it are understood in this
sense.  Under Assumption~\ref{ass:non-degeneracy}, the construction
and the statements hold surely.  
\end{definition}

The following corollary relates the two solutions. 

\begin{corollary}\label{corollary:prop:Markov_BSDE}
Consider~\eqref{eq:construction}.  
Then 
\begin{enumerate}
\item $\tilde F-1 \in
  \clZfo{\tau}(\rho)$, and 
\[
  0 < \alpha \le \tilde F(x) \le
  \alpha^{-1} < \infty, \quad \sP^\nu\text{-a.s.},
\quad x\in\bS.
\]
\item For the shifted solution, the variance and the dissipation
\begin{align*}
  \rho\big((\tilde{Y}_0-1)^2\big) &= \E^\nu\big((Y_{\bar
    t}(X_{\bar t})-1)^2 \mid \clZ_{\bar t}\big), \quad \sP^\nu\text{-a.s.}
,\\
 \tilde\E^\rho(\langle \tilde M\rangle_\tau) &= \sum_{s=\bar
   t}^{\bar t+\tau-1} \E^\nu\big( \ell(Y_{s+1},V_s,0;X_s)\mid \clZ_{\bar t}\big), \quad \sP^\nu\text{-a.s.}
\end{align*}
\end{enumerate}
\end{corollary}
\begin{proof}
Appendix~\ref{appdx:corollary:prop:Markov_BSDE}.  
\end{proof}

\subsection{Proof of \Theorem{thm:filter-stability}} Let
$(Y,V)=\mathrm{BS\Delta E}_T[\gamma_T]$.  We show
$\dvar^\nu(Y_0(X_0))\to0$ as $T\to\infty$.  Filter stability then
follows from Prop.~\ref{prop:chisq_filter_stability}.

The proof is by contradiction.  Suppose there exists $\varepsilon>0$
and time-horizons $T_n\to\infty$ such that, writing $(Y^n,V^n)=\mathrm{BS\Delta
  E}_{T_n}[\gamma_{T_n}]$,  
\[
\dvar^\nu\big(Y_0^n(X_0)\big) > \varepsilon, \qquad \forall\,n.
\]
Using the~\eqref{eq:Lyapunov_equation} together with the uniform bound in~\Prop{lem:uniform_ratio_bound},
\[
\sum_{s=0}^{T_n-1} \E^\nu\big( \ell (Y_{s+1}^n,V_s^n,0;X_s) \big) \leq \dvar^\nu\big(\gamma_{T_n}(X_{T_n})\big) < C.
\]
Partition $[0,T_n]$ into $\lfloor T_n/\tau\rfloor$ disjoint windows
of length $\tau$. Since the total dissipation over these windows is
at most $C$, there exists, for each $n$, a $t_n\leq T_n-\tau$ such
that
\[
\sum_{s=t_n}^{t_n+\tau-1} \E^\nu\big( \ell (Y_{s+1}^n,V_s^n,0;X_s) \big)   <  \frac{C}{\lfloor T_n/\tau\rfloor}.
\]
Since the right-hand side tends to $0$, passing to a subsequence,
we have
\[
  \dvar^\nu\big(Y_{t_n}^n(X_{t_n})\big) > \varepsilon,\qquad \sum_{s=t_n}^{t_n+\tau-1}\E^\nu\big( \ell (Y_{s+1}^n,V_s^n,0;X_s) \big)  < \frac{1}{n},\quad \forall n.
\]
where we used $\dvar^\nu\big(Y_{0}^n(X_{0})\big) \leq
\dvar^\nu\big(Y_{t_n}^n (X_{t_n})\big)$.

Apply Defn.~\ref{defn:shifted_solution} with $(T,\bar t) = (T_n,t_n)$, and
denote the resulting objects $\rho_n$, $\tilde F_n$, $(\tilde
Y^n,\tilde V^n)$, $\tilde M^n$. From
Corollary~\ref{corollary:prop:Markov_BSDE},  
\[
\E^\nu \big( \rho_n((\tilde Y_0^n-1)^2) \big) >\epsilon,\qquad \E^\nu
\big( \tilde{E}^{\rho_n}(
\langle\tilde M^n \rangle_\tau)\big) < \frac{1}{n},\quad \forall \, n.
\]

Set $K := \clP(\bS)\times[\alpha,\alpha^{-1}]^{d(m+1)^\tau} \subset
\clP(\bS) \times \clZf{\tau}$  
and
\[
K_0 := \big\{(\rho,F)\in K \;:\; F - 1 \in
\clZfo{\tau}(\rho)\big\}=\big\{(\rho,F)\in K \;:\;
\sigma_\tau^\rho(F)= \sigma_\tau^\rho(1)
\big\}.
\]
The map $(\rho,F) \mapsto \sigma_\tau^\rho(F)$ is linear in each
argument, and hence
continuous.  So, $K_0$ is a closed subset of the compact set $K$.  By
Corollary~\ref{corollary:prop:Markov_BSDE}, $(\rho_n,\tilde F_n)$ is a
$K_0$-valued random element for every $n$.

Let $\lambda_n\in\clP(K)$ denote the law of $(\rho_n,\tilde F_n)$.  Since $K$ is
compact, $\{\lambda_n\}$ is tight, and by Prohorov's theorem there is a
subsequence such that, after relabeling, $\lambda_{n}\Rightarrow\lambda$~\cite[Thm.~9.6.2]{resnick2003probability}.  Because $K_0$ is closed,
$\lambda(K_0)=1$.  The maps $(\rho,\tilde F)\mapsto\rho\big((\sfL_\tau[\tilde F]-1)^2\big)$ and
$(\rho,\tilde F)\mapsto\rho\big(\clE_\tau[\tilde F]\big)$ are linear in $\rho$ and quadratic in $F$,
hence continuous on $K$ (\Prop{prop:technical}(1)), so
\[
\int_K \rho\big((\sfL_\tau[\tilde F]-1)^2\big)\,\ud\lambda \;\geq\; \varepsilon,
\qquad
\int_K \rho\big(\clE_\tau[\tilde F]\big)\,\ud\lambda \;=\; 0 .
\]

Since $\rho(\clE_\tau[\tilde F])\geq0$, it vanishes $\lambda$-a.e.  The conditions of
\Lemma{lemma:Lasalle} are conditions on $\mathrm{HMM}(A,H)$ and $\tau$ alone,
so the lemma applies at $\lambda$-a.e.\ $(\rho,\tilde F)$.  Because
$\tilde F-1\in\clZfo{\tau}(\rho)$ and the dissipation vanishes,
$\sfL_\tau[\tilde F]=1$ on $\mathrm{supp}(\rho)$. 
Hence $\rho\big((\sfL_\tau[\tilde F]-1)^2\big)=0$ $\lambda$-a.e., contradicting the first
integral.

\section{Rates for variance decay}\label{sec:rates}
 Suppose that one of the three conditions of
\Lemma{lemma:Lasalle} holds, with $\tau$ fixed accordingly.  The lemma
states that
\[
  \rho\big(\clE[F]\big) = 0 \quad\Longrightarrow\quad
  \rho\big(|\sfL[F]|^2\big) = 0,
  \qquad \forall\, F\in\clZfo{\tau}(\rho).
\]
That is, zero energy forces zero variance.  The purpose of this section is to
make the relationship quantitative, by identifying the largest
constant $\beta(\rho)$ for which
\begin{equation*}
  \rho\big(\clE[F]\big) \;\ge\; \beta(\rho)\,
    \rho\big(|\sfL[F]|^2\big),
  \qquad \forall\, F\in\clZfo{\tau}(\rho) .
\end{equation*}
Under the conditions of \Lemma{lemma:Lasalle}, $\beta(\rho)>0$.  In
this section, we allow the possibility that the conditions do not
hold and $\beta(\rho)=0$. For this reason, unless
explicitly noted otherwise, we do not assume \Lemma{lemma:Lasalle} or
make any assumptions on the HMM.  Throughout, we fix $\tau$ and denote $\clE_\tau[F]$ and
$\sfL_\tau[F]$ simply as $\clE[F]$ and
$\sfL[F]$, respectively.

\subsection{Poincar\'e constant} The following definition is adapted
from~\cite[Defn.~4]{kim2024variance} stated therein for the
continuous-time model.  
\begin{definition}[Poincar\'e constant for
  $F\in\clZfo{\tau}(\rho)$]\label{def:poincare}
Let $\rho\in\clP(\bS)$.  Set
\[
  \beta(\rho) := \inf\big\{ \rho(\clE[F]) \;:\;
    F\in\clZfo{\tau}(\rho), \ \ \rho\big(|\sfL[F]|^2\big) = 1 \big\},
  \qquad \inf\emptyset := +\infty,
\]
and define the Poincar\'e constant
\[
  \kappa(\rho) := \frac{1}{\tau}\log\big(1+\beta(\rho)\big) .
\]
\end{definition}

The $\inf\emptyset := +\infty$ condition is included to handle the
case when $\rho$ lies in 
\[
  \sfN_\tau := \big\{\rho\in\clP(\bS) \;:\;
    \rho\big(|\sfL[F]|^2\big) = 0, \ \
    \forall\, F\in\clZfo{\tau}(\rho)\big\},
\]
so that $\beta(\rho)=\kappa(\rho)=+\infty$ on $\sfN_\tau$. 

\begin{remark}[Poincar\'e constant is a rate]
From the~\eqref{eq:Lyapunov_equation}, using the definition of the
Poincar\'e constant, 
\[
\dvar^\rho(Y_0(X_0)) \le e^{-\kappa(\rho)\, \tau}
\dvar^\rho(F(X_\tau)),\qquad \forall\,F\in \clZfo{\tau}(\rho),
\]
because $e^{-\tau \kappa(\rho)} = (1+\beta(\rho))^{-1}$.  Hence,
$\kappa(\rho)$ is the rate of variance decay per unit time.
For $\rho\in \sfN_\tau$, the identity holds with
$\kappa(\rho)=\infty$ where the convention $e^{-\infty}=0$ is adopted.  
\end{remark}

\begin{definition}[Stratum]
For $\bS'\subseteq\bS$, denote by
$\clP^\circ(\bS') := \{\rho\in\clP(\bS) : \supp(\rho)=\bS'\}$ the
\emph{stratum} of probability measures with support $\bS'$.  The strata partition
$\clP(\bS)$ into finitely many sets; $\clP^\circ(\bS)$ is the relative
interior of the simplex. 
\end{definition}

The strata with $|\bS'|=1$ are the singletons
$\sfN_0 := \{\delta_x : x\in\bS\}$, the vertices of $\clP(\bS)$.
Taking $\rho=\delta_x$, every $F\in\clZfo{\tau}(\delta_x)$ satisfies
\[
  0 = \rho\big(\sfL[F]\big) = \sfL[F](x)
  \quad\Longrightarrow\quad
  \rho\big(|\sfL[F]|^2\big) = |\sfL[F](x)|^2 = 0 ,
\]
so that $\sfN_0\subseteq\sfN_\tau$.  The following proposition shows
that $\sfN_\tau$ is a union of strata.

\begin{proposition}\label{prop:N-strata}
Suppose $\rho\in\sfN_\tau$ then $\clP^\circ(\supp(\rho)) \subset\sfN_\tau$.  Consequently, each stratum lies either entirely inside
or entirely outside $\sfN_\tau$, and $\sfN_\tau$ is a closed subset of
$\clP(\bS)$.  
\end{proposition}

\begin{proof}
Appendix~\ref{appdx:characterization_of_N}. 
\end{proof}

With the geometry of $\sfN_\tau$ in hand, the existence and regularity
of the Poincar\'e constant are described next.  Notably, parts
(1)-(3) do not require any assumptions on the HMM.  

\begin{lemma}\label{prop:beta-regularity}
The following hold:
\begin{enumerate}
\item For each $\rho\in\clP(\bS)\setminus\sfN_\tau$, the infimum
  defining $\beta(\rho)$ is attained.
\item The map $\rho\mapsto \kappa(\rho)$ is upper semicontinuous on $\clP(\bS)$.
\item For each $\bS'\subseteq\bS$, the restriction of $\kappa$ to the
  stratum $\clP^\circ(\bS')$ is continuous.
\item If one of the three conditions of \Lemma{lemma:Lasalle} holds,
  then $\kappa(\rho)>0$ for every $\rho\in\clP(\bS)$.
\end{enumerate}
\end{lemma}

\begin{proof}
Appendix~\Sec{appdx:prop:beta-regularity}.
\end{proof}

The lower semicontinuous envelope is denoted by,
\[
\kappa_*(\rho) := \liminf_{\rho'\to\rho} \kappa(\rho'),
\qquad \rho\in\clP(\bS),
\]
and the associated singular set is 
\[
\clD_* := \big\{\rho\in\clP(\bS) \;:\; \kappa_*(\rho) = 0\big\}.
\]
By construction, $\rho\mapsto \kappa_*(\rho)$ is lower semicontinuous,
$\kappa_*\leq\kappa$, and $\clD_*$ is a 
closed subset of $\clP(\bS)$.  From \Lemma{prop:beta-regularity}(3),
$\kappa_*$ and $\kappa$ can differ only at measures that are limits of
measures of strictly larger support.

\begin{remark}[PI for Markov chains]
It is instructive to compare Defn.~\ref{def:poincare} with the
Poincar\'e constant for Markov chains.  Suppose $\bar\mu$ is an invariant measure for
$A$. Then the Poincar\'e constant is
defined by,
\[
\bar\beta := \min\{\bar\mu(\Gamma f)\,:\,\;f\in\Re^d,\,\var^{\bar\mu}(f)=1\}.
\]
The variance decays geometrically as a function of time if and only if
$\bar\beta>0$~\cite[Prop.~4.1]{cattiaux2014long}.  In the classical
literature, which largely concerns reversible chains,
$\bar\beta$ is expressed in terms of the spectral gap; for a reversible chain,
$\bar\beta = 1-\lambda_2^2$ where $\lambda_2$ is the second largest eigenvalue
of $A$ in
modulus~\cite[Sec.~5.1.3]{stroock2005markov},~\cite[Ch.~12--13]{levin2017markov}.
\end{remark}

\subsection{Asymptotic behavior of the skeleton}

We now apply the theory for the asymptotic analysis as $T\to\infty$.
Throughout this section, Assumption~\ref{ass:equivalent-priors} is in
force.

Fix $\tau\geq1$ and consider the \emph{skeleton} $\{\pi_{n\tau}^\nu :
n\geq0\}$ of the filter~\cite[Ch.~5]{meyn2009markov}.  The associated sequence of rate constants is
$\{\kappa(\pi_{n\tau}^\nu) : n\geq 0\}$, taking values in
$[0,\infty]$, and we set 
\[
\clK_N := \sum_{n=0}^{N-1} \kappa\big(\pi_{n\tau}^\nu\big).
\]
The sequence is $\clZ$-adapted and, because $\kappa$ is upper
semicontinuous (\Lemma{prop:beta-regularity}\,(2)), hence Borel
measurable, each $\clK_N$ is a random variable.

\begin{proposition}[Variance decay]\label{prop:var_contractive}
	Consider the backward map~\eqref{eq:bmap} with $\gamma_T$
        defined according to Defn.~\ref{defn:likelihood_ratio}.  Then
	\begin{equation*}\label{eq:var_contractive}
		\dvar^\nu(y_0(X_0)) \leq \E^\nu \left( e^{-\tau \clK_N}\pi_T^\nu((\gamma_T-1)^2) \right)
		,\quad \forall \;T \ge 0
	\end{equation*}
	where $N={\lfloor T/\tau\rfloor}$. 
      \end{proposition}
      
\begin{proof}
	Appendix~\Sec{appdx:var_contractive}. 
      \end{proof}

Under Assumption~\ref{ass:equivalent-priors},
\[
\pi_T^\nu((\gamma_T-1)^2) < C, \quad \sP^\nu\text{-a.s.}\quad \forall
T. 
\]
Therefore, variance decay property reduces to the question about the
asymptotic properties of $\clK_N$. 
Because $\{\clK_N:N=1,2,\hdots\}$ is non-negative and monotone, define
\[
\clK_{\infty}(\omega) :=\lim_{N\to\infty} \uparrow \clK_N(\omega),\quad
\omega\in\Omega.
\]

\begin{theorem}[Dichotomy]\label{thm:dichotomy}
Suppose Assumption~\ref{ass:equivalent-priors} holds.  Then either
\begin{enumerate}
\item[(i)] $\sP^\nu\big([\clK_\infty = \infty]\big) = 1$, in which case the filter
is stable in $\chisq$-divergence; or
\item[(ii)] $\sP^\nu\big([\clK_\infty < \infty]\big) > 0$, in which case
\[
\mathrm{dist}\big(\pi_{n\tau}^\nu(\omega),\, \clD_*\big)
\;\xrightarrow{(n\to\infty)}\; 0, \qquad
\sP^\nu\text{-a.e. } \omega \in [\clK_\infty < \infty].
\]
\end{enumerate}
\end{theorem}
\begin{proof}
Suppose (i), $\sP^\nu([K_\infty=\infty])=1$.  Under
Assumption~\ref{ass:equivalent-priors}, the variance decay formula in
\Prop{prop:var_contractive} gives,
\[
\dvar^\nu(y_0(X_0)) \le C\,\E^\nu(e^{-\tau K_N}),
\]
with $N=\lfloor T/\tau\rfloor$.  Because $e^{-\tau K_N}\leq 1$, dominated convergence
gives $\dvar^\nu(y_0(X_0))\to0$.  Because $\nu(y_0)=1$, this forces $y_0(x)\to1$ for every
$x\in\supp(\nu)$, and the filter is stable by \Prop{prop:chisq_filter_stability}.

Suppose instead (ii), and fix $\omega\in[K_\infty<\infty]$.  Then
\[
\sum_n \kappa(\pi^\nu_{n\tau}(\omega))<\infty \quad \implies \quad 
\kappa(\pi^\nu_{n\tau}(\omega))\to0.
\]
Suppose $\text{dist}(\pi^\nu_{n\tau}(\omega),\clD_*)\not\to0$.  Then
there exists a
subsequence $\{\pi^\nu_{n_k\tau}(\omega):k=1,2,\ldots\}$ and an $\epsilon>0$ such that
\[
  \text{dist}(\pi^\nu_{n_k\tau}(\omega),\clD_*)>\epsilon.  
\]
By compactness of $\clP(\bS)$ we may assume $\pi^\nu_{n_k\tau}(\omega)\to\rho$ with
$\text{dist}(\rho,\clD_*)\ge\epsilon$, so that $\kappa_*(\rho)>0$.  But
$\kappa_*(\rho)\le\liminf_k \kappa(\pi^\nu_{n_k\tau}(\omega)) = 0$, a
contradiction.
\end{proof}

The two alternatives of the dichotomy inform the future work:
\begin{enumerate}
\item Converting the first alternative into a quantitative rate requires a
lower bound on $\clK_N/N$.  Such a
bound follows whenever the skeleton remains in a compact subset of a single
stratum (cf.~\cite{tamura2011hypoellipticity}), where $\kappa$ is continuous by \Lemma{prop:beta-regularity}\,(3) and
therefore attains a positive infimum.  An example of this is as follows:
\begin{itemize}
\item If $A$ is primitive (i.e., irreducible and aperiodic, so $\bS^{\text{\tiny RECC}}=\bS$), the support does not collapse.  Noting that
  the smallest entry of the emission matrix $H$ is positive by the
  non-degeneracy Assumption~\ref{ass:non-degeneracy},
\[
\pi_t^\nu(x') \;\geq\; H_{\min{}}^{\tau_0(A)}
\min_{x\in\bS} A^{\tau_0(A)}(x,x'), \qquad \forall\,x'\in\bS,\quad t\geq\tau_0(A),
\]
so the skeleton remains in a compact subset of $\clP^\circ(\bS)$ and the bound
of part~1 applies (cf.~\cite{atar1997lyapunov}).
\end{itemize}
\item Characterizing $\clD_*$
requires understanding the behavior of $\kappa(\rho)$ as the support of $\rho$
collapses to a strictly smaller subset of $\bS$.  An example of this
is as follows:
\begin{itemize}
\item If $A$ has transient states or more than one ergodic class, then under detectability the
filter concentrates on the correct recurrent class~\cite{kim2021detectable,baxendale2004asymptotic}, and the collapse is real.  This is
the case in which $\clD_*$ needs to be characterized, and in which the
compactness argument is unavailable.
\end{itemize}
\end{enumerate}

For a given $\mathrm{HMM}(A,H)$, both questions are amenable to numerical
treatment: $\kappa(\rho)$ is the solution of a finite-dimensional generalized
eigenvalue problem.

\section*{Acknowledgments}
It is a pleasure to thank Jin Won Kim, with whom the duality theory
for nonlinear filtering was developed, and whose work on the backward
map and the variance decay property in continuous time is the direct
antecedent of this paper.

The author used LLMs to assist with
reference checking, proofreading, and numerical verification of algebraic
identities.  The mathematics and the text are the author's own.

\appendix

\section{Proofs for \Sec{sec:hmm-prelim}}\label{appdx:prelim}

\begin{proof}[Proof of \Prop{prop:B_W_Gamma_R}]
The formulae for $\Gamma$ and $R$ are derived in the proof of Prop.~8
in \cite{chang2025dual}.  In this part, we show the bound under the
non-degeneracy Assumption~\ref{ass:non-degeneracy}.

The meaning of $R(x)$ is as
 follows: 
 \begin{equation*}
 u^\tp R(x) u = H(x,0) (-1^\tp u)^2 + \sum_{k=1}^m H(x,k) (u(k))^2 -
 \left( H(x,0) (-1^\tp u) + \sum_{k=1}^m H(x,k) u(k) \right)^2,
\end{equation*}
where $1^\tp u = \sum_{k=1}^m u(k)$.  Note that for each fixed
$x\in\bS$, $H(x,\cdot)\in \clP(\bO)$.  
Therefore, $u^\tp R(x) u$ is the variance of $\xi:=\begin{bmatrix} 
  (-\ones^\tp u) \\ u\end{bmatrix}\in \Re^{m+1}$, w.r.t. 
the probability vector $H(x,\cdot)$.
Therefore,
\[
u^\tp R(x) u = \min_{c\in \Re} \sum_{z\in\bO} H(x,z) ( \xi(z) -
c)^2 \geq H_{\min{}} \min_{c\in \Re} \sum_{z\in\bO}  ( \xi(z) -
c)^2 = H_{\min{}} \sum_{z\in\bO} \xi^2(z) \geq H_{\min{}} |u|^2.
\]
where the middle equality follows because $\sum_{z\in\bO} \xi(z)=0$.
Therefore,
\[
u^\tp R(x) u \geq H_{\min{}} |u|^2,\quad \forall\, u\in \Re^m,\quad
x\in\bS.  
\]
\end{proof}

\begin{proof}[Proof of \Prop{Prop:rank-is-constant}]
Let $F\in\clZf{T}$.  By the Doob--Dynkin lemma there is a unique
deterministic function $s:\bS\times\bO^T\to\Re$ such that
\[
  F(x) = s(x;Z_1,\ldots,Z_T), \qquad x\in\bS,
\]
and therefore $\clZf{T}\cong\Re^{d(m+1)^T}$, the coordinates being
indexed by pairs $(x,\zeta)\in\bS\times \bO^T$.

In these coordinates, using the notation in~\eqref{eq:r_rho_s_rho}, the constraint $\sigma_T^\rho(F)=0$ is expressed
as
\[
  \sum_{x\in\bS} \mathfrak{p}^\rho_T(\zeta,x)\, s(x;\zeta) = 0,
  \qquad \zeta\in\bO^T,
\]
which are a set of $|\bO^T| = (m+1)^T$ linear equations.

Fix $\zeta\in\bO^T$.  There are two cases to consider:
\begin{enumerate}
\item If $\mathfrak{q}^\rho_T(\zeta) >0$ then $\mathfrak{p}^\rho_T(\zeta,x)>0$ for at
  least one $x$ and the constraint is active.  
\item If $\mathfrak{q}^\rho_T(\zeta) =0$ then $\mathfrak{p}^\rho_T(\zeta,x)=0$ for
  every $x$ and the constraint is in-active.  
\end{enumerate}
The dimension count is obtained by noting that the active constraints are
linear independent because distinct $\zeta_1,\zeta_2\in \bO^T$ are disjoint coordinates.  
\end{proof}

\section{Proof of Lyapunov equation}\label{sec:prop:Lyapunov_proof}
\begin{proof}[Proof of~\Prop{Prop:Lyapunov}]
For each $t\in \mathbb{T}$,
\[
\E^\rho \big( M_t M_0) = \E^\rho \big( \E^\rho (M_t \mid \clF_0) M_0) =\E^\rho \big( \E^\rho (Y_t(X_t) \mid \clF_0)
Y_0(X_0)) = \E^\rho\big((Y_0(X_0))^2\big).
\]
Therefore, the quadratic variation of $M$ satisfies
\[
\E^\rho \big(\langle M\rangle_t\big) = \E^\rho \big( (M_t-M_0)^2\big) =
\E^\rho \big((Y_t(X_t))^2\big) - \E^\rho\big(Y_0(X_0))^2\big).
\]
So far, we have not used the fact that $F-c\ones \in \clZfo{T}(\rho)$.  For
such $F$, $\pi_T^\rho(F) = c$, $\sP^\rho$-a.s., and 
\[
\pi_t^\rho (Y_t) = \E^\rho(Y_t(X_t) \mid \clZ_t) = \E^\rho(F(X_T) \mid
\clZ_t ) =
\E^\rho(\E^\rho(F(X_T) \mid \clZ_T) \mid \clZ_t) = c,\quad \sP^\rho\text{-a.s.},\;t\in\mathbb{T}.
\]
Therefore, 
\begin{align*}
  \E^\rho \big((Y_t(X_t))^2\big) - \E^\rho\big(Y_0(X_0))^2\big)
& = \left( \E^\rho \big((Y_t(X_t))^2\big) -c^2 \right) - \left( \E^\rho\big(Y_0(X_0))^2\big)-c^2 \right)
  \\ &= \dvar^\rho\big(Y_t(X_t)\big) -
\dvar^\rho\big(Y_0(X_0)\big).
\end{align*}
Combined with the quadratic variation identity, upon 
rearranging, we obtain the formula given in the~\eqref{eq:Lyapunov_equation}.
\end{proof}

 \section{Existence-uniqueness theory for BS$\Delta$E}\label{sec:bsde_proofs}

 The proofs in this section rely on the following simple result of
 linear algebra.  
 \begin{proposition}[Prop.~17 in \cite{chang2025dual}]\label{prop:linear_algebra}
  Let $s:\bO\to \Re$ with $s^z$ denoting the value of the function $s$ at
  $z\in\bO$.  Then there exists unique $(\bar{s},\tilde{s}) \in \Re
  \times \Re^m$ such that the following decomposition holds:
  \begin{equation*}
  s^z = \bar{s} +\tilde{s}^\tp e(z) ,\quad z\in\bO.
\end{equation*}
Explicitly, 
\begin{equation*}
\bar{s}:= \frac{1}{m+1} \sum_{z\in\bO} s^z,\quad \text{and} \quad 
\tilde{s}^k = (s^k - \bar{s}),\quad
k=1,2,\hdots,m.
\end{equation*}
\end{proposition}
\begin{proof}[Proof of \Prop{prop:linear_algebra}]
For $z=1,\ldots,m$, $\bar s+\tilde s^\tp e(z) = \bar s+(s^z-\bar s)=s^z$
directly from the definition of $\tilde s$. For $z=0$, since
$e(0)=-\sum_{k=1}^me(k)$,
\[
\bar s+\tilde s^\tp e(0) = \bar s - \sum_{k=1}^m\tilde s^k = \bar s - \sum_{k=1}^m(s^k-\bar s) = (m+1)\bar s - \sum_{k=1}^m s^k = s^0,
\]
using $\bar s=\frac1{m+1}\sum_{z\in\bO}s^z$ in the last step. This
proves existence. For uniqueness, suppose $\bar s+\tilde s^\tp e(z)=0$
for all $z\in\bO$. Then $\bar s+\tilde s^k=0$ for $k=1,\ldots,m$, and
$\bar s-\sum_{k=1}^m\tilde s^k=0$. Summing the first $m$ equations,
$m\bar s+\sum_i\tilde s^k= (m+1)\bar s = 0$, so $\bar s=0$, and then
each $\tilde s^k=-\bar s=0$.
\end{proof}

The following corollary shows that every $\clZ_T$-measurable random
variable admits a representation of exactly this form---a fact used
repeatedly throughout the paper, from well-posedness of the
BS$\Delta$E to the characterization of the range space of its solution
operator.

\begin{corollary}\label{prop:existence_uniqueness_ST_U} 
Let $S_T\in \clZ_T$.  Then there exists a
unique $U\in \clU_T$ and a constant $c\in \Re$ such that
\[
S_T = c + \sum_{t=0}^{T-1} U_t^\tp e(Z_{t+1}).
\]
\end{corollary}
\begin{proof}
From Doob-Dynkin lemma, there is a unique deterministic function
$s:\bO^T \to \Re$ such that 
\begin{equation*}
S_T = s(Z_1,\hdots,Z_{T-1},Z_T).
\end{equation*}
We denote
\[
S^z:=S_T
\big|_{[Z_{T}=z]} = s(Z_1,\hdots,Z_{T-1},z),\quad z\in\bO. 
\]
From~\Prop{prop:linear_algebra},
\begin{equation*}
S_T = S^{Z_T} = S_{T-1} + (U_{T-1})^\tp e(Z_T),
\end{equation*}
where
\begin{equation*}
  S_{T-1} = \frac{1}{m+1} \sum_{z\in \bO} S^z,\qquad 
 U_{T-1}(k):=  (S^k - S_{T-1}),\quad
k=1,2,\hdots,m.
\end{equation*}
Uniqueness is from the uniqueness of the decomposition in \Prop{prop:linear_algebra}. The proof is
completed through induction by repeating the procedure for $S_{T-1}\in\clZ_{T-1}$. 
\end{proof}

\begin{remark}[Bijection]\label{rem:ST-bijection}
\Corollary{prop:existence_uniqueness_ST_U} shows that the map
\[
(U,c) \in \clU_T\times\Re \;\longmapsto\; S_T := c + \sum_{t=0}^{T-1}U_t^\tp e(Z_{t+1}) \;\in\; \clZ_T
\]
is a bijection. This fact underpins the duality between
controllability and observability described in \Prop{prop:duality}.
\end{remark}

\begin{example}[m=1] 
  Set
\[
S_T^{+} = s(Z_1,\hdots,Z_{T-1},1),\qquad S_T^{-} =
s(Z_1,\hdots,Z_{T-1},0).
\]
Then $S_T^{+}, S_T^{-}\in \clZ_{T-1}$ and
\begin{equation*}
S_T = \underbrace{\half\left( S_T^{+} + S_{T}^-\right)}_{=S_{T-1}} +
\underbrace{\half\left( S_T^{+} - S_{T}^-\right)}_{= U_{T-1}} e(Z_T),
\end{equation*}
where recall $e(1)=1$ and $e(0)=-1$ in the binary-valued case.
\end{example}

\subsection{Proof of existence and uniqueness of the solution of the BS$\Delta$E}\label{sec:prop:existence_BSDE}
\begin{proof}[Proof of \Prop{prop:existence_BSDE}] 
 Fix $x\in \bS$.  Then at time $t=T$, the BS$\Delta$E is
  \begin{equation*}
  (AF)(x) = Y_{T-1}(x) - h^\tp(x) (U_{T-1} + V_{T-1}(x)) + (V_{T-1}(x))^\tp
  e(Z_T),\quad x\in \bS.
  \end{equation*}
  Set
  \[
    F^z(x)=F(x)\big|_{[Z_{T}=z]},\qquad z\in\bO,\quad x\in\bS.
  \]
From~\Prop{prop:linear_algebra}, there exists unique
$\bar{F}_{T-1}(x), \tilde{F}_{T-1}(x) \in \clZ_{T-1}$ such that
\begin{align*}
F(x) &= \bar{F}_{T-1}(x) + (\tilde{F}_{T-1}(x))^\tp
  e(Z_T),\quad x\in \bS,\\
  \therefore,\quad (AF)(x) &= (A\bar{F}_{T-1})(x) + ((A\tilde{F}_{T-1})(x))^\tp
  e(Z_T),\quad x\in \bS.
\end{align*}
Set
\begin{align*}
  V_{T-1}(x) &=  (A\tilde{F}_{T-1})(x),\quad x\in\bS,\\
  Y_{T-1}(x) &= (A\bar{F}_{T-1})(x)  + h^\tp(x) (U_{T-1} + V_{T-1}(x)) ,\quad x\in\bS.
\end{align*}
Using the formulae for $\bar{F}_{T-1},\tilde{F}_{T-1}$, upon
simplifying,
\[
Y_{T-1}(x) = \sum_{z\in\bO} (\text{diag}(H(\cdot,z)) A F^z)(x)  + h^\tp(x) U_{T-1},\quad x\in\bS.
\]
Because $Y_{T-1}\in\clZ_{T-1}$, the proof is completed through induction.

For each $t=0,\ldots,T-1$ and $z\in\bO$, write
\[
Y_{t+1}^z = Y_{t+1}\big|_{[Z_{t+1}=z]}, \quad \bar{Y}_{t+1} = \frac{1}{m+1} \sum_{z\in\bO} Y_{t+1}^z,\quad \text{and} \quad 
\tilde{Y}_{t+1}^k = ( {Y}_{t+1}^k - \bar{Y}_{t+1}),\quad k=1,2,\hdots,m.
\]
Then
\begin{subequations}\label{eq:Y_t_soln_formula}
\begin{align}
  Y_t(x) &= (A \bar{Y}_{t+1})(x) + \sum_{k=1}^m
           \big((\mathrm{diag}(h^k)A)\tilde{Y}_{t+1}^k\big)(x) +
           h^\tp(x)\,U_t,\quad x\in\bS\label{eq:Y_t_soln_formula_a}\\
&= \sum_{z\in\bO} \big(\mathrm{diag}(H(\cdot,z))\,A\,Y_{t+1}^z\big)(x) + h^\tp(x)\,U_t, \qquad x\in\bS.
\label{eq:Y_t_soln_formula_b}
\end{align}
\end{subequations}
\end{proof}

\begin{example}[m=1]
  For each $x\in\bS$, set $h^{+}(x) = H(x,1)$, $h^{-}(x) = H(x,0)$,
  and 
\begin{align*}
F_T^{+}(x) = s(Z_1,\hdots,Z_{T-1},1;x), \quad F_T^{-}(x) =
  s(Z_1,\hdots,Z_{T-1},0;x).
\end{align*}
Then $F_T^{+}(x), F_T^{-}(x)\in \clZ_{T-1}$ and
\begin{align*}
  V_{T-1}(x) &=  \half\left(AF_T^{+} (x)- AF_{T}^-(x) \right),\quad x\in\bS,\\
  Y_{T-1}(x) &= \half\left(AF_T^{+}(x) + AF_{T}^-(x)\right) + h(x)
               (U_{T-1} + V_{T-1}(x))\\
               & = h^+(x) (AF_T^{+})(x) + h^-(x) (AF_T^{-})(x) + h(x)
                 U_{T-1},\quad x\in\bS. 
\end{align*}
\end{example}

\subsection{Analytical and probabilistic characterization of the 
  BS$\Delta$E solution}\label{sec:prop:backward-map-and-bsde}
\begin{proof}[
Proof of \Prop{prop:backward-map-and-bsde}] The existence and
uniqueness of $(Y,V)$ is the special case $U=0$ of
\Prop{prop:existence_BSDE}, and the analytical characterization is
the formula~\eqref{eq:Y_t_soln_formula_b}  with $U_t=0$.  For the
probabilistic characterization, specializing the martingale increment identity (formula~\eqref{eq:martingale_identity_appdx} in \Sec{sec:martingale_proofs}) derived in the proof of
\Prop{prop:MU-martingale} to $U=0$ gives
\[
Y_{t+1}(X_{t+1}) = Y_t(X_t) + N_{t+1}, \qquad t\in\mathbb{T}^-,
\]
where $N=\{N_t:1\le t\le T\}$ is a $(\clF,\sP^\rho)$-martingale
increment process (for any $\rho\in\clP(\bS)$). Telescoping from $t$
to $T-1$ and using $Y_T=F$,
\[
\gamma_T(X_T) = Y_t(X_t) + \sum_{s=t}^{T-1}N_{s+1}.
\]
Taking $\E^\rho(\cdot\mid\clZ_t\vee[X_t=x])$ therefore yields
\[
  Y_t(x) = \E^\rho\big(F(X_T)\mid\clZ_t\vee[X_t=x]\big),
  \qquad x\in\bS, \quad \sP^\rho\text{-a.s.}
\]
\end{proof}

\section{Proofs concerning martingales $M$ and $M^U$}
\label{sec:martingale_proofs}

\begin{proof}[Proof of \Prop{prop:MU-martingale}]

\newP{Martingale property} For the HMM, we have defined two martingale increment processes:
\begin{align*}
  B_{t+1}(f) &= f(X_{t+1}) - (Af)(X_t),\quad t\in\mathbb{T}^-,\\
  W_{t+1} &= e(Z_{t+1})- h(X_t) ,\quad t\in\mathbb{T}^-.
\end{align*}
Taking $f=Y_{t+1}$,
\begin{equation*}
 B_{t+1}(Y_{t+1}) = Y_{t+1}(X_{t+1}) - (A Y_{t+1})(X_t) ,\quad t\in\mathbb{T}^-.
\end{equation*}
Since $Y$ solves BS$\Delta$E~\eqref{eq:dual_BSDE},
\begin{equation*}
(AY_{t+1})(x) = Y_{t}(x) - h^\tp(x) (U_t + V_t(x)) + V_t^\tp(x) e(Z_{t+1}),\quad
\forall x\in \bS ,\quad t\in\mathbb{T}^-,
\end{equation*}
and therefore at $x=X_t$,
\begin{equation*}
(AY_{t+1})(X_t) = Y_{t}(X_t) -  h^\tp(X_t) (U_t + V_t(X_t)) + V_t^\tp(X_t) e(Z_{t+1}) ,\quad t\in\mathbb{T}^-.
\end{equation*}
Because $W_{t+1} = e(Z_{t+1}) - h(X_t)$, we have
\begin{align*}
  h^\tp(X_t) U_t &= U_t^\tp e(Z_{t+1}) - U_t^\tp W_{t+1},\quad t\in\mathbb{T}^-, \\
  h^\tp(X_t) V_t(X_t) &= V_t^\tp(X_t) e(Z_{t+1}) - V_t^\tp(X_t) W_{t+1} ,\quad t\in\mathbb{T}^-,
\end{align*}
and thus,
\begin{equation*}
(AY_{t+1})(X_t) = Y_{t}(X_t) - U_t^\tp e(Z_{t+1}) + (U_t + V_t(X_t))^\tp W_{t+1},\quad t\in\mathbb{T}^-.
\end{equation*}
All together,
\begin{align*}
  Y_{t+1}(X_{t+1}) & = (A Y_{t+1})(X_t) + B_{t+1}(Y_{t+1}) ,\quad t\in\mathbb{T}^- \\
  &= Y_{t}(X_t) - U_t^\tp e(Z_{t+1}) + (U_t + V_t(X_t))^\tp W_{t+1} + B_{t+1}(Y_{t+1}) ,\quad t\in\mathbb{T}^-.
\end{align*}
Denote
\begin{equation*}
N_{t+1} := (U_t + V_t(X_t))^\tp W_{t+1} + B_{t+1}(Y_{t+1}),\quad t\in\mathbb{T}^-.
\end{equation*}
to write this more succinctly as
\begin{equation}\label{eq:martingale_identity_appdx}
 Y_{t+1}(X_{t+1}) + U_t^\tp e(Z_{t+1}) = Y_{t}(X_t) + N_{t+1}, \quad t\in\mathbb{T}^-.
\end{equation}
Note $N_{t+1}$ is a martingale increment: since $U_t,V_t(X_t)\in\clF_t$,
\begin{align*}
  \E^\rho\big((U_t+V_t(X_t))^\tp W_{t+1}\mid\clF_t\big) &= (U_t+V_t(X_t))^\tp\,\E^\rho(W_{t+1}\mid\clF_t) = 0 , \quad t\in\mathbb{T}^-,\\
  \E^\rho(B_{t+1}(Y_{t+1})\mid\clF_t)&=\E^\rho(\E^\rho(B_{t+1}(Y_{t+1})\mid\clG_{t})\mid\clF_t)=0, \quad t\in\mathbb{T}^-,
\end{align*}
so $\E^\rho(N_{t+1}\mid\clF_t)=0$ for $t\in\mathbb{T}^-$. 
This shows $M^U$ is a $(\clF,\sP^\rho)$-martingale:
\begin{align*}
\E^\rho (M_{t+1}^U \mid \clF_t) &= \E^\rho \big( Y_{t+1}(X_{t+1}) + U_t^\tp
e(Z_{t+1}) + \sum_{s=0}^{t-1} U_s^\tp
  e(Z_{s+1}) \mid \clF_t)\big)\\
  &= Y_{t}(X_t) + \sum_{s=0}^{t-1} U_s^\tp
  e(Z_{s+1})  + \E^\rho (N_{t+1} \mid \clF_t) = M_t^U,\quad t\in\mathbb{T}^-.
\end{align*}

\newP{Quadratic variation} Squaring and
conditioning on $\clF_t$:
\begin{align*}
\E^\rho\big(N_{t+1}^2\mid\clF_t\big) = \E^\rho\big(B_{t+1}(Y_{t+1})^2\mid\clF_t\big)
  & + \E^\rho\big(((U_t+V_t(X_t))^\tp
                                       W_{t+1})^2\mid\clF_t\big) 
  \\
  &+ 2\,\E^\rho\big((U_t+V_t(X_t))^\tp W_{t+1}\,B_{t+1}(Y_{t+1})\mid\clF_t\big).
\end{align*}
The calculations for the three terms are as follows:
\begin{align*}
\E^\rho\big(B_{t+1}(Y_{t+1})^2\mid\clF_t\big) &=
  \E^\rho\big(\E^\rho\big(B_{t+1}(Y_{t+1})^2\mid\clG_{t}\big)\mid
                                                \clF_t) \\ & = \E^\rho\big((\Gamma
                                                    Y_{t+1})(X_t) \mid
  \clF_t \big),\\
\E^\rho\big(((U_t+V_t(X_t))^\tp W_{t+1})^2\mid\clF_t\big)  &= (U_t
                                                             +
                                                             V_t(X_t))^\tp
                                                             \E^\rho(W_{t+1}W_{t+1}^\tp|\clF_t) (U_t +  V_t(X_t)) \\
                                                   & =       (U_t
                                                     +
                                                     V_t(X_t))^\tp
                                                     R(X_t)  (U_t
                                                     + V_t(X_t)),
  \\
  \E^\rho\big((U_t+V_t(X_t))^\tp
  W_{t+1}\,B_{t+1}(Y_{t+1})\mid\clF_t\big) &=
\E^\rho\big(\E^\rho ((U_t+V_t(X_t))^\tp
                                             W_{t+1}B_{t+1}(Y_{t+1})\mid\clG_{t} )\mid \clF_t \big) \\
  & = \E^\rho\big(  (U_t+V_t(X_t))^\tp
    W_{t+1}\,\E^\rho\big(B_{t+1}(Y_{t+1})\mid\clG_{t}\big)  \mid
    \clF_t) \\& = 0.
\end{align*}
For the cross term, we used that $U_t+V_t(X_t)$ and $W_{t+1}=e(Z_{t+1})-h(X_t)$
are $\clG_{t}$-measurable. 
Combining,
\begin{align*}
\E^\rho\big(N_{t+1}^2\mid\clF_t\big) &= \E^\rho\big((\Gamma
                                                    Y_{t+1})(X_t) \mid
  \clF_t \big) + (U_t+V_t(X_t))^\tp R(X_t)(U_t+V_t(X_t)) \\ & =
  \E^\rho\big( \ell(Y_{t+1},V_t,U_t;X_t) \mid \clF_t \big),
\end{align*}
which gives the formula for quadratic variation:
\[
\langle M^U \rangle_t = \sum_{s=0}^{t-1} \E^\rho\big(\ell(Y_{s+1},V_{s},U_{s};X_{s}) \mid \clF_s\big),
\quad t\in\mathbb{T},\quad \sP^\rho\text{-a.s.}
\]

\newP{Duality principle} We have
\[
M_T^U = F(X_T)+\sum_{t=0}^{T-1}U_t^\tp e(Z_{t+1}), \qquad M_0^U=Y_0(X_0).
\]
Since $S_T=\rho(Y_0)-\sum_{t=0}^{T-1}U_t^\tp e(Z_{t+1})$,
\[
F(X_T)-S_T = \big(Y_0(X_0)-\rho(Y_0)\big) + (M_T^U-M_0^U).
\]
Squaring and taking expectation, the cross term vanishes because
$Y_0(X_0)-\rho(Y_0)$ is $\clF_0$-measurable and
$\E^\rho(M_T^U-M_0^U\mid\clF_0)=0$ by the martingale property.
Hence, 
\begin{align*}
\E^\rho\big(|F(X_T)-S_T|^2\big) &=
                                  \E^\rho\big((Y_0(X_0)-\rho(Y_0))^2\big)
                                  + \E^\rho\big((M_T^U-M_0^U)^2\big)
  \\& = \var^\rho(Y_0(X_0)) + \E^\rho(\langle M^U\rangle_T) = \sJ_T^\rho(U;F).
\end{align*}
\end{proof}

\subsection{Formula for the quadratic variation}\label{sec:martingale_proofs_M0}
\begin{proof}[Proof of \Prop{prop:M0-martingale}]
This is the special case of \Prop{prop:MU-martingale} with 
$U\equiv0$. Setting $U=0$ gives $M_t=Y_t(X_t)$
directly, and the running cost reduces to
$\ell(y,v,0;x)=(\Gamma y)(x)+v^\tp(x)R(x)v(x)$, recovering the stated
quadratic variation.

\end{proof}

\subsection{Technical estimates for solution of the BS$\Delta$E}\label{sec:prop:technical}

\begin{proof}[Proof of \Prop{prop:technical}]
  \emph{(Part 1, continuity).}
Let
\begin{align*}
(Y,V) &= \text{BS}\Delta\text{E}_\tau[F],\qquad M_t = Y_t(X_t),\quad 0\leq
        t\leq \tau,\\
(\tilde Y,\tilde V) &= \text{BS}\Delta\text{E}_\tau[\tilde F],\qquad \tilde
                      M_t = \tilde Y_t(X_t),\quad 0\leq
        t\leq \tau,  
\end{align*}
For each $x\in \bS$, we show that
\begin{align*}
|Y_0(x) - \tilde Y_0(x)|^2 &\leq \E^x \big( |F(X_\tau) - \tilde
  F(X_\tau)|^2\big) \\
\left|\E^x (\langle M \rangle_\tau) - \E^x (\langle \tilde M \rangle_\tau)
  \right|^2 &\leq \delta(F,\tilde F) \E^x \big( |F(X_\tau) - \tilde
  F(X_\tau)|^2\big),
\end{align*}
where
\[
  \delta (F,\tilde F)= 2 \left( \E^x \big( |F (X_\tau) + \tilde F(X_\tau) |^2 \big) +
    |Y_0(x) + \tilde Y_0(x)|^2\right) \leq 16\,C^2,
\]
where $C:=\max\{\|F\|_\infty,\|\tilde F\|_\infty\}$.

The first inequality is direct
using Jensen's inequality.  To show the second inequality, 
  \begin{align*}
    \E^x  (\langle M \rangle_\tau) = \E^x (F(X_\tau)^2) -
    Y_0^2(x),\quad 
    \E^x  (\langle \tilde M \rangle_\tau) = \E^x (\tilde F(X_\tau)^2)
    - \tilde Y_0^2(x)
  \end{align*}
  Therefore,
  \[
|\E^x (\langle M \rangle_\tau) - \E^x (\langle \tilde M \rangle_\tau)
  | \leq \E^x \big( |F^2(X_\tau) - \tilde
  F^2(X_\tau)|\big) + |Y_0^2(x) - \tilde Y_0^2(x)|.
\]
The estimate follows because
\begin{align*}
\left(\E^x \big( |F^2(X_\tau) - \tilde
  F^2(X_\tau)|\big) \right)^2 & \leq \E^x \big( |F(X_\tau) - \tilde
  F(X_\tau)|^2\big) \E^x \big( |F(X_\tau) + \tilde
  F(X_\tau)|^2\big),\\
  |Y_0^2(x) - \tilde Y_0^2(x)|^2 & = |Y_0(x) - \tilde Y_0(x)|^2
                                   |Y_0(x) + \tilde Y_0(x)|^2  \\ & \leq
                                                                 \E^x
                                                                 \big(
                                                                 |F(X_\tau)
                                                                -
                                                                 \tilde
                                                                 F(X_\tau)|^2\big)  |Y_0(x) + \tilde Y_0(x)|^2 .
\end{align*}

We now pass from $x$ to $\rho$.  
Let $\rho,\tilde\rho \in \clP(\bS)$ and consider 
\begin{align*}
\E^\rho (\langle M \rangle_\tau) = \sum_{x\in\bS} \rho(x) \E^x
(\langle M \rangle_\tau),\quad \dvar^\rho(Y_0(X_0)) = \sum_{x\in\bS}
  \rho(x) Y_0^2(x) - \big( \sum_{x\in\bS} \rho(x) Y_0(x) \big)^2,\\
  \E^{\tilde \rho} (\langle \tilde M \rangle_\tau) = \sum_{x\in\bS}
  \tilde \rho(x) \E^x
(\langle \tilde M \rangle_\tau),\quad \dvar^{\tilde \rho}(\tilde Y_0(X_0)) = \sum_{x\in\bS}
\tilde \rho(x) {\tilde Y}_0^2(x) - \big( \sum_{x\in\bS} \tilde \rho(x) {\tilde
  Y}_0(x) \big)^2.
\end{align*}
Therefore, 
\begin{align*}
\Big|\E^\rho\big(\langle M\rangle_\tau\big)
   - \E^{\tilde\rho}\big(\langle\tilde M\rangle_\tau\big)\Big|
\;\leq\; 2\,C^2\|\rho-\tilde\rho\|_{\text{TV}} + 4C\,\|F-\tilde F\|_\infty ,\\
  \big|\dvar^\rho(Y_0(X_0)) - \dvar^{\tilde\rho}(\tilde Y_0(X_0))\big|
\;\leq\;4 C^2\|\rho-\tilde\rho\|_{\text{TV}}  + 4C\,\|F-\tilde F\|_\infty .
\end{align*}
Both maps are therefore jointly continuous on $\clZf{\tau}\times\clP(\bS)$, and
Lipschitz on bounded sets.

\emph{(Part 2, boundedness.)}  
By the analytical characterization in
\Prop{prop:backward-map-and-bsde},
\[
Y_t(x) = \sum_{z\in\bO} H(x,z)\,\big(AY_{t+1}^z\big)(x),
\qquad Y_{t+1}^z(x) := Y_{t+1}(x)\big|_{[Z_{t+1}=z]} .
\]
Since $H$ and $A$ are row-stochastic, $Y_t(x)$ is a convex combination of the
values $\{Y_{t+1}^z(y): y\in\bS,\,z\in\bO\}$.  Hence, $
|Y_t(x)| \leq \|Y_{t+1}\|_\infty 
$
and by induction, because $Y_T=F$,
\[
  |Y_t(x)| \leq \|F\|_\infty,\qquad x\in\bS, \quad t\in \mathbb{T}.  
\]
So,
\[
  \dvar^\rho(Y_t(X_t))\leq\E^\rho(|Y_t(X_t)|^2)\leq\|F\|^2_\infty,\quad t\in \mathbb{T}.  
\]
Next, 
\[
\E^x\big(\langle M\rangle_t\big) = \E^x\big(M_t^2\big) - Y_0^2(x)
\;\leq\; \|F\|_\infty^2 , \quad t\in \mathbb{T},
\]
and therefore,
\[
  \E^\rho(\langle M\rangle_t)= \sum_{x\in\bS}
  \rho(x) \E^x\big(\langle M\rangle_t\big)  \leq\|F\|_\infty^2, \quad t\in \mathbb{T}.  
\]

\emph{(Part 3, restriction.)}  The BS$\Delta$E~\eqref{eq:dual_BSDE} couples only
consecutive times, so the equations indexed by $t<\tau$ involve
$\{(Y_t,V_t):0\leq t\leq\tau\}$ alone.  Hence this restricted pair solves the
BS$\Delta$E over $[0,\tau]$ with terminal condition $Y_\tau$, and by uniqueness
(\Prop{prop:existence_BSDE}) it is the solution. 

\emph{(Part 4, centering.)}  Let $F-c\ones \in \clZfo{T}(\rho)$.  Then $\pi_T^\rho(F) = c$, $\sP^\rho$-a.s., and 
\[
\pi_t^\rho (Y_t) = \E^\rho(Y_t(X_t) \mid \clZ_t) = \E^\rho(F(X_T) \mid
\clZ_t ) =
\E^\rho(\E^\rho(F(X_T) \mid \clZ_T) \mid \clZ_t) = c,\quad \sP^\rho\text{-a.s.}
\]
We have two cases:
\begin{enumerate}
\item If $\sigma_t^\rho(\ones)>0$ then
  \[
\pi_t^\rho (Y_t) = c \quad \implies \quad \sigma_t^\rho (Y_t) = c \sigma_t^\rho(\ones).
\]
\item  If $\sigma_t^\rho(\ones)=0$ then $\sigma_t^\rho(x)=0$, for all $x\in
  \bS$.
\end{enumerate}
In either case, $\sigma_t^\rho(Y_t-c\ones)=0$.  That is, $Y_t-c\ones\in\clZfo{t}(\rho)$.
\end{proof}

\section{Proof of Zero dissipation set under
  Assumption~\ref{ass:non-degeneracy}}
\label{appdx:prop:part-i-implication}

\begin{proof}[Proof of \Prop{prop:part-i-implication}]
As part of the proof of \Lemma{lemma:Lasalle}(i), we
showed that 
\[
  Y_\tau(x') = Y_0(x), \qquad \forall\, x' \in \mathcal
  R_\tau(x),\quad x \in \supp(\rho).
\]
That is, for each fixed $x \in \supp(\rho)$, 
restricted to the points $x'\in \mathcal R_\tau(x) \subset \bS$, $Y_\tau(\cdot)$ is a
deterministic function that moreover takes a constant value
($=Y_0(x)$).  Elsewhere on $\bS$ it need not be.  To
obtain a deterministic function on all of $\bS$, set
\[
g(x') := \E^\rho(Y_\tau(x')),\quad x'\in \bS.
\]
Then 
\[
  g(x') =\E^\rho(Y_\tau(x'))= \E^\rho(Y_0(x))=Y_0(x), \qquad \forall\, x'\in\mathcal R_\tau(x),
  \quad x\in\supp(\rho).
\]
Since $\tau\ge\tau_1(A)$, we claim
\[
  \textbf{(claim)}\qquad 
  \mathcal R_{\tau+kp}(x)\cap\bS^{\text{\tiny RECC}}
  = \mathcal R_{\tau}(x)\cap\bS^{\text{\tiny RECC}}=:\mathbb{B}_\tau^x,
  \qquad \forall\, k\ge 0,\quad  x\in\bS .
\]
Then
\[
g(X_{\tau+kp}) \ones_{[X_{\tau+kp} \in
  \mathbb{B}_\tau^x]}  = Y_0(x) \ones_{[X_{\tau+kp} \in
  \mathbb{B}_\tau^x]}  ,\quad \forall\, k\ge 0, \quad \sP^x\text{-a.s.}, \quad x \in\supp(\rho),
\]
and thus,
\[
\E^x (g(X_{\tau+kp})) = Y_0(x) \sP^x([X_{\tau+kp} \in
  \mathbb{B}_\tau^x])+  \E^x (g(X_{\tau+kp})
\ones_{[X_{\tau+kp} \notin \mathbb{B}_\tau^x]}), \quad x \in\supp(\rho).
\]
Average over $k$,
\[
\frac{1}{n} \sum_{k=0}^{n-1} \E^x (g(X_{\tau+kp})) = Y_0(x) \frac{1}{n} \sum_{k=0}^{n-1} \sP^x([X_{\tau+kp} \in
  \mathbb{B}_\tau^x])+
\frac{1}{n} \sum_{k=0}^{n-1}  \E^x (g(X_{\tau+kp}) \ones_{[X_{\tau+kp} \notin \mathbb{B}_\tau^x]})
\]
Then because $\sP^x ([X_{\tau+kp} \notin \mathbb{B}_\tau^x]) = \sP^x
([X_{\tau+kp} \notin \bS^{\text{\tiny RECC}}]) \to 0 $, 
\[
Y_0(x) = \lim_{n\to\infty} \frac{1}{n} \sum_{k=0}^{n-1} \E^x
(g(X_{\tau+kp})) , \quad x \in\supp(\rho).
\]

Since $\E^x(g(X_{\tau+kp})) = (A^{\tau+kp}g)(x)$, the limit above is
\[
\tilde f(x) := (\Pi A^\tau g)(x),\quad x\in\supp(\rho),\quad \text{where} \quad \Pi :=
\lim_{n\to\infty}\frac{1}{n}\sum_{k=0}^{n-1} A^{kp}
\]
is the Ces\`aro limit of the powers of the stochastic matrix $A^p$;
the limit exists by the mean ergodic theorem
(see~\cite[Sec.~3.2.4]{stroock2005markov}) and is the spectral
projection onto $\ker(A^p - I) = \clE^{\mathrm{PER}}(A)$. 
Hence $\tilde f\in\mathrm
E^{\mathrm{PER}}(A)$ and $\tilde f (x)= Y_0(x)$ for
$x\in\supp(\rho)$.

It remains to prove the claim. Denote $\mathbb{B}_t^x := \mathcal
R_t(x)\cap\bS^{\text{\tiny RECC}}$.  By definition,
\[
  \tau_1(A) \;=\; \min\big\{\, t \ge 1 \;:\;
    \mathbb{B}_t^x 
    = \mathbb{B}_{t+p}^x  ,
    \ \ \forall\, x \in \bS \,\big\}.  
\]
Now,
\[
  \mathbb{B}_{t+1}^x = \bigcup_{y\in\mathcal R_1(x)}
  \mathbb{B}_{t}^y.  
\]
So, the property at time $t=\tau_1(A)$, which holds for every
$x\in\bS$, is inherited at $t=\tau_1(A)+1$, and
hence by induction at every $t\ge\tau_1(A)$.  Applying it successively at
$t=\tau,\tau+p,\tau+2p,\ldots$ proves the claim.  
\end{proof}

 \section{Proofs concerning controllability and duality properties}
 \label{sec:proof:controllability}

\subsection{Proof of \Prop{prop:CT-recursion} and
  Corollary~\ref{cor:controllable_space}}
\label{sec:proof:controllability_main}
\begin{proof}[Proof of \Prop{prop:CT-recursion}]
By definition $\clC_0 = \mathrm{span}\{\ones\}$.  For $T=1$, setting
$Y_1 = c\ones$, the solution formula~\eqref{eq:Y_t_soln_formula_b} gives,
\[
Y_0 = c \ones + h^\tp U_0 \implies \clC_1 = \mathrm{span}\{\ones, h^1, h^2,\hdots, h^m\}.
\]

\emph{(Part (2), Monotonicity)}. Let $T>1$ and $(Y,V) =
\mathrm{BS}\Delta\mathrm{E}_T[c\ones,U]$ (horizon $T$). Choosing $U_{T-1}=0$
forces $V_{T-1}=0$ and $Y_{T-1}=c\ones$. Set
\begin{align*}
  \bar{Y}&:=(Y_0,Y_1,\hdots,Y_{T-2}, Y_{T-1}), \quad \text{recall } Y_{T-1} = c\ones,\\
  \bar{V}&:=(V_0,V_1,\hdots,V_{T-2}),\\
  \bar{U}&:=(U_0,U_1,\hdots,U_{T-2}).
\end{align*}
Then $\bar{U} \in \clU_{T-1}$ and $(\bar{Y},\bar{V}) =
\mathrm{BS}\Delta\mathrm{E}_{T-1}[c\ones,\bar{U}]$ (horizon $T-1$). Since this holds
for every $\bar U\in\clU_{T-1}$, it shows that
$\clC_{T-1}\subset\clC_T$: given $f=\sfL_{T-1}[c\ones,\bar U]$, the
control $U=(\bar U,0)\in\clU_T$ satisfies $\sfL_T [c\ones,U]=f$, so
$f\in\clC_T$. By the explicit form of $\clC_1$ established above,
\[
\ones, h^1, h^2,\hdots, h^m\in \clC_T,\quad \forall T\geq 1.
\]

\emph{(Part (1)}, ($\clC_T \subset A\clC_{T-1} + \sum_{k=1}^m
\big(\mathrm{diag}(h^k)\,A\big)\clC_{T-1}$). Let $f\in\clC_T$.  We
show $f$ lies in the right-hand side. By definition, there exist
$U\in\clU_T$ and $c\in\Re$ with $f=\sfL_T[c\ones,U]$.  Let $(Y,V) =
\mathrm{BS}\Delta\mathrm{E}_T[c\ones,U]$, so that
$Y_0=f$. Write $Y_1^z:=Y_1\mid_{[Z_1=z]}$ for $z\in\bO$. We claim:
\begin{equation*}
Y_1^z \in \clC_{T-1},\quad \forall\,z\in\bO. \tag{claim}
\end{equation*}

Supposing the claim to be true, the decomposition of
\Prop{prop:linear_algebra} gives
\[
 Y_1^z = \bar{Y}_1 + \tilde{Y}_1^\tp e(z),\quad z\in\bO,
\]
where $\bar{Y}_1 := \frac{1}{m+1}\sum_{z\in\bO}Y_1^z$ and
$\tilde{Y}_1^k := Y_1^k-\bar{Y}_1$ for $k=1,\ldots,m$.  Since
$\clC_{T-1}$ is a subspace, 
\[
 Y_1^z \in \clC_{T-1} \implies \bar{Y}_1 \in \clC_{T-1}, \qquad \tilde{Y}_1^k \in \clC_{T-1},\quad k=1,\ldots,m.
\]
The solution formula~\eqref{eq:Y_t_soln_formula_a} then expresses
$f$ exactly as
\[
f = A\bar{Y}_1 + \sum_{k=1}^m
\big(\mathrm{diag}(h^k)\,A\big)\tilde{Y}_1^k + h^\tp U_0,
\]
which lies in $A\clC_{T-1} + \sum_{k=1}^m
\big(\mathrm{diag}(h^k)\,A\big)\clC_{T-1}$: the first two terms
because $\bar{Y}_1 \in \clC_{T-1}$ and $\tilde{Y}_1^k \in \clC_{T-1}$,
and the remaining term $h^\tp U_0$ because $\ones\in\clC_{T-1}$.

It remains to
prove the claim.  Before doing so, we prove the reverse inclusion.

\medskip

($\clC_T \supset A\clC_{T-1} + \sum_{k=1}^m
\big(\mathrm{diag}(h^k)\,A\big)\clC_{T-1}$). Let $f\in\clC_{T-1}$,
realized as $f=\sfL_{T-1}[c\ones,\bar U]$ for some $\bar
U\in\clU_{T-1}$, $c\in\Re$. We describe, for each term on the
right-hand side, an explicit control $U\in\clU_T$ with
$(Y,V)=\mathrm{BS}\Delta\mathrm{E}_T[0,U]$ attaining
it: $Y_0=Af-c\ones$ in the first case below, $Y_0=\big(\mathrm{diag}(h^k)\,A\big)f-ch^k$
in the second. Throughout, write $Y_1^z:=Y_1\mid_{[Z_1=z]}$ for
$z\in\bO$.

\begin{itemize}
\item The term $Af$. Take $U$ so that
  \[
    Y_1^z = f - c\ones,\quad \forall\,z\in\bO,
  \]
  i.e., the same value on every branch. Then $\bar{Y}_1=f-c\ones$
  and $\tilde{Y}_1=0$, and the solution
  formula~\eqref{eq:Y_t_soln_formula_a} gives
  \[
    Y_0 = Af - c\ones.
  \]
  Since $\ones\in\clC_T$ (part (2)), $Af\in\clC_T$.

\item The term $\big(\mathrm{diag}(h^k)A\big)f$. Fix
  $k\in\{1,\ldots,m\}$, and take $U$ so that
  \[
    Y_1^z = \begin{cases} f-c\ones & Z_1=k, \\ -f+c\ones & Z_1=0, \\ 0 & Z_1\notin\{0,k\}. \end{cases}
  \]
  Then $\bar{Y}_1=0$ and
  \[
  \tilde{Y}_1^l = \begin{cases} f-c\ones & l=k, \\ 0 & l\neq k, \end{cases}
  \]
  so the solution formula~\eqref{eq:Y_t_soln_formula_a} gives
  \[
    Y_0 = \big(\mathrm{diag}(h^k)\,A\big)f - c\,h^k.
  \]
  Since $h^k\in\clC_T$ for $T\ge1$ (part (2)),
  $\big(\mathrm{diag}(h^k)\,A\big)f\in\clC_T$.
\end{itemize}

It remains to establish (i) the claim from the first direction, and
(ii) the existence of a control $U\in\clU_T$ realizing the branch
patterns $Y_1^z$ used in the second.

\newP{(i) Proof of the claim}  Since $U\in\clU_T$, the Doob--Dynkin lemma
gives a deterministic vector $\phi_0\in\Re^m$ and deterministic
functions $\phi_t:\bO^t\to\Re^m$, $t=1,\ldots,T-1$, such that
\[
U_0 = \phi_0,\quad U_t = \phi_t(Z_1,Z_2,\hdots,Z_t), \quad 1\leq t\leq T-1.
\]
Fix $z\in\bO$ and set
\[
\bar{U}_0 =  \phi_1(z),\quad \bar{U}_t = \phi_{t+1}(z,Z_1,\hdots,Z_t), \quad 1\leq t\leq T-2.
\]
Then $\bar{U} \in \clU_{T-1}$ and by uniqueness of the BS$\Delta$E solution
(\Prop{prop:existence_BSDE}), $Y_1^z =
\sfL_{T-1}[c\ones,\bar{U}]$.  

\newP{(ii) Explicit construction of control $U$} Because $f\in \clC_{T-1}$,
there exists $\bar{U}\in\clU_{T-1}$ and $c\in \Re$~s.t.~$f =
\sfL_{T-1}[c\ones,\bar{U}]$. Apply the Doob--Dynkin
lemma to $\bar U$ (fixed at the outset of this direction): there
exist a deterministic vector $\phi_0\in\Re^m$ and deterministic
functions $\phi_t:\bO^t\to\Re^m$, $t=1,\ldots,T-2$, such that
\[
\bar{U}_0 = \phi_0,\quad \bar{U}_t = \phi_t(Z_1,\ldots,Z_t), \quad 1\leq t\leq T-2.
\]
We construct $U\in\clU_T$ separately for each of the two cases
above.

\begin{itemize}
\item For $Y_0=Af$. Set
  \[
U_0 = 0,\quad U_1 = \phi_0,\quad U_{t} = \phi_{t-1}(Z_2,\ldots,Z_{t}), \quad 2\leq t\leq T-1.
\]
Then $Y_1=f-c\ones$, and hence $Y_1^z=f-c\ones$ for every
$z\in\bO$, as required.

\item For $Y_0=\big(\mathrm{diag}(h^k)A\big)f$. Set
  $U_0=0$ and
  \[
U_1 = \begin{cases} \phi_0 & Z_1=k, \\ -\phi_0 & Z_1=0, \\ 0 & Z_1\notin\{0,k\}, \end{cases}
\;
U_t = \begin{cases} \phi_{t-1}(Z_2,\ldots,Z_{t}) & Z_1=k, \\ -\phi_{t-1}(Z_2,\ldots,Z_{t}) & Z_1=0, \\ 0 & Z_1\notin\{0,k\}, \end{cases}
\quad 2\leq t\leq T-1.
\]
Then $Y_1$ takes exactly the branch pattern required above.
\end{itemize}

\emph{(Part (3), stabilization.)}
By monotonicity (part (2)), $\dim\clC_T$ is a nondecreasing sequence
of integers bounded above by $d$, with $\dim\clC_0=1$. Hence, there exists a smallest
$\tau\le d-1$ such that $\dim\clC_{\tau+1}=\dim\clC_\tau$, and since
$\clC_\tau\subset\clC_{\tau+1}$, equal dimension forces
$\clC_{\tau+1}=\clC_\tau$.   The recursion form (part (1)) gives 
\[
\clC_{\tau+1} = \clC_{\tau} \implies \clC_{T} = \clC_{\tau},\quad
\forall T\geq \tau.
\]

\end{proof}

\begin{proof}[Proof of Corollary~\ref{cor:controllable_space}]
Let $\bar{\clC}\subset \Re^d$ be any subspace satisfying the two
properties.  We claim
\begin{equation*}
\clC_T\subset \bar{\clC},\quad \forall \;T.\tag{claim}
\end{equation*}
The proof of the claim is by induction.
\begin{enumerate}
\item The base case $\clC_0=\mathrm{span}\{\ones\}\subset \bar{\clC}$
  by property 1.
\item Suppose $\clC_{T-1}\subset \bar{\clC}$.  By \Prop{prop:CT-recursion}(1), every
$f\in\clC_T$ is of form $f=Ag_0+\sum_kh^k (Ag_k)$ for
some $g_0,\ldots,g_m\in\clC_{T-1}\subset \bar{\clC}$, and property 2 (applied
to each $g_i\in \bar{\clC}$) together with $\bar{\clC}$ a subspace gives $f\in \bar{\clC}$.
Hence $\clC_T\subset \bar{\clC}$.
\end{enumerate}
From \Prop{prop:CT-recursion}(3), $\clC_{\tau+1}=\clC_\tau=\clC$, so the
recursion of part (1) gives,
\[
  \clC=A\clC+\sum_{k=1}^m \big(\text{diag}(h^k\,A\big)\clC.
\]
In particular, $\clC$
satisfies the two properties.  Since $\bar{\clC}$ was an arbitrary subspace
satisfying the two properties, $\clC$ is the smallest such subspace.
\end{proof}

\subsection{Duality between controllability and observability
  properties}
\label{sec:proof:controllability_duality}
The two results just proved (\Prop{prop:CT-recursion} and
\Corollary{cor:controllable_space}) are statements of pure linear algebra:
$\clC_T$ and its limit $\clC$ are characterized entirely in terms of
$A$ and $h$, with no reference to probability. The next result is
different in character. Defn.~\ref{def:observability} of observability
is an inherently probabilistic notion---equality of path laws. The
bridge to the control-theoretic object $\clC_T$ is supplied by the
martingale $M^U$ of \Prop{prop:MU-martingale}.

\begin{proof}[Proof of \Prop{prop:duality}]
Let $\rho\in\clP(\bS)$, $U\in\clU_T$, $F\in\clZ_T$, and
$(Y,V)=\mathrm{BS\Delta E}_T[F,U]$. Recall the associated martingale
(\Prop{prop:MU-martingale}),
\[
M_t^U := Y_t(X_t) + \sum_{s=0}^{t-1}U_s^\tp e(Z_{s+1}), \qquad 0\leq t\leq T.
\]
Since $M^U$ is a $(\clF,\sP^\rho)$-martingale, 
\[
  \E^\rho(M_T^U)=\E^\rho(M_0^U) \implies 
\rho(Y_0) = \E^\rho\big(F(X_T)\big) + \sum_{t=0}^{T-1} \E^\rho\big(U_t^\tp e(Z_{t+1})\big).
\]
Set $F=c\ones$, and let $S_T:=c+\sum_{t=0}^{T-1}U_t^\tp
e(Z_{t+1})\in\clZ_T$. Then
$
\rho(Y_0) = \E^\rho(S_T)$.  
Take $\rho=\mu,\nu$, and subtract:
\[
\mu(Y_0) - \nu(Y_0) = \E^\mu(S_T)-\E^\nu(S_T).
\]
Since
$Y_0=\sfL_T[c\ones,U]$, this shows
\[
(\mu-\nu)\in\clC_T^\perp \quad\iff\quad \E^\mu(S_T)=\E^\nu(S_T),\;\;
\forall\,(c,U)\in\Re\times \clU_T.
\]
By the bijection noted in the Remark~\ref{rem:ST-bijection},
$S_T\mapsto(c,U)$ identifies $\clZ_T$ with $\clU_T\times\Re$: ranging over all $(c,U)\in\Re\times \clU_T$ is equivalent to ranging over all
$S_T\in\clZ_T$.  Hence
\[
(\mu-\nu)\in\clC_T^\perp \quad\iff\quad \E^\mu(S_T)=\E^\nu(S_T),\;\;\forall\,S_T\in\clZ_T \quad\iff\quad \sP^\mu|_{\clZ_T}=\sP^\nu|_{\clZ_T}.
\]
That is,
\[
  (\mu-\nu) \in\clC_T^\perp \quad\iff\quad \mu \sim_T
  \nu,\quad \forall\,\mu,\nu \in \clP(\bS).
\]
This shows $\clN_T\subseteq \clC_T^\perp$.  For the reverse inclusion, let
$w\in\clC_T^\perp$.  Because $\ones\in\clC_T$,
\[
  w(\ones) = \sum_{x\in\bS} w(x) = 0 .
\]
Set $c:=\sum_{x\in \bS} w(x) \ones_{[w(x)>0]}=-\sum_{x\in \bS} w(x) \ones_{[w(x)<0]}$.  If $c=0$ then $w=0$
and there is nothing to prove.  Otherwise, define
\[
\mu(x) := \frac{1}{c} w(x) \ones_{[w(x)>0]}, \quad x\in \bS, \quad \nu(x) := -\frac{1}{c}
w(x) \ones_{[w(x)<0]},\quad x\in \bS. 
\]
Then $\mu,\nu\in \clP(\bS)$ and $(\mu-\nu) = c^{-1}w \in
\clC_T^\perp$.  The  equivalence above gives $\mu\sim_T\nu$ and hence
$w = c\,(\mu-\nu) \in \clN_T$.  Thus $\clN_T = \clC_T^\perp$.  By
\Prop{prop:CT-recursion}(3) the sequence $\{\clC_T\}$ stabilizes, and
hence so does $\{\clN_T\}$.  By definition, $\tau_2(A,H)=\min\{T :
\clN_T = \clN\}$, and thus,
\[
  \clC=\clC_T=\clN_T^\perp=\clN^\perp,\qquad \forall \, T \geq \tau_2(A,H).
\]
In particular $\clC_T = \clC$ if and only if $\clN_T = \clN$, so the two
indices coincide. 
This also shows that controllability and
observability are dual properties.  

We next show the duality between detectability and stabilizability.
Note $\clC$ is a subspace of the space of functions on $\bS$
(identified with $\Re^d$), while
$\clN$ is a subspace of its dual, the space of signed measures on
$\bS$ (also identified with $\Re^d$).  The two are related by annihilation,
\[
  \clN = \clC^\perp:= \big\{ \mu : \mu(f) = 0, \ \ \forall\, f\in\clC \big\},
\]
Consider direct sum decompositions of the function space and the
measure space as 
\[
  \Re^d = \clC \oplus \clC^c, \qquad
  \Re^d = \clN \oplus \clN^c .
\] 

The matrix $A:\Re^d\mapsto \Re^d$ is a linear operator on the space of
functions and its dual $A^\tp: \Re^d\mapsto \Re^d$ is a linear
operator on the space of measures.   By
Corollary~\ref{cor:controllable_space}, $\clC$ is $A$-invariant, and
therefore $\clN=\clC^\perp$ is $A^\tp$-invariant.  Upon choosing dual
bases for the two decompositions, 
\[
A=\begin{pmatrix}A_c&\ast\\0&A^\perp\end{pmatrix} , \qquad A^{\mathsf T}=\begin{pmatrix}A_c^{\mathsf T}&0\\
  \ast&(A^\perp)^{\mathsf T}\end{pmatrix}
\]
The $\text{spec}(A)=\text{spec}(A_c)\cup \text{spec}(A^\perp)$.

With this notation,
\begin{align*}
  \text{HMM}(A,H) \text{ detectable}
  &\iff A^\tp|_{\clN} \text{ is Schur stable} 
  \iff A^\perp \text{ is Schur stable} \\
  &\iff \mathrm{E}^{\mathrm{PER}}(A)\subseteq\clC 
  \iff \text{dual control system \eqref{eq:dual_BSDE} is stabilizable}.
\end{align*}
The first and last are Definitions~\ref{def:detectability}
and~\ref{def:controllable} for detectability and stabilizability, respectively.
For the other two:
\begin{enumerate}
\item $A^\tp|_{\clN}$ has matrix $(A^\perp)^\tp$, and a matrix and its transpose have the
  same eigenvalues.
\item The peripheral eigenvalues of $A$ are semisimple.  Since $\clC$ and
  $\mathrm{E}^{\mathrm{PER}}(A)$ are both $A$-invariant, so is
  $\clC\cap\mathrm{E}^{\mathrm{PER}}(A)\subset\clC$.  
  \begin{align*}
 \# \; \{\text{unit modulus eigenvalues of
    }A\}&=\dim(\mathrm{E}^{\mathrm{PER}}(A)),\\
    \# \; \{\text{unit modulus eigenvalues of
    }A_c\}&=\dim(\clC\cap\mathrm{E}^{\mathrm{PER}}(A)).
  \end{align*}
  Because
  $\text{spec}(A)=\text{spec}(A_c)\cup \text{spec}(A^\perp)$,
  \[
\mathrm{E}^{\mathrm{PER}}(A)\subseteq\clC \iff \dim(\mathrm{E}^{\mathrm{PER}}(A))=\dim(\clC\cap\mathrm{E}^{\mathrm{PER}}(A))\iff
A^\perp \text{ is Schur stable}. 
  \]
\end{enumerate}
\end{proof}

\section{Proof of the supermartingale property}\label{appendix:OCP_supermg}

\begin{proof}[Proof of \Prop{cor:U-opt-zero_elegant}]
  Recall the $(\clF,\sP^\rho)$-martingale from \Prop{prop:MU-martingale}(1),
  \[
  M_t^U = Y_t(X_t) + \sum_{s=0}^{t-1} U_s^\tp e(Z_{s+1}),\quad t\in \mathbb{T}.
\]
Therefore, conditional expectation and conditional variance are given by,
\begin{align*}
\pi_t^\rho(Y_t)  &= \E^\rho (M_t^U \mid \clZ_t) - \sum_{s=0}^{t-1} U_s^\tp e(Z_{s+1}),\qquad t\in \mathbb{T}.\\
\pi_t^\rho\big( (Y_t - \pi_t^\rho(Y_t))^2 \big) &=
\E^\rho\big((M^U_t)^2\mid\clZ_t\big) - \left( \E^\rho\big(M^U_t\mid\clZ_t\big) \right)^2, \qquad t\in \mathbb{T}.
\end{align*}
Therefore,
\[
\Theta^U_{t+1} = \E^\rho\big((M^U_{t+1})^2\mid\clZ_{t+1}\big) - \big(\E^\rho(M^U_{t+1}\mid\clZ_{t+1})\big)^2
             - \sum_{s=0}^{t}\E^\rho\big(\ell(Y_{s+1},V_s,U_s;X_s)
             \mid\clZ_s\big),\quad t\in \mathbb{T}^-.
           \]
The formula for the quadratic variation gives,
\begin{align*}
\E^\rho\big((M^U_{t+1})^2\mid\clF_t\big) &= (M^U_t)^2
   + \E^\rho\big(\ell(Y_{t+1},V_t,U_t;X_t)\mid\clF_t\big), \qquad t\in\mathbb{T}^-,\\
\E^\rho\big((M^U_{t+1})^2\mid\clZ_t\big) &= \E^\rho\big((M^U_t)^2\mid\clZ_t\big)
   + \E^\rho\big(\ell(Y_{t+1},V_t,U_t;X_t)\mid\clZ_t\big), \qquad t\in\mathbb{T}^-,
\end{align*}
the second by the tower property, because $\clZ_t\subseteq\clF_t$.
        
This gives
\begin{align*}
\E^\rho \big( \Theta^U_{t+1} \mid \clZ_t \big) &=
  \E^\rho\big((M^U_{t+1})^2\mid\clZ_{t}\big) -
  \E^\rho\big(\big(\E^\rho(M^U_{t+1}\mid\clZ_{t+1})\big)^2 \mid
                                               \clZ_t) \\
            & \qquad \qquad - \sum_{s=0}^{t}\E^\rho\big(\ell(Y_{s+1},V_s,U_s;X_s)
              \mid\clZ_s\big),\quad t\in \mathbb{T}^-\\
  &=  \E^\rho\big((M^U_t)^2\mid\clZ_t\big) -
  \E^\rho\big(\big(\E^\rho(M^U_{t+1}\mid\clZ_{t+1})\big)^2 \mid
    \clZ_t)  \\
   & \qquad \qquad - \sum_{s=0}^{t-1}\E^\rho\big(\ell(Y_{s+1},V_s,U_s;X_s)
     \mid\clZ_s\big),\quad t\in \mathbb{T}^-\\
  & =  \Theta^U_t +
  \big(\E^\rho(M^U_{t}\mid\clZ_{t})\big)^2  -
  \E^\rho\big(\big(\E^\rho(M^U_{t+1}\mid\clZ_{t+1})\big)^2  \mid
                                               \clZ_t), \quad t\in \mathbb{T}^-.
\end{align*}
By the Jensen's inequality, tower property, and the martingale property,
\[
  \E^\rho\Big(\big(\E^\rho(M^U_{t+1}\mid\clZ_{t+1})\big)^2 \,\Big|\, \clZ_t\Big)
    \;\ge\; \big(\E^\rho(M^U_{t+1}\mid\clZ_t)\big)^2
    \;=\; \big(\E^\rho(M^U_{t}\mid\clZ_t)\big)^2,
  \]
  with equality iff $\E^\rho(M^U_{t+1}\mid\clZ_{t+1})=\E^\rho(M^U_{t}\mid\clZ_t)$.  
Then 
\[
  \E^\rho\big(\Theta^U_{t+1}\mid\clZ_t\big)
    \;\le\; \Theta^U_t , \quad t\in \mathbb{T}^-.
\]
Hence $\Theta^U$ is a $(\clZ,\sP^\rho)$-supermartingale.  At $U =
U^\opt = 0$,
\[
  \E^\rho(M_t\mid\clZ_t) = \pi^\rho_t(Y_t) = c,\quad t\in \mathbb{T}.
\]
Since $c$ does not depend on $t$, Jensen's inequality above holds with equality for every
$t\in\mathbb{T}^-$, and $\Theta^{U^\opt}$ is a $(\clZ,\sP^\rho)$-martingale.
\end{proof}

\section{Proof of necessity}\label{appdx:thm:necessity}
We begin by recalling some background.

\newP{Notation and background}
From~\Sec{sec:model_props}, $\clE^{\mathrm{PER}}(A) = \ker\big(A^p - I\big)$, and
let $\Pi:\Re^d\to \clE^{\mathrm{PER}}(A)$ be the projection matrix
(Appendix.~\ref{appdx:prop:part-i-implication}).  Let $\{\bS_j : j=1,2,\ldots,J\}$
be the recurrent classes of $A^p$.  These are the cyclic sub-classes of the
recurrent classes of $A$~\cite[Rem.~8.4.9]{horn_johnson_1985}.  Restricted to
$\bS_j$, $A^p$ is a primitive (irreducible and aperiodic) stochastic matrix.  Let
$\rho_j\in\sP(\bS)$ be the unique invariant measure of $A^p$ supported on
$\bS_j$.

For each $x \in \bS$ and $g \in \Re^d$,
\begin{align*}
  (\Pi g)(x)
  &= \lim_{n\to\infty} \frac1n \sum_{k=0}^{n-1} \E^x\big(g(X_{kp})\big)\\
  &= \lim_{n\to\infty} \frac1n \sum_{k=0}^{n-1} \sum_{j=1}^J
     \E^x\Big(g(X_{kp})\,\ones_{\bS_j}(X_{kp})\Big)
     \qquad (\because\;\; \sP^x(X_{kp}\notin\bS^{\text{\tiny RECC}}) \to 0)\\
  &= \sum_{j=1}^J \rho_j(g)\,
     \lim_{n\to\infty} \frac1n \sum_{k=0}^{n-1} \E^x\big(\ones_{\bS_j}(X_{kp})\big)
     = \sum_{j=1}^J \rho_j(g)\,\big(\Pi\,\ones_{\bS_j}\big)(x) .
\end{align*}
In particular,
\begin{equation}\label{eq:cyclic_class}
  \Pi g \;=\; \sum_{j=1}^J \rho_j(g)\,\Pi\,\ones_{\bS_j},
  \qquad \forall\, g\in\Re^d .
\end{equation}

\begin{proof}[Proof of \Prop{thm:necessity}]
Suppose the $\mathrm{HMM}(A,H)$ is not detectable.  Fix $
\tau = p \,\tau_1(A)$. 
Our goal is to find $\rho\in\clP(\bS)$, $c\in\Re$ and
$F-c\ones\in\clZfo{\tau}(\rho)$ such that dissipation is zero while
the variance $\dvar^\rho(\sfL_\tau[F](X_0))>0$.   

We begin with a \textbf{claim}:  There exist $\rho\in\clP(\bS)$ with
$\supp(\rho)\subseteq\bS_{\mathrm{RECC}}$ and $f\in\mathrm
E^{\mathrm{PER}}(A)$ such that (a) $\rho(f)=0$, (b) $\rho(f^2)>0$, and (c) $\rho(fg)=0$ for all
$g\in\clC$.

Suppose the claim is true.  Consider
\[
(Y,V) = \mathrm{BS\Delta E}_\tau[f], \qquad M_t = Y_t(X_t),\quad 0\leq t\leq\tau .
\]
Its unique solution is given by
\[
Y_t = A^{\tau-t} f, \quad V_t = 0, \quad 0\leq t\leq \tau .
\]
Then $Y_0 = A^\tau f = f$, and the
dissipation is
\[
\E^\rho\big(\langle M\rangle_\tau\big)
= \E^\rho\big(|f(X_\tau) - f(X_0)|^2\big) = 0 ,
\]
where the last equality follows because
$\supp(\rho)\subseteq\bS_{\mathrm{RECC}}$ and $\tau$ is a multiple of
$p$, so that $X_\tau$ lies in the same cyclic subclass $\bS_j$ as
$X_0$ ($\bS_j$ is a recurrent class of $A^p$, hence closed), and $f$ is
constant on $\bS_j$ (because $f\in\text{range}(\Pi)=\mathrm
E^{\mathrm{PER}}(A)$). 
The proof is completed by showing that $f\in\clZfo{\tau}(\rho)$.  For
this purpose, consider the~\eqref{eq:dual-optimal-control} with
terminal condition $f$.  For $U\in\clU_\tau$, let
\[
(\tilde Y,\tilde V) = \mathrm{BS\Delta E}_\tau[f,U],\qquad \tilde{M}_t^U=
\tilde Y_t(X_t) + \sum_{s=0}^{t-1} U_s^\tp e(Z_{s+1}),\quad 0\leq t\leq\tau.
\]
Linearity of the BS$\Delta$E gives
\[
\tilde Y_0 = \sfL_\tau[f,U] \;=\; \sfL_\tau[f,0] + \sfL_\tau[0,U] \;=\; f + g,
\qquad g := \sfL_\tau[0,U] \in \clC_\tau \subseteq \clC .
\]
By parts (a) and (c) of the claim, $\rho(f)=0$ and $\rho(fg)=0$, so the cross
term vanishes:
\[
\dvar^\rho\big(\tilde Y_0(X_0)\big)
= \rho(f^2) + 2\rho(fg) + \dvar^\rho\big(g(X_0)\big)
= \rho(f^2) + \dvar^\rho\big(g(X_0)\big)
\;\geq\; \rho(f^2).
\]
Since $\E^\rho(\langle \tilde M^U\rangle_\tau)\geq0$ and, as shown above,
$\sfL_\tau[f,0]=f$ with $\E^\rho(\langle M\rangle_\tau)=0$,
\[
\sfJ_\tau^\rho(U;f)
= \dvar^\rho\big(\tilde Y_0(X_0)\big) + \E^\rho\big(\langle \tilde M^U\rangle_\tau\big)
\;\geq\; \rho(f^2) \;=\; \sfJ_\tau^\rho(0;f), \qquad \forall\,U\in\clU_\tau .
\]
That is, $U=0=:U^\opt$ is an optimal control for~\eqref{eq:dual-optimal-control} with optimal value
$\rho(f^2)$.  Because the MMSE estimator is unique, 
\[
\pi_\tau^\rho(f) = \rho(Y_0) - \sum_{t=0}^{\tau-1} (U^\opt)^\tp
e(Z_{t+1})=\rho(f) = 0,\quad \sP^\rho\text{-a.s.} 
\]
Therefore, $\sigma_\tau^\rho(f) =0$ and hence $f\in
\clZfo{\tau}(\rho)$.  We have thus found an $F=f$ such that,
\[
\E^\rho\big(\langle M\rangle_\tau\big)= 0, \qquad
\dvar^\rho\big(\sfL_\tau[f](X_0)\big) = \rho(f^2) > 0
\]
by part (b) of the claim, so the conclusion of \Lemma{lemma:Lasalle} fails at
$(\rho,f)$.

It remains to verify the claim.  Define
\[
 \Lambda:\Re^d\to \Re^J\qquad \text{by}\qquad \Lambda g := (\rho_1(g),\ldots,\rho_J(g)).
\]
Because the HMM is non-detectable, there exists a non-zero $b \in \Re^J$ with
\[
  b^{\mathsf T}\Lambda g=0,\quad \forall\, g\in\clC. 
\]
Indeed, suppose no such $b$ exists, i.e.,\ $\Lambda(\clC)=\Re^J$.  
Since $\clC$ is $A$-invariant, $\Pi(\clC)\subseteq\clC$.  For each
$j$, choose $g\in\clC$ with $\Lambda g = e_j$.  Then, using~\eqref{eq:cyclic_class}, 
\[
\Pi\ones_{\bS_j} = \sum_{l=1}^{J}\rho_l(g)\,\Pi\ones_{\bS_l} = \Pi g
\in \clC,\quad j=1,2,\ldots,J.
\]
So, $\mathrm
E^{\mathrm{PER}}(A)=\text{span}\{\Pi\ones_{\bS_j}:j=1,2,\ldots,J\}\subseteq\clC$.
That is, the HMM is detectable.

Set,
\[
\rho := \tfrac1J\textstyle\sum_j \rho_j, \qquad
f := J \sum_{j=1}^J b_j\,\Pi\ones_{\bS_j} .
\]
Then $\supp(\rho)\subseteq\bS_{\mathrm{RECC}}$ and
$\rho(fg)=b^{\mathsf T}\Lambda g = 0$
for every $g\in \clC$. Because $\ones\in \clC$, $\rho(f)=0$.  Finally,
$\rho(f^2)=J \sum_jb_j^2>0$.  This verifies the claim.

 
\end{proof}

\section{Proofs for \Sec{sec:filter_stability}}

\subsection{Proof of uniform bounds under Assumption~\ref{ass:equivalent-priors}}
\label{appdx:lem:uniform_ratio_bound}

\begin{proof}[Proof of \Prop{lem:uniform_ratio_bound}]
The bound for $\gamma_T$ was already derived in
\Sec{sec:gamma_T_defn}.  
By the analytical characterization in
\Prop{prop:backward-map-and-bsde},
\[
Y_t(x) = \sum_{z\in\bO} H(x,z)\,\big(AY_{t+1}^z\big)(x),
\qquad Y_{t+1}^z(x) := Y_{t+1}(x)\big|_{[Z_{t+1}=z]} .
\]
Since $H$ and $A$ are row-stochastic, $Y_t(x)$ is a convex combination of the
values $\{Y_{t+1}^z(y): y\in\bS,\,z\in\bO\}$.  Hence if $Y_{t+1}$ takes values in
$[\alpha,\alpha^{-1}]$, so does $Y_t$.  Since $Y_T=\gamma_T$, the
result follows by induction.  
\end{proof}

\subsection{Proof of the Markov property of the BS$\Delta$E}\label{sec:markov_property}

\begin{proof}[Proof of \Prop{prop:Markov_BSDE}]
The proof relies on the fact that the un-normalized measures
$\{\sigma_t^\mu:0\leq t\leq T\}$ and $\{\sigma_t^\nu:0\leq t\leq T\}$ are
Markov processes.  Indeed, the un-normalized filter recursion
is defined by the linear map $\rho\mapsto \Phi(\rho;z)$ by
\[
\Phi(\rho;z)(f):=\rho(\text{diag}(H(:,z))
    Af),\;\; z\in\bO,\; f\in \Re^d.
\]
Set $a = \pi_{\bar
  t}^\nu$ and $b = \pi_{\bar t}^\mu$.  By 
Assumption~\ref{ass:equivalent-priors}, 
$[\sigma_{\bar t}^\nu(\ones)>0]=[\sigma_{\bar t}^\mu(\ones)>0]$, and
$(\pi_{\bar
  t}^\nu, \pi_{\bar
  t}^\mu)$ are well-defined on $[\sigma_{\bar t}^\nu(\ones)>0]$.
Therefore, everything
below is understood on this event. 

Over the time-interval
$[\bar t, T]$, the recursions for the two un-normalized filters are given by,
\begin{align*}
\sigma_{t+1}^{(a,\bar{t})} &= \Phi(\sigma_{t}^{(a,\bar{t})};Z_{t+1}),\quad
  \bar t\leq t\leq T-1,\quad \sigma_{\bar t}^{(a,\bar{t})} =a, \\
\sigma_{t+1}^{(b,\bar{t})} &= \Phi(\sigma_{t}^{(b,\bar{t})};Z_{t+1}),\quad
\bar t\leq t\leq T-1,\quad \sigma_{\bar t}^{(b,\bar{t})} =b.
\end{align*}
Because $\Phi(\cdot\,;z)$ is linear,
\begin{align*}
  \sigma_T^\nu(x) &= \sigma_{\bar
  t}^\nu(\ones)\,\sigma_T^{(a,\bar t)}(x), \quad x\in \bS,\\
\sigma_T^\mu(x)&=\sigma_{\bar
  t}^\mu(\ones)\,\sigma_T^{(b,\bar t)} (x), \quad x\in \bS.
\end{align*}

Over the same time-interval $[\bar t, T]$, the BS$\Delta$E
recursion is given by,
\begin{align*}
Y_t(x) &= (AY_{t+1})(x) + h^\tp(x) V_t(x) - V_t^\transpose(x)
         e(Z_{t+1}),
         \quad x\in\bS, \quad \bar t\leq t\leq T-1, \\
  Y_T(x)&=\gamma_T(x) = \begin{cases}
    \frac{\sigma_{T}^{(b,\bar{t})}(x)}{\sigma_{T}^{(a,\bar{t})}(x)}
    \cdot \frac{\sigma_{T}^{(a,\bar{t})}(\ones)}{\sigma_{T}^{(b,\bar{t})}(\ones)}, & \sigma_T^{(a,\bar t)} (x)>0 \\
  1 & \sigma_T^{(a,\bar t)} (x)=0 , 
\end{cases}
    \qquad x\in\bS.
\end{align*}
where note $\sigma_T^{(a,\bar t)} (x)>0$ iff $\sigma_T^{\nu} (x) = \sigma_{\bar
  t}^\nu(\ones)\,\sigma_T^{(a,\bar t)}(x) >0$, and $\sigma_T^{(a,\bar
  t)} (x)>0$ implies $\sigma_{T}^{(b,\bar{t})}(\ones) \geq
\gamma_0^{\min{}} \sigma_{T}^{(a,\bar{t})}(\ones) >0 $, and 
\[
\frac{\sigma_{T}^{(b,\bar{t})}(x)}{\sigma_{T}^{(a,\bar{t})}(x)}
    \cdot
    \frac{\sigma_{T}^{(a,\bar{t})}(\ones)}{\sigma_{T}^{(b,\bar{t})}(\ones)}
    = \frac{\sigma_T^\mu(x)}{\sigma_T^\nu(x)}\cdot
      \frac{\sigma_T^\nu(\ones)}{\sigma_T^\mu(\ones)}.
    \]
    the right hand side being the formula used for defining $\gamma_T$
    (see Defn.~\ref{defn:likelihood_ratio}). 

Because of the nature of forward filter recursion, the terminal
condition $Y_T(x) = \gamma_T(x)$ is a 
deterministic function of $(a, b, Z_{\bar t+1}, Z_{\bar
t+2},\ldots,Z_T)$.  By \Prop{prop:backward-map-and-bsde}, each $Y_t$,
for $\bar t\leq t\leq T$, is a deterministic function of $(a,b,
Z_{\bar t+1}, Z_{\bar t+2},\ldots, Z_t)$ alone. At time 
$t=\bar t+\tau$, denoting this function by $\shift$,
\[
Y_{\bar t+\tau}(x) = \shift (\pi_{\bar t}^\nu,\pi_{\bar t}^\mu, Z_{\bar
t+1}, Z_{\bar t+2},\ldots, Z_{\bar t+\tau};x), \qquad x\in\bS, \quad \text{on } [\sigma^\nu_{\bar t}(\ones)>0] .
\]
The function depends upon $\bar t,T$ and HMM parameters $(A,H)$, but not upon
$(Z_1,Z_2,\ldots,Z_{\bar t})$.  
\end{proof}

\subsection{Proof of Corollary~\ref{corollary:prop:Markov_BSDE}}
\label{appdx:corollary:prop:Markov_BSDE}
\begin{proof}[Proof of Corollary~\ref{corollary:prop:Markov_BSDE}]
The two solutions,
$\{(Y_t,V_t):\bar t\leq t\leq \bar t+\tau\}$ and 
$\{(\tilde Y_t,\tilde V_t):0\leq t\leq \tau\}$, are defined using
identical recursions using identical terminal conditions, by uniqueness of the BS$\Delta$E solution
(\Prop{prop:existence_BSDE}), it follows that 
\begin{equation}\label{eq:pathwise_identity}
\tilde Y_t(x) = Y_{\bar t+t}(x), \qquad \tilde V_t(x) = V_{\bar
t+t}(x), \quad x \in \bS, \quad 0 \leq t \leq \tau.\tag{pathwise identity}
\end{equation}
Taking $t=0$ and using $\rho = \pi_{\bar t}^\nu$,
\[
\rho\big((\tilde Y_0-1)^2\big) = \sum_{x\in\bS} \pi_{\bar t}^\nu(x)
\big(Y_{\bar t}(x)-1\big)^2 = \E^\nu\big((Y_{\bar t}(X_{\bar
t})-1)^2 \mid \clZ_{\bar t}\big).
\]

To prove $\tilde F - 1 \in \clZfo{\tau}(\rho)$ and the equality for
dissipation, we first claim that the process $\{(\tilde X_t, \tilde Z_t) : 0 \leq t
\leq \tau\}$ under $\tilde\sP^\rho$ has the same law as the process
$\{(X_s, Z_s) : \bar t \leq s \leq \bar t+\tau\}$ under $\sP^\nu$,
conditioned on $\clZ_{\bar  t}$.  Suppose the claim is true.

Since $\pi_{\bar t+\tau}^\nu\big(Y_{\bar t+\tau}\big) = \E^\nu\big(\gamma_T(X_T)
\mid \clZ_{\bar t+\tau}\big) = 1$, and by construction, $\tilde F$ and
$Y_{\bar t+\tau}$ are given by the same function $\shift (\cdot)$,
evaluated at $(\tilde Z_1,\ldots,\tilde Z_\tau)$ and $(Z_{\bar t+1},\ldots,Z_{\bar t+\tau})$ respectively,
\[
\pi_\tau^\rho\big(\tilde F\big) = 1,
\]
that is, $\tilde F - 1 \in \clZfo{\tau}(\rho)$. Next, by
\Prop{prop:M0-martingale} applied to $(\tilde Y, \tilde V)$,  
\[
\tilde \E^\rho\big(\langle \tilde M \rangle_\tau\big) = \sum_{t=0}^{\tau-1}
\tilde \E^\rho\big( \ell (\tilde Y_{t+1}, \tilde V_t, 0; \tilde X_t) \big)
= \sum_{t=\bar
t}^{\bar t+\tau-1} \E^\nu\big( \ell (Y_{t+1},V_t,0;X_t) \mid \clZ_{\bar t}\big).
\]
where the last equality follows from using
the~\eqref{eq:pathwise_identity} plus the claim to identify
$\tilde\sP^\rho$-expectation with $\sP^\nu$-expectation conditional on
$\clZ_{\bar t}$. 

It remains to prove the claim.  We have
\begin{align*}
\E^\nu\big( &f(X_{\bar t},\ldots, X_{\bar t+\tau}, Z_{\bar t+1},\ldots,
  Z_{\bar t+\tau}) \mid \clZ_{\bar t} \big) \\
  & = \E^\nu\big( \E^\nu( f(X_{\bar t},\ldots, X_{\bar t+\tau}, Z_{\bar t+1},\ldots,
  Z_{\bar t+\tau}) \mid \clZ_{\bar t} \vee \sigma(X_{\bar t}))\mid \clZ_{\bar t} \big),
\end{align*}
and by the Markov property,
\begin{align*}
\E^\nu\big(& f(X_{\bar t},\ldots, X_{\bar t+\tau}, Z_{\bar t+1},\ldots,
Z_{\bar t+\tau}) \mid \clZ_{\bar t}\vee \sigma(X_{\bar t})\big) \\&= \E^\nu\big( f(X_{\bar t},\ldots, X_{\bar t+\tau}, Z_{\bar t+1},\ldots,
  Z_{\bar t+\tau}) \mid  \sigma(X_{\bar t})\big).
\end{align*}
Therefore,
\begin{align*}
\E^\nu\big( &f(X_{\bar t},\hdots, X_{\bar t+\tau}, Z_{\bar t+1},\ldots,
  Z_{\bar t+\tau}) \mid \clZ_{\bar t} \big) \\
  &= \sum_{x\in \bS}
  \E^\nu\big( f(X_{\bar t},\ldots, X_{\bar t+\tau}, Z_{\bar t+1},\ldots,
  Z_{\bar t+\tau}) \mid [X_{\bar t} =x]\big) \sP^\nu( [X_{\bar t} =x]
  \mid \clZ_{\bar t} ) \\
  &= \sum_{x\in \bS} \E^\nu\big( f(X_{\bar t},\ldots, X_{\bar t+\tau}, Z_{\bar t+1},\ldots,
  Z_{\bar t+\tau}) \mid [X_{\bar t} =x]\big) \pi_{\bar t}^\nu(x).
\end{align*}
Finally, by time-homogeneity of $\mathrm{HMM}(A,H)$,
\[
\E^\nu\big(
f(X_{\bar t},\ldots, X_{\bar t+\tau}, Z_{\bar t+1},\ldots, Z_{\bar t+\tau}) \mid X_{\bar t}=x\big) =
\tilde \E^x\big(f(\tilde X_0 ,\ldots, \tilde X_{\tau}, \tilde Z_{1},\ldots,\tilde Z_\tau)\big), \quad \forall\;x\in\bS.
\]
Substituting above and using $\rho = \pi_{\bar t}^\nu$,
\begin{align*}
\E^\nu\big( & f(X_{\bar t},\ldots, X_{\bar t+\tau}, Z_{\bar t+1},\ldots, Z_{\bar t+\tau}) \mid \clZ_{\bar
t} \big) \\&= \sum_{x\in\bS} \rho(x)\, \tilde \E^x\big(f(\tilde X_0,\ldots, \tilde X_{\tau}, \tilde Z_{1},\ldots,
\tilde Z_\tau)\big) = \tilde\E^\rho\big(f(\tilde X_0,\ldots, \tilde X_{\tau}, \tilde Z_{1},\ldots,\tilde
Z_\tau)\big),
\end{align*}
proving the claim.
\end{proof}

\section{Proofs for \Sec{sec:rates}}
\subsection{Characterization of
  $\sfN_\tau$}\label{appdx:characterization_of_N}

\begin{proof}[Proof of \Prop{prop:N-strata}]
Let $\rho\in\sfN_\tau$.  The key to the proof is the following
claim: 
\[
\textbf{(claim)}\qquad  \clZfo{\tau}(\rho')=\clZfo{\tau}(\rho),\quad
\forall \, \rho'\in\clP(\bS) \quad \text{s.t.} \quad \supp(\rho')=\supp(\rho).
\]
Suppose the claim holds.  Let $\rho'\in \clP(\bS)$ with $\supp(\rho')=\supp(\rho)$ and let
$F\in \clZfo{\tau}(\rho')$.  Then $F\in \clZfo{\tau}(\rho)$, and
therefore
\[
\sfL_\tau[F](x) = 0, \quad
\forall\,x\in \supp(\rho) = \supp(\rho').
\]
Therefore, $\rho' \in \sfN_\tau$.  That is,
$\clP^\circ(\supp(\rho)) \subset\sfN_\tau$. 

It remains to prove the claim.  Because the map $\rho\mapsto
\sigma_\tau^\rho$ is linear, 
\[
\sigma_\tau^\rho(F) = \sum_{x\in \supp(\rho)} \rho(x) \sigma_\tau^x(F).
\]
Therefore, without any assumptions on $\rho$, it always holds that
\begin{equation}\label{eq:one_direction}
\bigcap_{x\in \supp(\rho)} \clZfo{\tau}(\delta_x) \subseteq \clZfo{\tau}(\rho).
\end{equation}
We show that the reverse inclusion holds for $\rho\in \sfN_\tau$.  For
this note that the space $\clZfo{\tau}(\rho)$ is closed under multiplication by
$\clZ_\tau$-measurable scalars: if $\sigma_\tau^\rho(F)=0$ and $c$ is
$\clZ_\tau$-measurable, then $\sigma_\tau^\rho(cF)=c\,\sigma_\tau^\rho(F)=0$.

Suppose $\rho \in \sfN_\tau$, $F\in \clZfo{\tau}(\rho)$, and let $x\in
\supp(\rho)$.   Taking
\[
c= \begin{cases} \pi_\tau^x(F), &
  \sigma_\tau^x(\ones)>0 \\ 0, & \text{o.w.} \end{cases},
\]
we have
\[
  0=\sfL_\tau[cF](x) = \E^x(\pi_\tau^x(F) F(X_\tau) ) =
  \E^x(|\pi_\tau^x(F)|^2)\implies \pi_\tau^x(F)=0 ,\quad \sP^x\text{-a.s.}
\]
where we used the tower property.  We have two cases to consider:
\begin{enumerate}
\item If $\sigma^x_\tau(\ones)>0$ then $\sigma_\tau^x(F) =
  \pi_\tau^x(F) \sigma^x_\tau(\ones) = 0$.
\item If $\sigma^x_\tau(\ones)=0$ then $\sigma^x_\tau(F)=0$ even
  though $\pi_\tau^x(F)$ is undefined.
\end{enumerate}
In particular, $F\in \clZfo{\tau}(\delta_x)$.  We have thus shown
\[
\rho \in \sfN_\tau \implies \clZfo{\tau}(\rho)\subseteq\clZfo{\tau}(\delta_x),
\qquad \forall\,x\in\supp(\rho).
\]
Combined with~\eqref{eq:one_direction},
\begin{equation}\label{eq:N-pointwise}
\rho \in \sfN_\tau \implies\clZfo{\tau}(\rho)= \bigcap_{x\in \supp(\rho)} \clZfo{\tau}(\delta_x).
\end{equation}
Applying~\eqref{eq:one_direction} once again, now to $\rho'$, noting
$\supp(\rho)=\supp(\rho')$, 
\[
\rho \in \sfN_\tau \implies \clZfo{\tau}(\rho) \subseteq \clZfo{\tau}(\rho') .
\]
By \Prop{Prop:rank-is-constant}, $\dim\clZfo{\tau}(\cdot)$ depends on its argument only
through the support.  A subspace contained in another of
the same dimension coincides with it, and hence
\[
\rho \in \sfN_\tau \implies \clZfo{\tau}(\rho) = \clZfo{\tau}(\rho').
\]
This proves the claim.  

It remains to show $\sfN_\tau$ is closed.  
Let $\rho_n\to\rho$ with $\rho_n\in\sfN_\tau$.  Then $\supp(\rho)
\subseteq \supp(\rho_n)$ for $n$ large enough.  Let
$F\in\clZfo{\tau}(\rho)$ and $x\in\supp(\rho)$.  Set
\[
c_n= \begin{cases} \pi_\tau^{\rho_n}(F), &
  \sigma_\tau^{\rho_n}(\ones)>0 \\ 0, & \text{o.w.} \end{cases},
\]
so that
$\sigma_\tau^{\rho_n}(F-c_n\ones)=0$.  That is, $F-c_n\ones\in\clZfo{\tau}(\rho_n)$.  Since
$\rho_n\in\sfN_\tau$ and $x\in\supp(\rho_n)$,
\[
  0 \;=\; \sfL_\tau[F-c_n\ones](x) \;=\; \E^x(F(X_\tau)) - \E^x(c_n).
\]
We claim that $\lim_{n\to\infty} \E^x(c_n)=0$.  Then 
\[
\sfL_\tau[F](x) = 0,\quad \forall\, F\in\clZfo{\tau}(\rho),\quad \forall\, x\in \supp(\rho)\quad  \implies \quad \rho\in\sfN_\tau.
\]
To show the limit, note that $\rho(x)>0$ implies
$\sP^x\ll\sP^\rho$. So,
\[
\sigma^\rho_\tau(\ones)>0,\quad \sP^\rho\text{-a.s.} \quad \implies
\quad \sigma^\rho_\tau(\ones)>0,\quad \sP^x\text{-a.s.}
\]
Because $\rho\mapsto\sigma^\rho_\tau$ is linear,
$\sigma^{\rho_n}_\tau\to\sigma^{\rho}_\tau$.  Therefore,
\[
  c_n \;=\; \pi_\tau^{\rho_n}(F) = \frac{\sigma^{\rho_n}_\tau(F)}{\sigma^{\rho_n}_\tau(\ones)}
  \;\longrightarrow\;
  \frac{\sigma^{\rho}_\tau(F)}{\sigma^{\rho}_\tau(\ones)} \;=\; \pi_\tau^{\rho}(F)=
  0,\quad \sP^x\text{-a.s.}
\]
As $|c_n|\le\|F\|_\infty$, dominated convergence gives $\lim_{n\to\infty} \E^x(c_n) = 0$.
\end{proof}

\subsection{Proof of regularity of the Poincar\'e
  constant}\label{appdx:prop:beta-regularity} 

We begin with some preliminaries.  
Fix $\rho\in \clP(\bS)$.  The feasible set for the optimization is
\[
K(\rho):=\big\{F \in\clZfo{\tau}(\rho), \ \ \rho\big(|\sfL[F]|^2\big) = 1 \big\}.
\]
In general, the feasible set need not be compact.  The issue arises
when there exists $\tilde{F} \in \clZfo{\tau}(\rho)$ such that
$\E^\rho \big( |\tilde{F}(X_\tau)|^2 \big) = 0$.  Then
\[
\rho\big(|\sfL[\tilde F]|^2\big)   \leq \E^\rho \big(
|\tilde{F}(X_\tau)|^2 \big) =0 \implies \sfL[\tilde F](x) = 0,\quad
x\in \supp(\rho).
\]
Therefore, $\rho\big(|\sfL[F+c\tilde F]|^2\big) =
\rho\big(|\sfL[F]|^2\big)$.  That is, if $F\in K(\rho)$, 
then $F+c \tilde F \in K(\rho)$, for every $c\in
\Re$.  These unbounded directions, however, do not pose a problem,
since the energy is invariant along them,
\[
\rho (\clE[\tilde F]) \leq \E^\rho \big(
|\tilde{F}(X_\tau)|^2 \big) =0 \implies \rho (\clE[F+c\tilde F]) = \rho (\clE[F]).
\]

The problem of unbounded directions is resolved by restricting the
optimization to a suitably chosen subspace as follows:
\[
  \clS^\rho_\tau := \big\{\, F\in\clZfo{\tau}(\rho) \;:\;
    F\,\ones_{[\sigma^\rho_\tau = 0]} = 0 \,\big\}.
\]
To see that $\clS^\rho_\tau$ is a subspace, we construct an
explicit projection operator as follows.
\begin{align*}
  \text{(projection)} \quad & \Pi(\rho) : \clZf{\tau} \to
    \clS^\rho_\tau \quad\text{by}\quad
    (\Pi(\rho)[F])(x) := \big(F(x) - \pi^\rho_\tau(F)\big)\,
      \ones_{[\sigma^\rho_\tau(x) > 0]},\quad x\in \bS.
\end{align*}
Indeed, $\Pi(\rho)$ is the identity on $\clS^\rho_\tau$, and
\[
  \sigma^\rho_\tau\big(\Pi(\rho)[F]\big)
  = \sigma^\rho_\tau\big( (F-\pi^\rho_\tau(F))\,
      \ones_{[\sigma^\rho_\tau>0]}\big)
  = \sigma^\rho_\tau\big( F - \pi^\rho_\tau(F)\big) = 0 ,
\]
so that $\Pi(\rho)[F]\in\clS^\rho_\tau$ for every $F\in\clZf{\tau}$.
Therefore $\Pi(\rho)$ is a projection with
$\mathrm{range}(\Pi(\rho)) = \clS^\rho_\tau$.

\begin{remark}\label{rem:sigma-positive}
Set $\clR_\tau(\rho) := \bigcup_{x\in\supp(\rho)}\clR_\tau(x)$.  Under
the non-degeneracy Assumption~\ref{ass:non-degeneracy},
$\sigma^\rho_\tau(\ones)>0$ surely (Remark~\ref{rem:surely}) and 
\[
\sigma^\rho_\tau(x)>0 \quad \text{iff} \quad x\in
\clR_\tau(\rho),\quad x\in \bS.
\]
The indicator is then the deterministic function
$\ones_{\clR_\tau(\rho)}$ on $\bS$, and
\[
  \big(\Pi(\rho)[F]\big)(x)
  = \big(F(x)-\pi^\rho_\tau(F)\big)\,\ones_{\clR_\tau(\rho)}(x),
  \qquad x\in\bS .
\]
In this case, the dimension has a simple formula,
\[
  \dim\big(\clS^\rho_\tau\big)
  = \big(|\clR_\tau(\rho)| - 1\big)\,(m+1)^\tau.
\]
\end{remark}
The following proposition collects the properties of $\clS^\rho_\tau$
and $\Pi(\rho)$. 
\begin{proposition}\label{prop:S_rho_Pi_rho}
Let $\rho\in\clP(\bS)$.  Then:
\begin{enumerate}
\item $\clS^\rho_\tau$ is a finite-dimensional inner-product space
  under $\langle F,G\rangle_\rho := \E^\rho\big(F(X_\tau)G(X_\tau)\big)$.
\item For every $F\in\clZfo{\tau}(\rho)$, with
  $\tilde F := \Pi(\rho)[F]$,
  \[
    \rho\big(\clE[F]\big) = \rho\big(\clE[\tilde F]\big),
    \qquad
    \rho\big(|\sfL[F]|^2\big) = \rho\big(|\sfL[\tilde F]|^2\big).
  \]
  Consequently,
  \[
    \beta(\rho) = \inf\big\{ \rho(\clE[F]) \;:\;
      F\in\mathrm{range}(\Pi(\rho)), \ \ \rho\big(|\sfL[F]|^2\big) = 1 \big\},
    \qquad \inf\emptyset := +\infty .
  \]
\item Let $F\in\clS_\tau^\rho$.  Then the map $\tilde
  \rho\mapsto\Pi(\tilde \rho)[F]$ is locally Lipschitz at $\rho$.  That
  is, there exists a neighborhood $\mathscr{N}(\rho)   \subset \clP(\bS)$ such
  that
  \[
    \left\| \big(\Pi(\rho) - \Pi(\tilde \rho)\big) [F] \right\|_\infty \leq C
\|F\|_\infty \|\rho - \tilde\rho\|_{\mathrm{TV}},\quad \forall\,\tilde\rho\in \mathscr{N}(\rho).
\]
\item Consider $\tilde \rho \in \clP(\bS)$ such that
  $\supp(\rho)=\supp(\tilde\rho)=\bS'$.  For $F\in\clZf{\tau}$, let 
\[
  G:=
  \big(\Pi(\rho) - \Pi(\tilde \rho)\big) [F], \quad \text{and} \quad r(x) = \begin{cases}
    \frac{\tilde\rho(x)}{\rho(x)}, & x\in \bS' \\ 1, &
    \text{o.w.}\end{cases}, \quad x\in \bS,
\]
and $r_{\min{}}=\min\{r(x):x\in \bS\}$. 
  Then
  \[
    \E^\rho \big( |G(X_\tau)|^2 \big) \leq \frac{\|r-1\|^2_\infty}{r_{\min{}}^2} \dvar^\rho(F(X_\tau)).
  \]
  In particular, if $\|r-1\|_\infty \leq \half$ then $r_{\min{}} \geq
  \half$ and the constant is at most $4 \|r-1\|^2_\infty$.  
\end{enumerate}
\end{proposition}
\begin{proof}[Proof of \Prop{prop:S_rho_Pi_rho}]
(Proof of part 1). We need to show that $\langle F,G\rangle_\rho$ defines a norm on
$\clS_\tau^\rho$.  Suppose $\E^\rho\big(|F(X_\tau)|^2\big)=0$.  Then using
the tower property,
\[
\E^\rho\big(|F(X_\tau)|^2\big)=\E^\rho\big(\pi_\tau^\rho(F^2)\big)=0 \implies
\pi_\tau^\rho(F^2)=0,\quad \sP^\rho\text{-a.s.}. 
\]
Because $F(x)$ is $\clZ_\tau$-measurable, this gives
\[
F(x)=0, \quad \forall \; x\in \bS, \quad \text{such that},\quad
\sigma_\tau^\rho(x)>0.  
\]
Because $F\in\clS_\tau^\rho$, $F(x)=0$ for $x$ such that
$\sigma_\tau^\rho(x)=0$.  We thus have $F=0$ in $\clZf{\tau}$.

\medskip

\emph{(Proof of part 2).} Let $F\in\clZfo{\tau}(\rho)$ then
$\sigma_\tau^\rho(F) = 0$.  So,
\[
  \tilde F(x) := \Pi(\rho)[F](x) = F(x)
  \ones_{[\sigma_\tau^\rho(x)>0]} \quad \implies \quad
  \tilde{F}(X_\tau)  = F(X_\tau), \quad \sP^\rho\text{-a.s.}
\]
This implies
\begin{align*}
\E^\rho(|F(X_\tau)|^2) &= \E^\rho(|\tilde F(X_\tau)|^2)\\
  |L[F](x) - L[\tilde F](x)|^2&= |\E^x (F(X_\tau)- \tilde F(X_\tau)) |^2 \\
                       &\leq \E^x (|F(X_\tau) -\tilde
F(X_\tau)|^2) =0,
\quad x\in \supp(\rho).
\end{align*}
Therefore,
\[
  \rho\big(|\sfL[F]|^2\big)  = \rho\big(|\sfL[\tilde F]|^2\big).
\]
Because $\clS_\tau^\rho\subset \clZfo{\tau}(\rho)$,
\begin{align*}
\E^\rho(|F(X_\tau)|^2)  = \rho\big(|\sfL[F]|^2\big) +
  \rho\big(\clE[F]\big), \\
\E^\rho(|\tilde F(X_\tau)|^2)  = \rho\big(|\sfL[\tilde F]|^2\big) +
  \rho\big(\clE[\tilde F]\big).
\end{align*}
Subtracting, we obtain
\[
  \rho\big(\clE[F]\big) = \rho\big(\clE[\tilde F]\big).
\]

\emph{(Proof of part 3).}  Recall the branchwise un-normalized measure
from~\eqref{eq:r_rho_s_rho},
\[
\mathfrak{p}^\rho_\tau(\zeta,x) :=
\sP^\rho(Z_1=\zeta_1,\ldots,Z_\tau=\zeta_\tau,X_\tau=x),\quad
x\in\bS,\;\zeta=(\zeta_1,\ldots,\zeta_\tau)\in \bO^\tau.  
\]
Let $
\mathfrak{p}_{\min{}}:=\min\{\mathfrak{p}^\rho_\tau(\zeta,x) \;:\;
\mathfrak{p}^\rho_\tau(\zeta,x)>0,\;x\in \bS, \; \zeta\in \bO^\tau\}$ and set,
\[
\mathscr{N}(\rho)  := \{\tilde \rho\in \clP(\bS)\; :\; |
\mathfrak{p}^\rho_\tau(\zeta,x) - \mathfrak{p}^{\tilde\rho}_\tau(\zeta,x)|< \half \mathfrak{p}_{\min{}}, \; x\in \bS, \; \zeta\in \bO^\tau\}.
\]
Because $\rho\mapsto \mathfrak{p}^\rho_\tau$ is linear, $\mathscr{N}(\rho)$
is an open neighborhood of $\rho\in \clP(\bS)$.  Then for $\tilde \rho
\in \mathscr{N}(\rho)$,  
\[
\mathfrak{p}^{\tilde\rho}_\tau(\zeta,x) > \half \mathfrak{p}_{\min{}} \quad \text{whenever}
\quad \mathfrak{p}^{\rho}_\tau(\zeta,x) > 0, \quad x\in \bS, \quad \zeta\in \bO^\tau.
\]
Therefore,
\[
[\sigma_\tau^\rho(x) > 0] \subset [\sigma_\tau^{\tilde\rho}(x) >
0],\quad x\in \bS,\quad \forall \tilde \rho
\in \mathscr{N}(\rho).
\]
which means
\[
\ones_{[\sigma_\tau^\rho(x) > 0]} \ones_{[\sigma_\tau^{\tilde\rho}(x)
  >0]} = \ones_{[\sigma_\tau^\rho(x) > 0]}, \quad x\in \bS, \quad \tilde \rho
\in \mathscr{N}(\rho)
\]

Let $F\in \clS_\tau^\rho$.
\[
\Pi(\rho) [F] =F, \quad \big(\Pi(\tilde \rho) [F]\big)(x) = \left( F(x) - \pi_\tau^{\tilde \rho}(F) \right) \ones_{[\sigma_\tau^{\tilde\rho}(x)
  >0]},\quad x\in \bS,
\]
and by the indicator product identity just described,
\[
F(x) \ones_{[\sigma_\tau^{\tilde\rho}(x) >0]} = F(x)
  \ones_{[\sigma_\tau^\rho(x) > 0]}, \quad
  x\in \bS.  
\]
Then 
\[
\Pi(\rho) [F](x) - \Pi(\tilde \rho) [F] (x) = \pi_\tau^{\tilde \rho}(F)
    \ones_{[\sigma_\tau^{\tilde\rho}(x)>0]},\quad x\in \bS.
  \]
  To estimate $ \pi_\tau^{\tilde \rho}(F)$, we consider two cases as follows.
\begin{enumerate}
\item Suppose $\sigma_\tau^{\rho}(\ones) = 0$.  Then because $F=F
  \ones_{[\sigma_\tau^{\rho}>0]}$, 
  \[
\sigma_\tau^{\tilde
     \rho}( F )= \sigma_\tau^{\tilde
     \rho}( F \ones_{[\sigma_\tau^{\rho}>0]}) =0 \implies \pi_\tau^{\tilde
     \rho}( F ) =0.
 \]
\item Suppose $\sigma_\tau^{\rho}(\ones) > 0$.   Then, by definition
  of $\mathfrak{p}_{\min{}}$, we have $\sigma^\rho_\tau(\ones)\geq \mathfrak{p}_{\min}$.
  For $\tilde\rho\in\mathscr{N}(\rho)$, then 
  $\sigma^{\tilde\rho}_\tau(\ones)>\half \mathfrak{p}_{\min}>0$.  Next, because $F\in \clS_\tau^\rho$, we have $\sigma_\tau^{\rho}(F)=0$, and thus
\[
 |\pi_\tau^{\tilde \rho}(F)| = \frac{|\sigma_\tau^{\tilde
     \rho}(F)-\sigma_\tau^{\rho}(F)|}{ \sigma_\tau^{\tilde
     \rho}(\ones)} \leq \frac{4\|\tilde \rho- \rho\|_{\text{TV}}
\|F\|_\infty}{ \mathfrak{p}_{\min{}}}.
\]
\end{enumerate}
Combining the two cases,
\[
|(\Pi(\rho) [F] - \Pi(\tilde \rho) [F])(x)|\leq |\pi_\tau^{\tilde
  \rho}(F)| \leq \begin{cases} \frac{4\|\tilde \rho- \rho\|_{\text{TV}}
    \|F\|_\infty}{ \mathfrak{p}_{\min{}}},& \sigma_\tau^{\rho}(\ones) > 0, \\
  0, & \sigma_\tau^{\rho}(\ones) =0. 
  \end{cases}
\]
and thus
\[
  \big\| \big(\Pi(\rho)-\Pi(\tilde\rho)\big)[F] \big\|_\infty
  \;\le\; C\,
    \|\rho-\tilde\rho\|_{\mathrm{TV}}\,\|F\|_\infty ,
  \qquad \tilde\rho\in\mathscr{N}(\rho).
\]
where the constant $C:=\frac{4}{ \mathfrak{p}_{\min{}}}$.  

\emph{(Proof of part 4).} Because $\supp(\rho)=\supp(\tilde\rho)=\bS'$ and $\sigma^\rho_\tau(x) =
\sum_{x'\in\bS'}\rho(x')\sigma^{x'}_\tau(x)$ is a sum of nonnegative terms with
strictly positive weights,
\[
  [\sigma^\rho_\tau>0] \;=\; \bigcup_{x_0\in\bS'} [\sigma^{x_0}_\tau>0]
  \;=\; [\sigma^{\tilde\rho}_\tau>0] \quad \implies\quad 
  \ones_{[\sigma_\tau^\rho>0]} = \ones_{[\sigma_\tau^{\tilde\rho}>0]}
\]
Therefore,
\[
G=\big(\Pi(\rho)[F] - \Pi(\tilde\rho) [F]\big)= (\pi_\tau^{\tilde\rho}(F) - \pi_\tau^\rho(F))\ones_{[\sigma_\tau^\rho>0]}.
\]
From the change of measure formula,
\[
\pi_\tau^{\tilde\rho}(F) = \E^{\tilde\rho}(F(X_\tau) \mid \clZ_\tau) =
\frac{\E^{\rho}(r(X_0) F(X_\tau) \mid \clZ_\tau)}{D_\tau}.
\]
where $D_\tau=\E^{\rho}(r(X_0) \mid \clZ_\tau)\geq r_{\min{}}$. 
Therefore,
\begin{align*}
|G(X_\tau)|^2  \leq  |(\pi_\tau^{\tilde\rho}(F) -
              \pi_\tau^\rho(F))|^2 &= |\frac{1}{D_\tau}\E^{\rho}\big((r(X_0)
  -1) (F(X_\tau) - \pi^\rho_\tau(F))\mid \clZ_\tau\big)|^2  \\
  &\leq
    \frac{\|r-1\|_\infty^2}{ r_{\min{}}^2}\E^{\rho}\big(|F(X_\tau) - \pi^\rho_\tau(F)|^2 \mid \clZ_\tau\big).
\end{align*}
Taking expectation of both sides, 
\[
\E^\rho(|G(X_\tau)|^2) \leq  \frac{\|r-1\|^2_\infty}{r_{\min{}}^2} \dvar^\rho(F(X_\tau)).
\]
\end{proof}

Now we are ready to prove \Lemma{prop:beta-regularity}.

\begin{proof}[Proof of \Lemma{prop:beta-regularity}]
\emph{(Part 1, existence.)}  Fix $\rho\in \clP(\bS)\setminus \sfN_\tau$.  Suppose $F_n\in \clS_\tau^\rho$ is an infimizing sequence with
$\beta_n=\rho(\clE[F_n])$, $ \rho
\big(|\sfL[F_n]|^2 \big)=1$, and $\beta_n \to \beta(\rho)<\infty$
(the limit must be finite because $\beta(\rho)$ is bounded from below,
by zero, and the feasible set is not empty). From the~\eqref{eq:Lyapunov_equation},  
\[
\E^\rho \big( |F_n (X_\tau)|^2 \big) = \rho(\clE[F_n]) +\rho
\big(|\sfL[F_n]|^2 \big) =  \beta_n + 1 < C < \infty,\quad \forall \; n.
\]
By \Prop{prop:S_rho_Pi_rho}(1), $\E^\rho(|F_n(X_\tau)|^2)$ is the
squared norm of $F_n\in\clS^\rho_\tau$, a finite-dimensional inner
product space; the
bound above therefore shows $\{F_n:n=1,2,\hdots\}$ is bounded, and so
has a convergent subsequence.  Relabelling, $F_n \to F$ for some
$F\in\clS^\rho_\tau$.  Then, for each $x\in
\supp(\rho)$, 
\begin{align*}
|\sfL[F] (x) - \sfL[F_n] (x) |^2&= |\E^x(F(X_\tau) - F_n(X_\tau))|^2 \leq
                                  \E^x(|F(X_\tau) - F_n (X_\tau)|^2).
\end{align*}
Therefore,
\[
\sfL[F_n] (x)\to \sfL[F] (x), \quad x\in \supp(\rho) \implies \rho
\big(|\sfL[F]|^2 \big) = \lim_{n\to\infty} \rho
\big(|\sfL[F_n]|^2 \big)=1.
\]
Finally, we show that $\rho(\clE[F])=\beta(\rho)$. Because $F\in
\clS_\tau^\rho$, from~\eqref{eq:Lyapunov_equation},   
\begin{align*}
\rho(\clE[F])  &= \E^\rho \big( |F (X_\tau)|^2 \big) - \rho
\big(|\sfL[F]|^2 \big)\\&= \lim_{n\to\infty} \E^\rho \big( |F_n
(X_\tau)|^2 \big) - \rho
\big(|\sfL[F_n]|^2 \big)= \lim_{n\to\infty} \beta_n = \beta(\rho).
\end{align*}

\medskip

\emph{(Part 4, positivity.)}  If $\rho\in \sfN_\tau$, $\beta(\rho)=+\infty$.
Suppose $\rho\in \clP(\bS)\setminus \sfN_\tau$, and the minimum is $0$,
attained at $F \in \clS_\tau^\rho$.  Then $\rho(\clE[F]) =0$ and $\rho(|L[F]|^2)= 1$. By
\Lemma{lemma:Lasalle}, $\sfL[F](x) = 0$ on $\mathrm{supp}(\rho)$, so
$\rho(|L[F]|^2)= 0$, which provides a contradiction.

\medskip

\emph{(Part 2, u.s.c.)}  If
$\beta(\rho)=+\infty$ there is nothing to prove.  Suppose
$\beta(\rho)<\infty$.  For any sequence $\rho_n\to\rho$ we show that 
\begin{equation}\label{eq:upper-semi-cont}\tag{u.s.c.}
\limsup_{n\to\infty} \beta(\rho_n) \leq \beta(\rho).
\end{equation}
Since we are interested only in the tail of the sequence, upon
relabeling, we assume that $\{\rho_n:n=1,2,\ldots\}\subset
\mathscr{N}(\rho)$ (see \Prop{prop:S_rho_Pi_rho} (part 3)).

By (part 1), there exists a minimizer $F\in \clS_\tau^\rho$ with
$\rho(\clE[F])=\beta(\rho)$ and $\rho(|L[F]|^2)=1$.  Set
\[
F_n= \Pi(\rho_n)[F],\quad \tilde{F}_n= F_n - F = \Pi(\rho_n)[F] -
\Pi(\rho)[F].  
\]
Let $\epsilon_n:=\|\rho - \rho_n\|_{\mathrm{TV}}$.  Then,
$\epsilon_n\to 0$ and from \Prop{prop:S_rho_Pi_rho},
$\|\tilde{F}_n\|_\infty=O(\epsilon_n)$.

Consider the
BS$\Delta$E solution and the associated martingale,
\[
  (\tilde{Y}_n,\tilde{V}_n) = \mathrm{BS\Delta E}_\tau[\tilde{F}_n], \qquad \tilde{M}_{n,t} = \tilde{Y}_{n,t}(X_t),
  \quad 0\le t\le\tau .
\]
Then, upon using the bounds in \Prop{prop:technical},
\begin{align*}
\sfL[F_n](x) &= \sfL[F](x) +\sfL[\tilde F_n](x) =  \sfL[F](x) +
  O(\epsilon_n),\quad x\in \bS.\\
\clE[F_n](x) &= \clE[F](x) + 2 \E^x\big(\langle M, \tilde
               M_{n}\rangle_{\tau} 
                 \big)+ \clE[\tilde F_n](x), \quad x\in\bS\\
  & = (\clE[F])(x) + O(\epsilon_n),\quad x\in\bS .  
\end{align*}
Therefore,
\begin{align*}
\rho_n(|\sfL[F_n]|^2) & = \rho(|\sfL[F_n]|^2) +
  (\rho-\rho_n)(|\sfL[F_n]|^2) =  \rho(|\sfL[F]|^2) +
  O(\epsilon_n),\\
\rho_n(\clE[F_n]) &= \rho(\clE[F_n]) +  (\rho-\rho_n)(\clE[F_n]) =
                    \rho(\clE[F]) +  O(\epsilon_n). 
\end{align*}
Finally,
\[
\beta(\rho_n) \leq
\frac{\rho_n(\clE[F_n])}{\rho_n(|\sfL[F_n]|^2)} \to \beta(\rho).
\]
Passing to a subsequence along which $\beta(\rho_n)$ converges to
$\limsup_{n} \beta(\rho_n)$ gives~\eqref{eq:upper-semi-cont} property.  

\medskip

\emph{(Part 3, continuity.)}  Consider a sequence
$\{\rho_n:n=1,2,\ldots\}$ with limit $\rho$ such that
$\mathrm{supp}(\rho_n)=\mathrm{supp}(\rho)=:\bS'$.   
We show that
\begin{equation}\label{eq:lower-semi-cont}\tag{l.s.c.}
\beta(\rho) \leq \liminf_{n\to\infty} \beta(\rho_n) .
\end{equation}
Combined with the~\eqref{eq:upper-semi-cont} property, which holds for
arbitrary sequences, and hence also for this particular one, we have
shown continuity.

From \Prop{prop:N-strata}, we have two alternatives to consider:
\begin{enumerate}
\item Either $\clP^\circ(\bS') \subset\sfN_\tau$ in which case 
  $\beta(\rho) = \beta(\rho_n)=+\infty$ and~\eqref{eq:lower-semi-cont} holds. 
\item Or else, for each $\rho_n$, there exists a minimizer $F_n\in \clS_\tau^{\rho_n}$ with
$\rho_n(\clE[F_n])=\beta(\rho_n)$ and $\rho_n(|L[F_n]|^2)=1$.   Let
\[
\bar{\beta} = \liminf_{n\to\infty}
\beta(\rho_n), \quad \delta = \min\{\rho(x): x\in \bS'\}.
\]
From the~\eqref{eq:upper-semi-cont} property,
  \[
    \bar{\beta} \;\le\; \limsup_{n\to\infty}\beta(\rho_n) \;\le\; \beta(\rho) \;<\; \infty .
  \]
\end{enumerate}
Because $\bS'$ is finite,
$\delta>0$.  For sufficiently large $n$ then
\[
\beta(\rho_n) < \bar{\beta} + 1,\quad \rho_n(x) > 
\frac{\delta}{2},\quad \forall \; x\in \bS'.
\]
We set
\[
r_n(x) = \begin{cases}
    \frac{\rho_n(x)}{\rho(x)}, & x\in \bS' \\ 1, &
    \text{o.w.}\end{cases}, \quad x\in \bS.
\]
and choose a subsequence such that $\|r_n - 1\|_\infty <
\frac{1}{n}$.  

From~\eqref{eq:Lyapunov_equation},  
\[
\E^{\rho_n} \big( |F_n (X_\tau)|^2 \big) = \beta(\rho_n) +\rho_n
\big(|\sfL[F_n]|^2 \big) \leq \bar\beta +2 < \infty.
\]
The left hand side
\[
\E^{\rho_n} \big( |F_n (X_\tau)|^2 \big) = \sum_{x\in \bS'} \rho_n(x) \E^{x} \big( |F_n (X_\tau)|^2 \big)
\]
and using the bounds,
\[
\E^{x} \big( |F_n (X_\tau)|^2 \big) <
\frac{2}{\delta} (2+\bar\beta),\quad \forall\;x\in\bS'. 
\]
Therefore,
\begin{align*}
\E^{\rho} \big( |F_n (X_\tau)|^2 \big) &= \E^{\rho_n} \big( |F_n
(X_\tau)|^2 \big) + \sum_{x\in \bS'} (\rho(x)-\rho_n(x)) \E^{x} \big(
|F_n (X_\tau)|^2 \big)\\
&= \E^{\rho_n} \big( |F_n
(X_\tau)|^2 \big) + O(\frac{1}{n}).
\end{align*}
Similarly,
\[
\rho_n
\big(|\sfL[F_n]|^2 \big) = \rho
\big(|\sfL[F_n]|^2 \big)  + O(\frac{1}{n}).
\]
Therefore, we arrive at an approximate dissipation equation,
\begin{equation}\label{eq:approx_dissip_eqn}
\E^{\rho} \big( |F_n (X_\tau)|^2 \big) = \beta(\rho_n) +\rho
\big(|\sfL[F_n]|^2 \big) + O(\frac{1}{n}).
\end{equation}

Now set
\[
G_n= \Pi(\rho)[F_n],\quad \tilde{G}_n= G_n - F_n = \Pi(\rho)[F_n] -
\Pi(\rho_n)[F_n].  
\]
From~\Prop{prop:S_rho_Pi_rho},
\[
 \E^\rho \big( |\tilde G_n(X_\tau)|^2 \big) \leq 4 \|r_n-1\|^2_\infty \E^\rho
 \big( |F_n(X_\tau)|^2 \big) = O(\frac{1}{n^2}).
\]
Noting $G_n = F_n +\tilde G_n$, we have, 
\[
\E^\rho \big( |G_n(X_\tau)|^2 \big)  = \E^\rho \big( |F_n (X_\tau)|^2
\big)  + O(\frac{1}{n}) < \bar\beta +2 < \infty.
\]
From (part 1) of~\Prop{prop:S_rho_Pi_rho}, there exists a subsequential
limit $G\in
\clS_\tau^\rho$ such that~\eqref{eq:Lyapunov_equation} gives,  
\begin{equation}\label{eq:exact_dissip_eqn}
\E^{\rho} \big( |G(X_\tau)|^2 \big) = \rho(\clE[G]) +\rho
\big(|\sfL[G]|^2 \big).
\end{equation}
Our plan is to compare this to~\eqref{eq:approx_dissip_eqn} to show
that $\lim_{n\to\infty} \beta(\rho_n) = \rho(\clE[G])$.  

Now,
\[
\E^{\rho} \big( |G(X_\tau) - F_n(X_\tau)|^2 \big) \leq 2 \E^{\rho}
\big( |G(X_\tau) - G_n(X_\tau)|^2 \big) + 2 \E^{\rho} \big( |\tilde
G_n(X_\tau)|^2 \big)\to 0. 
\]
This has two implications as follows:
\begin{enumerate}
\item The left hand side of~\eqref{eq:approx_dissip_eqn} converges to
  the corresponding term in~\eqref{eq:exact_dissip_eqn},
\[
\lim_{n\to\infty} \E^{\rho} \big( | F_n(X_\tau)|^2 \big) = \E^{\rho} \big( |G(X_\tau)|^2 \big).
\]
\item Next, for each $x\in \supp(\rho)$,
  \[
\E^{x} \big( |G(X_\tau) - F_n(X_\tau)|^2 \big) \to 0
\]
which gives
\begin{align*}
  |\sfL[G](x) - \sfL[F_n](x)|^2 & = |\E^{x} \big( G(X_\tau) -
F_n(X_\tau) \big)  |^2 \\
& \leq \E^{x} \big( |G(X_\tau) -
F_n(X_\tau)|^2 \big) \to 0, \quad x\in \supp(\rho).
\end{align*}
Therefore,
\[
\lim_{n\to\infty} \rho
\big(|\sfL[F_n]|^2 \big) = \rho
\big(|\sfL[G]|^2 \big).
\]
\end{enumerate}
Subtracting~\eqref{eq:approx_dissip_eqn}
from~\eqref{eq:exact_dissip_eqn}, upon taking the limit, we obtain
\[
\lim_{n\to\infty} \beta(\rho_n) = \rho(\clE[G]).
\]
Finally, note
\[1=\rho_n
\big(|\sfL[F_n]|^2 \big) = \rho
\big(|\sfL[F_n]|^2 \big)  + O(\frac{1}{n})
\]
Therefore,
\[
 \rho \big(|\sfL[G]|^2 \big)= \lim_{n\to\infty} \rho
\big(|\sfL[F_n]|^2 \big) =1.
\]
That is $G\in \clS_\tau^\rho$ is in the
feasible set.  Since $\beta(\rho)$ is the infimizing constant
$\beta(\rho)\leq \rho(\clE[G])$ and this proves
the~\eqref{eq:lower-semi-cont} property.
\end{proof}

\subsection{Proof of variance decay}\label{appdx:var_contractive}
\begin{proof}[Proof of \Prop{prop:var_contractive}]
  Let $(Y,V)=\mathrm{BS\Delta E}_T[\gamma_T]$.  Set $N := \lfloor T/\tau \rfloor$,
  $t_k := k\tau$, $\clK_k= \sum_{j=0}^{k-1} \kappa(\pi_{t_j}^\nu)$ for $k=0,1,\ldots,N$.  We claim that
  \begin{align}
\pi_{t_{k}}^\nu\big( (Y_{t_{k}}-1)^2\big) &\leq e^{-\tau \kappa({\pi_{t_k}^\nu})}\E^\nu\big(\pi_{t_{k+1}}^\nu\big( (Y_{t_{k+1}}-1)^2\big)
\mid\clZ_{t_k}\big) \,
,\quad \sP^\nu\text{-a.s.},\quad
    k<N, \label{eq:claim-window}\\
    \pi_{t_{N}}^\nu\big( (Y_{t_{N}}-1)^2\big) &\leq \E^\nu\big(\pi_{T}^\nu\big( (Y_{T}-1)^2\big)
    \mid\clZ_{t_N}\big). \label{eq:claim-final}
\end{align}
Supposing the claim to be true,
\begin{align*}
	\E^\nu\big(e^{-\tau \clK_{k}}\pi_{t_{k}}^\nu\big( (Y_{t_{k}}-1)^2\big)\big)
	 &\le \E^\nu\big(e^{-\tau \clK_{k}}e^{-\tau
  \kappa({\pi_{t_k}^\nu})}\E^\nu\big(\pi_{t_{k+1}}^\nu\big(
  (Y_{t_{k+1}}-1)^2\big) \mid\clZ_{t_{k}}\big)\big)\\
	 &=\E^\nu\big(e^{-\tau \clK_{k+1}}\pi_{t_{k+1}}^\nu\big(
  (Y_{t_{k+1}}-1)^2\big) \big),\quad k<N.
\end{align*}
A recursive application of this inequality gives
	\begin{align*}
\dvar^\nu(Y_0(X_0)) \le \E^\nu\big(e^{-\tau \clK_1}\pi_{t_{1}}^\nu\big(
  (Y_{t_{1}}-1)^2\big)\big) \le \cdots \le \E^\nu\big(e^{-\tau \clK_N}\pi_{t_{N}}^\nu\big(
  (Y_{t_{N}}-1)^2\big)\big).
	\end{align*}
Using~\eqref{eq:claim-final} completes the proof:
        \begin{align*}
          \dvar^\nu(Y_0(X_0))  \leq \E^\nu\big(e^{-\tau \clK_N}\pi_{t_{N}}^\nu\big(
  (Y_{t_{N}}-1)^2\big)\big)
          \leq \E^\nu\big(e^{-\tau \clK_N}\pi_{T}^\nu\big( (Y_{T}-1)^2\big)\big).
        \end{align*}
        
It remains to prove the claim.  By
\Prop{cor:U-opt-zero_elegant} at $U = U^\opt = 0$, the following process is a
$(\clZ,\sP^\nu)$-martingale:
\[
  \Theta_t = \pi^\nu_t\big((Y_t-1)^2\big)
    - \sum_{s=0}^{t-1}\E^\nu\big(\ell(Y_{s+1},V_s,0;X_s)\mid\clZ_s\big),
  \qquad t\in\mathbb{T} .
\]
Therefore,
\begin{align*}
\E^\nu \big( \Theta_{t_{k+1}}  \mid \clZ_{t_k} \big) &=\E^\nu \big(
                                                       \pi_{t_{k+1}}^\nu\big( (Y_{t_{k+1}}-1)^2\big) \mid \clZ_{t_k} \big) -
\sum_{s=0}^{t_{k+1}-1}\E^\nu\big(\ell(Y_{s+1},V_s,0;X_s)\mid\clZ_s\big)\\
  & =\Theta_{t_{k}} = \pi_{t_k}^\nu\big((Y_{t_k}-1)^2\big) - \sum_{s=0}^{t_{k}-1}\E^\nu\big(\ell(Y_{s+1},V_s,0;X_s)\mid\clZ_s\big),\quad k<N.
\end{align*}
Because $t_{k+1} = t_k + \tau$, this gives,
\begin{align*}
\pi_{t_k}^\nu\big((Y_{t_k}-1)^2\big)+\sum_{s=t_k}^{t_k+\tau-1}
  \E^\nu\big( \ell(Y_{s+1},V_s,0;X_s) \mid \clZ_{t_k} \big) =\E^\nu \big(
\pi_{t_{k+1}}^\nu\big( (Y_{t_{k+1}}-1)^2\big) \mid \clZ_{t_k} \big),
\quad k<N.
\end{align*}
Use Defn.~\ref{defn:shifted_solution} to construct a shifted
solution over the fixed time-interval $[0,\tau]$.  We have two cases
to consider:
\begin{enumerate}
\item If $\pi_{t_k}^\nu\in \sfN_\tau$ then, upon using \Corollary{corollary:prop:Markov_BSDE}(2), $\pi_{t_k}^\nu \big(
  (Y_{t_k}-1)^2\big)=0$.  Therefore,~\eqref{eq:claim-window} holds with 
  $\kappa(\pi_{t_k}^\nu)=\infty$.
\item If $\pi_{t_k}^\nu\notin \sfN_\tau$ then
  $\beta(\pi_{t_k}^\nu) <\infty$.  
From
\Corollary{corollary:prop:Markov_BSDE} using the definition of
Poincar\'e constant,
\[
\sum_{s={t_k}}^{{t_k}+\tau-1} \E^\nu\big(\ell(Y_{s+1},V_s,0;X_s) \mid \clZ_{t_k}\big) \geq
  \beta(\pi_{t_k}^\nu) \; \pi_{t_k}^\nu \big( (Y_{t_k}-1)^2\big),\quad k<N,
\]
which, because $1 + \beta(\pi_{t_k}^\nu) = e^{\tau \kappa(\pi_{t_k}^\nu)}$,
gives~\eqref{eq:claim-window}.
\end{enumerate}
The proof of~\eqref{eq:claim-final} is similar, over the
interval $[t_N,T]$.
\end{proof}

\bibliographystyle{siam}
\bibliography{bibfiles/_master_bib_jin.bib,bibfiles/jin_papers.bib}
\end{document}

%% file: _definitions.tex
\def\Re{\mathbb{R}}

\def\Lemma#1{Lemma~\ref{#1}}

\def\Theorem#1{Thm.~\ref{#1}}
\def\Corollary#1{Cor.~\ref{#1}}
\def\Sec#1{Sec.~\ref{#1}}
\def\Fig#1{Fig.~\ref{#1}}

\def\notes#1{\marginpar{\tiny #1}\typeout{Notes!
Notes!
Notes!
}}
\renewcommand{\notes}[1]{\typeout{notes!}}
\def\FRAC#1#2#3{\genfrac{}{}{}{#1}{#2}{#3}}
\def\half{{\mathchoice{\FRAC{1}{1}{2}}%
{\FRAC{2}{1}{2}}%
{\FRAC{3}{1}{2}}%
{\FRAC{4}{1}{2}}}}

\def\Re{\field{R}}

\def\Sec#1{Sec.~\ref{#1}}

\def\transpose{{\hbox{\rm\tiny T}}}

\def\clE{{\cal E}}

\def\clP{{\cal P}}
\def\clZ{{\cal Z}}

\def\Sec#1{Sec~\ref{#1}}

\def\E{{\sf E}}

\def\clC{{\cal C}}

\def\Fig#1{Fig.~\ref{#1}}
\def\Sec#1{Sec.~\ref{#1}}

\def\clZ{{\cal Z}}

\newtheorem{example}{Example}
\newtheorem{remark}{Remark}
\newtheorem{assumption}{Assumption}

\def\beq{\begin{eqnarray}} 
\def\bc{\begin{center}} 
\def\be{\begin{enumerate}}
\def\bi{\begin{itemize}} 
\def\bs{\begin{small}}
\def\bS{\begin{slide}}
\def\ec{\end{center}} 
\def\ee{\end{enumerate}}
\def\ei{\end{itemize}}
\def\es{\end{small}}
\def\eS{\end{slide}}
\def\eeq{\end{eqnarray}}

\newcommand{\newP}[1]{\medskip\noindent{\bf #1:}}

\newcommand{\ud}{\,\mathrm{d}}

\def\Re{\mathbb{R}}
\def\E{{\sf E}}

\def\clY{{\cal Y}}

\def\Sec#1{Sec.~\ref{#1}}
\def\Thm#1{Thm.~\ref{#1}}
\def\Prop#1{Prop.~\ref{#1}}
\def\Defn#1{Defn.~\ref{#1}}

\def\clD{{\cal D}}
\def\clE{{\cal E}}

\def\clP{{\cal P}}
\def\clZ{{\cal Z}}

\renewcommand{\Re}{\mathbb{R}}

\def\FRAC#1#2#3{\genfrac{}{}{}{#1}{#2}{#3}}

\newcommand{\var}{\text{var}}

\newcommand{\sfJ}{{\sf J}}

%% file: jin_thesis_macros.tex
\def\clC{{\cal C}}

\def\clD{{\cal D}}
\def\clE{{\cal E}}
\def\clF{{\cal F}}
\def\clG{{\cal G}}

\def\clK{{\cal K}}

\def\clN{{\cal N}}

\def\clP{{\mathcal P}}

\def\clR{{\cal R}}
\def\clS{{\mathcal H}}

\def\clU{{\cal U}}
\def\clV{{\cal V}}

\def\clY{{\cal Y}}
\def\clZ{{\cal Z}}
\def\E{{\mathsf E}}
\def\bS{\mathbb{S}}
\def\bO{\mathbb{O}}
\def\sJ{{\mathsf J}}

\def\ones{{\mathsf 1}}
\def\sP{{\mathsf P}}

\def\shift{\theta}

\def\tp{{\hbox{\rm\tiny T}}}

\def\dvar{\operatorname{var}}

\def\chisq{{\chi^2}}

\def\tv{{\text{\tiny TV}}}

\def\opt{{\text{\rm (opt)}}}

\renewcommand{\limsup}{\mathop{\operatorname{limsup}}}
\renewcommand{\liminf}{\mathop{\operatorname{liminf}}}

\newcommand{\clZf}[1]{\mathfrak{Z}_{#1}}

\newcommand{\clZfo}[1]{\mathfrak{Z}^\circ_{#1}}

\def\supp{\text{supp}}
\def\sfN{{\sf N}}
\def\sfL{{\sf L}}

%% file: figure/hmm.tex








\begin{tikzpicture}[
    node distance=1.2cm,
    latent/.style={draw, circle, minimum size=1cm, inner sep=0pt, fill=gray!20},
    obs/.style={draw, circle, minimum size=1cm, inner sep=0pt},
    arrow/.style={-stealth, thick},
    dots/.style={inner sep=3pt, minimum size=0.2cm}
]

    \node (start_x) {};

    \node (x0) at ($(start_x) + (0.5cm, 0)$) [latent] {$X_0$};
    \node (x1) [latent, right=of x0] {$X_1$};
    \node (dots) [dots, right=of x1] {$\cdots$};
    \node (xT_1) [latent, right=of dots] {$X_{T-1}$};
    \node (xT) [latent, right=of xT_1] {$X_T$};

    \node (z1) [obs, above=0.1cm of x1] {$Z_1$};
    \node (z2) [obs, right=of z1] {$Z_2$};
    \node (zd) [right=of z2] {$\cdots$};
    \node (zT) [obs, right=of zd] {$Z_T$};

    \draw [arrow] (x0) -- (x1);
    \draw [arrow] (x1) -- (dots);
    \draw [arrow] (dots) -- (xT_1);
    \draw [arrow] (xT_1) -- (xT);

    \draw [arrow] (x0) -- (z1);
    \draw [arrow] (x1) -- (z2);
    \draw [arrow] (xT_1) -- (zT);

    \draw [arrow] ($(x0) + (-0.5cm, -1.0cm)$) -- ($(xT) + (0.5cm, -1.0cm)$) node[right] {time};

    \draw ($(x0) + (0, -1.0cm)$) -- ($(x0) + (0, -1.2cm)$) node[below] {$0$};
    \draw ($(x1) + (0, -1.0cm)$) -- ($(x1) + (0, -1.2cm)$) node[below] {$1$};
    \draw ($(xT_1) + (0, -1.0cm)$) -- ($(xT_1) + (0, -1.2cm)$) node[below] {$T-1$};
    \draw ($(xT) + (0, -1.0cm)$) -- ($(xT) + (0, -1.2cm)$) node[below] {$T$};

\end{tikzpicture}

%% file: figure/variance_decay.tex
\begin{tikzpicture}[x=0.85cm, y=1.35cm, >=stealth,
    dot/.style={circle, fill=black, inner sep=1.1pt}]

\draw[dashed, gray] (-0.1,0.30) -- (7.6,0.30);
\draw[dashed, gray] (-0.1,1.10) -- (7.6,1.10);
\draw[dashed, gray] (-0.1,2.20) -- (7.6,2.20);
\node[left, inner sep=3pt] at (-0.1,0.30) {$\var^\rho\big(y_0(X_0)\big)$};
\node[left, inner sep=3pt] at (-0.1,1.10) {$\var^\rho\big(Y_t(X_t)\big)$};
\node[left, inner sep=3pt] at (-0.1,2.20) {$\var^\rho\big(\gamma_T(X_T)\big)$};

\draw[->] (-0.1,0) -- (7.9,0) node[below right, inner sep=2pt] {$t$};
\draw[->] (-0.1,0) -- (-0.1,2.50);

\foreach \t/\v in {0/0.30, 1/0.39, 2/0.50, 4/1.10, 6/1.72, 7/2.20} {
  \draw (\t,0) -- (\t,\v);
  \node[dot] at (\t,\v) {};
}
\node at (3,0.78) {$\cdots$};
\node at (5,1.40) {$\cdots$};

\foreach \t/\lab in {0/{$0$}, 1/{$1$}, 2/{$2$}, 3/{$\cdots$}, 4/{$t$},
                     5/{$\cdots$}, 6/{$T\!-\!1$}, 7/{$T$}}
  \node[below, inner sep=3pt] at (\t,0) {\lab};

\draw[<->] (7.6,0.30) -- (7.6,1.10)
  node[midway, right, inner sep=3pt] {$\E^\rho\big(\langle M\rangle_t\big)$};

\end{tikzpicture}